\documentclass{amsart}
\usepackage{CGW}

\makeatletter
\let\oldl@part\l@part
\renewcommand{\l@part}[2]{\oldl@part{#1}{}}
\makeatother

\usepackage[utf8]{inputenc}
\usepackage{amsfonts,amsthm,amsmath,amssymb,stmaryrd}
\usepackage{mathtools}
\usepackage{mathdots}
\usepackage{tikz}
\usepackage{tikz-cd}
\usepackage{graphicx}
\usepackage{color}
\usepackage{changepage}
\usepackage[margin=1.4in]{geometry}
\usepackage[capitalise]{cleveref} 

\newcommand{\A}{\mathcal{A}}
\newcommand{\B}{\mathcal{B}}
\newcommand{\C}{\mathcal{C}}
\newcommand{\D}{\mathcal{D}}

\newcommand{\E}{\mathcal{E}}
\newcommand{\M}{\mathcal{M}}
\newcommand{\V}{\mathcal{V}}
\newcommand{\W}{\mathcal{W}}
\newcommand{\X}{\mathcal{X}}
\newcommand{\Y}{\mathcal{Y}}

\newcommand{\im}{\mathrm{im}}

\newcommand{\Set}{\mathsf{Set}}

\newcommand{\finset}{\mathsf{FinSet}}

\newcommand{\Top}{\mathsf{Top}}
\newcommand{\cG}{\mathcal{G}}
\newcommand{\fG}{\mathfrak{G}}
\newcommand{\cGh}{\hat{\mathcal{G}}}
\newcommand{\fGh}{\hat{\mathfrak{G}}}
\DeclareMathOperator{\sspan}{span}

\newcommand{\btk}{\begin{tikzcd}}
\newcommand{\etk}{\end{tikzcd}}

\DeclareMathOperator{\cok}{coker}
\DeclareMathOperator{\id}{id}
\DeclareMathOperator{\ob}{ob}
\DeclareMathOperator{\op}{op}
\DeclareMathOperator{\g}{g}
\DeclareMathOperator{\cube}{cube}
\DeclareMathOperator{\gcube}{\g-\cube}
\DeclareMathOperator{\n}{n}
\DeclareMathOperator{\Hom}{Hom}
\DeclareMathOperator{\Ob}{Ob}

\newcommand{\fcgwa}{\text{ECGW}}

\theoremstyle{plain}
\newtheorem{thm}{Theorem}[section]
\newtheorem*{theorem*}{Theorem}
\newtheorem{prop}[thm]{Proposition}
\newtheorem{lemma}[thm]{Lemma}
\newtheorem{cor}[thm]{Corollary}

\theoremstyle{definition}
\newtheorem{defn}[thm]{Definition}
\newtheorem{notation}[thm]{Notation}

\theoremstyle{remark}
\newtheorem{ex}[thm]{Example}
\newtheorem{rmk}[thm]{Remark}

\crefname{lem}{Lemma}{Lemmas}
\crefname{claim}{Claim}{Claims}
\crefname{thm}{Theorem}{Theorems}
\crefname{defn}{Definition}{Definitions}
\crefname{prop}{Proposition}{Propositions}
\crefname{rmk}{Remark}{Remarks}
\crefname{cor}{Corollary}{Corollaries}
\crefname{ex}{Example}{Examples}
\crefname{notation}{Notation}{Notations}
\crefname{convention}{Convention}{Conventions}

\makeatletter
\pgfqkeys{/tikz/commutative diagrams}{
  row sep/.code={\tikzcd@sep{row}{#1}{}},
  column sep/.code={\tikzcd@sep{column}{#1}{}},
  bo row sep/.code={\tikzcd@sep{row}{#1}{between origins}},
  bo column sep/.code={\tikzcd@sep{column}{#1}{between origins}},
  bo column sep/.default=normal,
  bo row sep/.default=normal,
}
\def\tikzcd@sep#1#2#3{
  \pgfkeysifdefined{/tikz/commutative diagrams/#1 sep/#2}%
    {\pgfkeysalso{/tikz/#1 sep={\ifx\\#3\\1*\else1.7*\fi\pgfkeysvalueof{/tikz/commutative diagrams/#1 sep/#2},#3}}}%
    {\pgfkeysalso{/tikz/#1 sep={#2,#3}}}}
\makeatother

\title{Additivity and Fiber Sequences for Combinatorial $K$-Theory}

\author[M.\ Sarazola]{Maru Sarazola}
\address{School of Mathematics, University of Minnesota, Minneapolis MN, 55415, USA}
\email{maru@umn.edu}

\author[B.\ T.\ Shapiro]{Brandon T.\ Shapiro}
\address{Kerchof Hall, 141 Cabell Dr, Charlottesville, VA 22903}
\email{brandonshapiro@virginia.edu}

\begin{document}

\begin{abstract}
     The (A)CGW categories of Campbell and Zakharevich show how finite sets and varieties behave like the objects of an exact category for the purpose of algebraic $K$-theory. These structures admit a well-behaved Q-construction akin to Quillen's, and satisfy analogues of the D\'evissage and Localization theorems. In this work, we modify Campbell and Zakharevich's axioms to obtain a framework called \emph{$\fcgwa$ categories} that allows for an $S_\bullet$-construction akin to Waldhausen's, and show how it produces a $K$-theory spectrum which satisfies an analogue of the Additivity Theorem. We also define a notion of ``relative $\fcgwa$ categories'' which have weak equivalences determined by a subcategory of acyclic objects satisfying minimal conditions; these satisfy analogues of the Fibration and Localization Theorems that generalize previous versions in the literature. We illustrate these results with examples including exact categories, extensive categories, algebraic varieties, and polytopes up to scissors congruence.
\end{abstract}

\maketitle

\setcounter{tocdepth}{1}
\tableofcontents

\section{Introduction}

In recent work \cite{CZ}, Campbell and Zakharevich introduced \emph{CGW categories}, which generalize exact categories to include finite sets and algebraic varieties by requiring only the properties that make exact categories so particularly well-suited for algebraic $K$-theory. Their key insight lies in the fact that a $K$-theory machinery for exact categories only truly sees the admissible monomorphisms and epimorphisms, along with the bicartesian squares between them. In particular, it is not necessary to be able to compose monomorphisms and epimorphisms with one another. This suggests that these two types of morphisms should be treated as if they belong to separate categories connected by a rule for commuting them past one another. This type of structure is called a \emph{double category}, and can be used to axiomatize all of the properties necessary to obtain a $K$-theory machinery analogous to the $Q$-construction.

The main appeal of these double categories is that they generalize the structure of exact sequences to key non-additive settings such as finite sets and varieties, where the notion of complements replaces that of kernels and cokernels. This makes it possible to study finite sets and varieties as if they were the objects of an exact category, for the purposes of $K$-theory. Aside from setting the stage for further study of derived motivic measures and of the $K$-theory of varieties, this new framework has already been used by Haesemeyer and Weibel \cite{Hchuck} and by Coley and Weibel \cite{ianchuck} to develop the $K$-theory of partially cancellative $A$-sets. 

In order to recover analogues of Quillen's classical results \cite{Qui73} such as the Localization and Devissage theorems, one needs to pass from CGW categories to a structure with additional information: ACGW categories. This is reminiscent of the restriction from exact to abelian categories---indeed, the ``A'' stands for ``Abelian''---and just like in the classical case, exact categories are no longer an example. Varieties also fail to form an ACGW category, though unlike exact categories this is not a notable loss, as the equivalent $K$-theory of reduced schemes is modeled by an ACGW category.

ACGW categories also admit an $S_\bullet$-con\-struc\-tion in the flavor of Waldhausen's \cite{Waldhausen}. However, the $S_\bullet$-construction of an ACGW category cannot be iterated more than once, as the double category $S_n\A$ is  not ACGW. As a result, this $S_\bullet$-construction yields a space but not necessarily a spectrum. In addition, the Additivity Theorem is only proven for the particular example of ACGW categories arising from subtractive categories \cite{Campbell}. 

Given the increasing interest in combinatorial $K$-theory, and in particular in squares-based $K$-theoretical frameworks, in this article we study what modifications need to be made to the (A)CGW framework in order to fix these shortcomings. In doing so, we find that the obstructions for ACGW categories to encompass all of the motivating examples, admit a proof of the Additivity theorem, and allow for the $S_\bullet$-construction to be iterated, are all essentially the same. Campbell and Zakharevich use pullback diagrams as the morphisms between arrows and demand that kernels, cokernels, and restricted pushouts apply to all pullbacks. We argue that this role should be played by a class of ``good squares'' which are potentially more restrictive than pullbacks. The intuition behind these stems from Waldhausen's original construction, which tells us that in order for a square 
\begin{diagram}
{A & B\\ C & D\\};
\mto{1-1}{1-2}\mto{1-1}{2-1} \mto{2-1}{2-2} \mto{1-2}{2-2}
\end{diagram}
to admit a cofiber (when considered as a morphism in the category whose objects are cofibrations), it must satisfy an additional property: the induced map $B\cup_A C\to D$ must be a cofibration as well. Our ``good squares'' are defined in close analogy with Waldhausen's, and we discuss the comparison in \cref{goodsqpushout}.

In the process, we expand on the work of \cite{CZ} to allow for the addition of homotopical information; in other words, we define a notion of ``ACGW categories with weak equivalences'', along with the corresponding $S_\bullet$-construction. To introduce the (horizontal and vertical) weak equivalences, we borrow intuition from an archetypal algebraic category with weak equivalences: chain complexes with quasi-isomorphisms. A quick study finds that among monomorphisms (resp.\ epimorphisms), the quasi-isomorphisms are those whose cokernel (resp.\ kernel) is an exact chain complex. Mirroring this scenario, the weak equivalences in our setting are determined by a choice of ``acyclic objects'' satisfying certain closure properties, much like the weak equivalences considered in \cite{cotorsion}. Our main structure, which we call a \emph{relative $\fcgwa$ category}, then consists of a pair $(\C,\W)$, where $\C$ is a suitable double category and $\W$ is the full double subcategory of acyclic objects (\cref{fcgwa}).

Restricting the weak equivalences to isomorphisms, our more robust axioms allow us to recover the examples of exact categories and varieties. These are not ACGW categories, but they do form $\fcgwa$ categories, so they can now be studied using the full force of our foundational theorems of $K$-theory. We also show that the double category $\C^\D$ of $\D$-shaped diagrams in an $\fcgwa$ category $\C$, for any double category $\D$, is $\fcgwa$. We interpret this as an ``exponentiability'' property, which inspires the ``E'' in $\fcgwa$. Notably, this allows us to show that each $S_n \C$ is $\fcgwa$ which implies that this $S_\bullet$-construction can be iterated, and thus $K(\C)$ is an $\Omega$-spectrum, as shown in \cref{spectra}.

\begin{theorem*}[Delooping]
Let $(\C,\W)$ be a relative $\fcgwa$ category. Then, iterating the $S_\bullet$-construction exhibits  $K(\C,\W)=\Omega|wS_\bullet\C|$ as an $\Omega$-spectrum.
\end{theorem*}

Just as \cite{CZ} captures the essential features required to carry out Quillen's major foundational theorems, our $\fcgwa$ categories allow us to obtain many of Waldhausen's structural results. Chief among them are the Additivity Theorem (\cref{additivity})---which 
 as the modern perspective shows us \cite{blumberggepnertabuada, barwick}, characterizes algebraic $K$-theory---and the Fibration Theorem (\cref{fibrationthm}), which compares the $K$-theory of a category equipped with two classes of weak equivalences by constructing a homotopy fiber.  

 \begin{theorem*}[Additivity]
Let $\A,\B\subseteq\C$ be full $\fcgwa$ subcategories of a relative $\fcgwa$ category $(\C,\W)$. Then, the map
$$wS_\bullet E(\A,\C,\B)\to wS_\bullet \A\times wS_\bullet\B$$ induced by $$(A\mrto C\elto B)\mapsto (A,B)$$ is a homotopy equivalence.
\end{theorem*}

\begin{theorem*}[Fibration]
Let $(\C,\V)$ and $(\C,\W)$ be two $\fcgwa$ category structures on $\C$ such that $\V\subseteq\W$. Then there exists a homotopy fiber sequence
$$K(\W,\V) \to K(\C,\V) \to K(\C,\W)$$
\end{theorem*}

Unlike the Fibration Theorem for Waldhausen categories, our theorem requires neither the existence of a cylinder functor nor any special factorizations, instead relying on the symmetry between the two types of morphisms in an $\fcgwa$ category, analogous to the dual properties of mono- and epimorphisms in an exact category. Of course, not every Waldhausen category gives rise to a relative $\fcgwa$ category, and for this reason our Fibration Theorem should not be interpreted as a generalization of the classical result; rather, as an exploration of this result in a different direction.

As a consequence of our result in the case where $\V$ is trivial, we obtain a Localization Theorem (\cref{localization}) that generalizes many of those existing in the literature; this includes Quillen's original theorem for abelian categories \cite{Qui73}, Schlichting's \cite{Sch} and Cardenas' \cite{Cardenas} Localization Theorems for exact categories, the first author's Localization Theorem obtained from cotorsion pairs \cite{cotorsion}, and the Localization Theorem for ACGW categories of \cite{CZ}. In the setting of $\fcgwa$ categories arising from exact categories, it reads as follows:

\begin{theorem*}[Localization]
Let $\B$ be an exact category and $\A\subseteq\B$ a full subcategory such that if any two terms in an exact sequence in $\B$ are in $\A$, then the third term is as well. 
Then there exists an $\fcgwa$ category $(\B,\A)$ such that
$$K(\A)\to K(\B) \to K(\B,\A)$$
is a homotopy fiber sequence.
\end{theorem*}

This version of the Localization Theorem has fewer requirements than any of those mentioned above, which with the exception of \cite{cotorsion} all require the subcategory $\A$ to be closed under all subobjects and quotients. 

Throughout the paper we highlight the applicability of these results to new combinatorial examples. In \cref{subsec:extensive} we discuss extensive categories which generalize from finite sets to categories with suitably well behaved coproducts and pullbacks to form ACGW/ECGW categories. In particular we highlight the $K$-theories of free and arbitrary finite $H$-sets for $H$ a finite group, and of various types of polytopes in dimension $n$ which were previously studied using assemblers \cite{assemblers}.

\subsection*{Acknowledgements}

The authors would like to thank Inna Zakharevich for her consistent support and sharing boundless intuition for the ideas behind CGW categories. We are also grateful to Emily Riehl for many suggestions that improved the presentation of this paper, to David Gepner and Julia Bergner for helpful conversations, and to an anonymous referee for a detailed reading and many suggestions that helped improved the readability of this work.

During parts of the production of this work, the first author was supported by Cornell University's Hutchinson Fellowship, and the second author was supported by the NDSEG fellowship.



\section{Categorical preliminaries}

\subsection{Extensive categories}\label{subsec:extensive}

Throughout this paper, results motivated by the category of finite sets will apply equally well to any category with suitably well-behaved coproducts. What it means for coproducts to be suitable for this purpose is captured by the notion of an \emph{extensive category}, which will also illustrate in a more straightforward setting many of the principles of our more general framework for combinatorial $K$-theory.

\begin{defn}[{\cite{extensive}}]\label{defnextensive}
A category $\X$ is \textbf{finitely extensive} (henceforth simply \textbf{extensive}) if it has finite coproducts such that
\begin{itemize}
    \item Coproduct injections are monic and the cospan $A \hookrightarrow A \sqcup B \hookleftarrow B$ has a pullback given by the initial object $\varnothing$
    \item For any morphism $X \to A \sqcup B$, we have pullback squares as below with $X \cong Y \sqcup Z$
    \begin{diagram}
    {Y & X & Z \\ A & A \sqcup B & B \\};
    \cofib{1-1}{1-2} \cofibl{1-3}{1-2}
    \cofib{2-1}{2-2} \cofibl{2-3}{2-2}
    \to{1-1}{2-1} \to{1-2}{2-2} \to{1-3}{2-3}
    \end{diagram}
\end{itemize}
\end{defn}

Basic examples of extensive categories include finite sets, finite $M$-sets for a monoid $M$, functors from any category to $\Set$, finite and small categories, and topological spaces, united by the intuition that coproducts behave like disjoint unions. 

We now include several results and examples that will be used later in the paper.

\begin{lemma}
    Given morphisms $A \hookrightarrow X$, $X \hookleftarrow B_1$, and $X \hookleftarrow B_2$ in an extensive category which exhibit $X$ as the coproducts $X \cong A \sqcup B_1 \cong A \sqcup B_2$, there is a unique isomorphism $B_1 \cong B_2$ which commutes over $X$.
\end{lemma}

\begin{proof}
    This is an immediate consequence of \cref{defnextensive}, based on the following diagram of pullback squares.
    \begin{diagram}
    {\varnothing & B_2 & B_2 \\ A & X & B_1 \\ A & A & \varnothing \\};
    \cofib{1-1}{1-2} \eq{1-3}{1-2}
    \cofib{2-1}{2-2} \cofibl{2-3}{2-2}
    \eq{3-1}{3-2} \cofibl{3-3}{3-2}
    \cofib{1-1}{2-1} \cofib{1-2}{2-2} \cofib{1-3}{2-3}
    \eq{3-1}{2-1} \cofibl{3-2}{2-2} \cofibl{3-3}{2-3}
    \end{diagram}
    Every row and column is a complementary pair of coproduct injections, so the map $B_2 \to B_1$ must be an isomorphism.
\end{proof}

We can therefore refer to the \emph{complement} of a coproduct injection $A \hookrightarrow X$ which is uniquely defined up to isomorphism.

\begin{lemma}\label{extensiveshavePO}
    Extensive categories have pushouts along coproduct injections.
\end{lemma}

\begin{proof}
    Given a morphism $f \colon A \to X$ and coproduct injection $A \hookrightarrow A \sqcup B$, we show that $X \sqcup B$ is their pushout. A pair of morphisms $X \to Y$ and $A \sqcup B \to Y$ which commutes under $A$ is the same data as a morphism $g \colon X \to Y$, a morphism $h \colon A \to Y$ which factors as $g \circ f$, and a morphism $B \to Y$. The map $h$ is therefore uniquely determined and so this data amounts to a pair of maps from $X$ and $B$ to $Y$, or equivalently a single map from $X \sqcup B$.
\end{proof}

\begin{cor}\label{intersectionunion}
    Any pair of coproduct injections $A \hookrightarrow X \hookleftarrow B$ has an ``intersection'' $A \times_X B$ and a ``union'' $A \cup_X B = A \sqcup_{A \times_X B} B$, both of which admit coproduct injections into $X$.
\end{cor}

\begin{proof}
    Let $X \cong A \sqcup X' \cong B \sqcup X''$ for some objects $X',X''$. Taking pullbacks we get that
    \[
    A \cong (A \times_X B) \sqcup (A \times_X X''), \qquad 
    B \cong (A \times_X B) \sqcup (X' \times_X B),\]
    and  
    \[X \cong (A \times_X B) \sqcup (A \times_X X'') \sqcup (X' \times_X B) \sqcup (X' \times_X X''),
    \]
    which includes both $A \times_X B$ and $A \cup_X B \cong (A \times_X B) \sqcup (A \times_X X'') \sqcup (X' \times_X B)$ as coproduct components. 
\end{proof}

\begin{defn}
    An \textbf{extensive subcategory} $\Y$ of an extensive category $\X$ is a subcategory containing the initial object and closed under coproducts and pullbacks along coproduct injections.

    For $\Y$ a full subcategory of an extensive category $\X$, we can also define the \emph{complementary} full subcategory $\X - \Y$ spanned by objects in $\X$ which are not isomorphic to any coproduct $A \sqcup B$, where $A$ ranges among all non-initial objects in $\Y$.
\end{defn}

\begin{ex}[Finite $G$-Sets]\label{gsets}
    For a group $G$, the category of finite $G$-sets is extensive, and when $G$ is finite there is a full extensive subcategory of free finite $G$-sets given by finite disjoint unions of the free transitive $G$-set $G$. In general, the category of finite $G$-sets is freely generated under coproducts by the finite transitive $G$-sets $G/H$ for $H$ a finite-index subgroup of $G$. Any subset of those subgroups generates under finite coproducts a full extensive subcategory $\Y$ of $\X$, where $\X - \Y$ is the full extensive subcategory generated by the complementary set of finite-index subgroups. 
\end{ex}

Note that the subcategory $\X - \Y$ is not generally extensive, even if $\Y$ is. 

\begin{defn}
    For $\X$ an extensive category, a nonempty subcategory $\Y$ is \textbf{Serre} if it is full, replete, and $A \sqcup B$ is in $\Y$ if and only if both $A$ and $B$ are in $\Y$. 
\end{defn}

Note that the Serre condition is strictly stronger than the condition for an extensive subcategory: a Serre subcategory is closed under coproducts by definition, the initial object is included as $A \cong \varnothing \sqcup A$ for any $A$ in $\Y$, and while the pullback of a coproduct injection $A\sqcup B \hookrightarrow A$ along $X \to A \sqcup B$ is (up to isomorphism) a coproduct injection into $X$.

\begin{lemma}
    For $\Y$ a Serre subcategory of an extensive category $\X$, the subcategory $\X - \Y$ is also Serre.
\end{lemma}

\begin{proof}
    It is immediate that if $X \sqcup X'$ is in $\X - \Y$, then so are $X,X'$: if $X \cong A \sqcup X''$ with $A$ in $\Y$ then $X \sqcup X' \cong A \sqcup (X'' \sqcup X')$, a contradiction for $X \sqcup X'$ in $\X - \Y$.

    If $X,X'$ are in $\X - \Y$, we can also show by contradiction that $X \sqcup X'$ is in $\X - Y$. Assume that $X \sqcup X' \cong A \sqcup B$ for $A$ in $\Y$. Then by extensivity we have $X \cong (X \times_{X \sqcup X'} A) \sqcup (X \times_{X \sqcup X'} B)$ and $A \cong (X \times_{X \sqcup X'} A) \sqcup (X' \times_{X \sqcup X'} A)$. By the Serre condition for $\Y$, $X \times_{X \sqcup X'} A$ is therefore in $\Y$, a contradiction as $X$ is in $\X - \Y$.
\end{proof}
 
Given an object $X$ in $\X$ and $\Y$ a Serre subcategory, the category whose objects are coproduct injections $A \hookrightarrow X$ with $A$ in $\Y$, and whose morphisms are commuting coproduct injections between them, is filtered: given two such objects $A,B$, both $A \times_X B$ and $A \cup_X B$ admit coproduct injections into $X$ (\cref{intersectionunion}), and as coproduct injections are monic any two parallel morphisms in this category agree. 

\begin{defn}
    An object $X$ in an extensive category $\X$ is \textbf{finitary} if every filtered subcategory of coproduct injections into $X$ has a terminal object.
\end{defn}

In all of our examples, finitary objects will be finite coproducts of objects which are indecomposable under coproduct.

\begin{prop}\label{extensiveprojection}
    If every object in an extensive category $\X$ is finitary and $\Y$ is a Serre subcategory, there is a functor $\pi_\Y$ from the subcategory of coproduct injections in $\X$ to that of $\Y$ sending an object $X$ to the maximal object in $\Y$ with a coproduct injection to $X$, which preserves coproducts and pullbacks of injections from $\X$. There is also a functor $\pi_{-\Y}$ from coproduct injections in $\X$ to coproduct injections in $\X - \Y$ sending $X$ to the complement of $\pi_Y X \hookrightarrow X$, with the same properties.
\end{prop}

We would like to be able to call these ``extensive functors'', but the category of coproduct injections in $\X$ is not generally extensive, lacking for instance the folding maps $X \sqcup X \to X$ in a category with coproducts. In \cref{subsec:fcgwa} we introduce a categorical structure which only sees the coproduct injections, making the structure of these functors easier to express.

\begin{proof}
    The definition of finitary objects ensures that $\pi_\Y X$ exists, as the terminal object of the filtered subcategory of coproduct injections into $X$ with domain in $\Y$. Given an injection $X \to X \sqcup X'$, $\pi_\Y X$ is in $\Y$ with an injection into $X'$, and hence into $\pi_\Y X'$.
    
    It is straightforward to check using the Serre condition that $\pi_\Y(X \sqcup X') \cong \pi_\Y X \sqcup \pi_Y X'$ and $\pi_\Y$ preserves pullbacks of injections from $\X$. The analogous results for $\pi_{-\Y}$ follow from the observation that $\pi_{-\Y}$ agrees with the functor $\pi_{\X-\Y}$ corresponding to the complementary Serre subcategory.
\end{proof}

\begin{ex}\label{gsetprojection}
    In the extensive category of finite $G$-sets, any subset of the finite-index subgroups $H$ of $G$ induces a Serre subcategory of $\finset_G$ generated under finite coproducts by the transitive $G$-sets $G/H$ where $H$ is in the given subset. The functor $\pi_\Y$ sends a finite $G$-set to its subset of elements whose orbits have stabilizer $H$ in the given subset, while $\pi_{-\Y}$ is the complement of that subset whose orbits have the remaining finite index stabilizers. Note that this assignment is not functorial on maps such as $G/K \to G/H$ where $K < H$ and the generating subset of subgroups contains $K$ but not $H$: $\pi_\Y G/K = G/K$ and $\pi_\Y G/H$ is empty.
\end{ex}

\subsubsection{Categories of Polytopes}\label{polytopes}

We now discuss a new family of examples involving different categories of polytopes.

In \cite[Section 5.2]{assemblers}, Zakharevich defines two  types of categories of polytopes:
\begin{itemize}
	\item The category $\cG_n$ whose objects are ``closed homogeneous $n$-polytopes'', i.e.\ finite unions of closed $n$-simplices in $\mathbb{R}^\infty$. Morphisms $P \to Q$ are given by isometries $\sspan(P) \cong \sspan(Q)$ mapping $P$ to a subset of $Q$, where $\sspan(P)$ is the smallest affine subspace of $\mathbb{R}^\infty$ containing $P$. 
	\item The category $\fG_n$ whose objects are ``open heterogeneous $n$-polytopes'', i.e.\ finite unions of open $m$-simplices for any $m \le n$ (where an open 0-simplex is a point). Morphisms are defined analogously to $\cG_n$.
\end{itemize}

Note that in $\fG_n$, polytopes may be open or closed as they can include an open $n$-simplex along with all of the open $m$ simplices on its boundary ($m=0,1,...,n-1$), and there can be polytopes of dimension lower than $n$. In contrast, all polytopes in $\cG_n$ are closed and $n$-dimensional.

The category $\fG_n$ is closed under pullbacks as a subcategory of $\Top$, and these are precisely the intersections of polytopes. On the other hand, in $\cG_n$ the intersection of closed homogeneous polytopes need not be closed homogeneous, so the pullback is instead defined as the closure of the interior of the intersection within its span. As an example, while in $\fG_n$ the inclusions of two copies of the closed $n$-simplex into the union of two $n$-simplices along a common $(n-1)$-simplex face have a pullback given by the closed $(n-1)$-simplex, in $\cG_n$ the pullback is the empty polytope as the interior of the $(n-1)$-simplex in $n$-dimensional space is empty.

There is a functor $\cG_n \to \fG_n$ sending a closed homogeneous polytope to itself regarded as an open heterogenous polytope, but by the example above it does not preserve pullbacks. An alternative functor $\mathsf{int} \colon \cG_n \to \fG_n$ which does preserve pullbacks sends a closed homogeneous polytope to its interior within its span (in other words, ``forgetting the boundary''). The image of the pullback square described in the example above then consists of the disjoint inclusions of two simplices into the interior of the union of two simplices along a common face, whose pullback is empty as desired. 

However, neither $\cG_n$ not $\fG_n$ has coproducts, as two non-isometric disjoint unions of the same two polytopes in $\mathbb{R}^\infty$ are not isomorphic. In \cite{assemblers} this issue is avoided by considering a category of formal sums $(A_1,...,A_k)$ of polytopes, but this is not conducive to defining (co)kernels which are unique up to isomorphism, as both maps $(A,B) \to A \sqcup B$ and $A \sqcup B = A \sqcup B$ have the same empty complement. To remedy this we must consider a category of \emph{piecewise} isometric inclusions of polytopes, in which all disjoint unions of two polytopes are isomorphic. As Zakharevich pointed out to us, this category can also be thought of in terms of formal sums where the maps $(A,B) \to A \sqcup B$ are formally inverted.

\begin{defn}\label{piecewisemaps}
    A \emph{piecewise map} between polytopes $P \to Q$ consists of equivalence classes of the following data:
    \begin{itemize}
        \item a finite set of polytopes $A_1,...,A_k$,
        \item for all $i=1,...,k$, an isometric inclusion $\sspan(A_i) \hookrightarrow \sspan(P)$ including $A_i$ into $P$ such that each element of $P$ is in the image of exactly one element of exactly one polytope $A_i$,
        \item for all $i=1,...,k$, an isometric inclusion $\sspan(A_i) \hookrightarrow \sspan(Q)$ including $A_i$ into $Q$.
    \end{itemize}
    Here the equivalence relation is generated by:
    \begin{itemize}
        \item permutations of the set $A_1,...,A_k$,
        \item isomorphisms of polytopes $A_i \cong B_i$ in the sense discussed above, commuting over $P,Q$, for all $i=1,...,k$,
        \item replacing a pair of polytopes $A_i,A_j$ which are disjoint in $\mathbb{R}^\infty$ with their union $A_i\cup A_j$, so that the inclusions into $P,Q$ extend to an isometric inclusion of $A_i \cup A_j$.
    \end{itemize}
    We denote a piecewise map as a span $P \xleftarrow{\sim} \{A_1,...,A_k\} \to Q$.

    Piecewise maps form the morphisms of a category $\cGh_n$ of closed homogeneous $n$-polytopes and a category $\fGh_n$ of open heterogeneous $n$-polytopes, where the identity on $P$ is the piecewise map given by the span $P=P=P$. The composition of piecewise maps
    \[
    P \xleftarrow{\sim} \{A_1,...,A_k\} \to Q \xleftarrow{\sim} \{B_1,...,B_\ell\} \to R 
    \]
    is given by
    \[
    P \xleftarrow{\sim} \{A_i \times_Q B_j\}_{i,j} \to R.
    \]
\end{defn}

Note that unlike $\cG_n$ and $\fG_n$, these categories allow for maps which are not monic, as while each map $A_i \to Q$ must be an inclusion, they may overlap for different polytopes $A_i$.

The isomorphisms in these categories of piecewise maps are the \emph{scissors congruences} 
\[
P \xleftarrow{\sim} \{A_1,...,A_k\} \xrightarrow{\sim} Q,
\]
which partition two polytopes $P,Q$ into finitely many isometric pieces. This includes isomorphisms between any two disjoint unions in $\mathbb{R}^\infty$ of the same two polytopes, which are now well-defined coproducts. The coproduct injections are then those piecewise maps 
$$P \xleftarrow{\sim} \{A_1,...,A_k\} \hookrightarrow Q$$ 
where the forward map is ``injective'' in the sense that each element of $Q$ is in the image of at most one point in at most one $A_i$. For such an injection, its complement is any injection whose image is the complement of those points in $Q$, and any two such maps will be isomorphic over $Q$ by some scissors congruence.

The pullback of two piecewise morphisms
\[
P \xleftarrow{\sim} \{A_1,...,A_k\} \to R \leftarrow \{B_1,...,B_\ell\} \xrightarrow{\sim} Q 
\]
is defined similarly to composition as the span
\[
P \leftarrow \{A_i \times_R B_j\}_{i,j} \xrightarrow{\sim} \coprod\{A_i \times_R B_j\}_{i,j} \xleftarrow{\sim} \{A_i \times_R B_j\}_{i,j} \to Q
\]
for any choice of coproduct, where any two such choices are uniquely isomorphic over $P,Q$.

It is then straightforward to check that $\cGh_n$ and $\fGh_n$ are extensive categories, and that the functor $\mathsf{int}$ extends to one between subcategories of coproduct injections in $\cGh_n$ and $\fGh_n$ which preserves the ambient coproducts and pullbacks similar to \cref{extensiveprojection}.


\subsection{Double categories}\label{section:dblcats}


Double categories, originally defined as categories internal to categories, describe categorical settings with two different types of morphisms, related by higher cells called squares. In this section, we recall the well-known notions of double categories, double functors, and the natural transformations between them, as well as the space associated to a double category. We also introduce a notion of double categories with shared isomorphisms and discuss a natural notion of equivalence between them that will be useful in later sections. 

\begin{defn}\label{defn:dblcat}
A \textbf{double category} $\C$ consists of: 
\begin{itemize}
    \item a set of objects $\Ob(\C)$
    \item two categories $\M$ and $\E$ with the same objects as $\C$. We call their maps \textbf{m-morphisms} ($\mrto$) and  \textbf{e-morphisms} ($\erto$), respectively
    \item a set of squares of the form
    $$\csq{A}{B}{C}{D}{f}{g}{g'}{f'}$$
    \item categories $\Ar_\circlearrowleft \M$, $\Ar_\circlearrowleft \E$ with objects the m-morphisms (resp.\ e-morphisms) and maps from $f$ to $f'$ (resp.\ $g$ to $g'$)  given by the squares above, such that
    \item composite and identity squares respect those of the e-morphisms (resp.\ m-morphisms) along their sides, and satisfy the interchange law: in a grid  
    \begin{diagram}
        { \bullet & \bullet & \bullet \\
          \bullet & \bullet & \bullet \\ 
          \bullet & \bullet & \bullet \\};
        \mto{1-1}{1-2} \mto{1-2}{1-3}
        \mto{2-1}{2-2} \mto{2-2}{2-3}
        \mto{3-1}{3-2} \mto{3-2}{3-3}
        \eto{1-1}{2-1} \eto{2-1}{3-1}
        \eto{1-2}{2-2} \eto{2-2}{3-2}
        \eto{1-3}{2-3} \eto{2-3}{3-3}
        \comm{1-1}{2-2}
        \comm{1-2}{2-3}
        \comm{2-1}{3-2}
        \comm{2-2}{3-3}
    \end{diagram} 
    applying the composition operations in either order yields the same result.
\end{itemize}
\end{defn}

\begin{rmk}
In the definition above, we use the symbol $\circlearrowleft$ to denote that there exists a square having the depicted boundary; this should not be interpreted as the square being a commutative diagram, especially since m- and e-morphisms need not compose among each other.
\end{rmk}

\begin{notation}
To simplify nomenclature, we will refer to the squares in a double category simply as ``squares'', in contrast to \cite{CZ} where they are called ``pseudo-commutative squares''. Whenever there is a square-shaped diagram  \begin{diagram}
    { A & B \\
    A' & B' \\};
    \mto{1-1}{1-2} \mto{2-1}{2-2}
    \eto{1-2}{2-2} \eto{1-1}{2-1} 
    \end{diagram} which might not necessarily form a square in the double category, we will refer to it as a ``diagram''.
\end{notation}

\begin{defn}
Let $\C$ and $\D$ be double categories. A \textbf{double functor} $F:\C\to \D$ consists of an assignment on objects, m-morphisms, e-morphisms, and squares, which are compatible with domains and codomains and preserve all double categorical compositions and identities.
\end{defn}

\begin{defn}
A double functor is \textbf{full} (resp.\ \textbf{faithful}) if it is surjective (resp.\ injective) on each set of m-morphisms and e-morphisms with fixed source and target, and on each set of squares with fixed boundary.

We say a double subcategory $\C\subseteq\D$ is full if the inclusion is a full double functor.
\end{defn}

The category of double categories is cartesian closed, and thus there exists a double category whose objects are the double functors. We briefly describe the horizontal morphisms, vertical morphisms, and squares of this double category; the reader unfamiliar with double categories is encouraged to see \cite[\S 3.2.7]{Grandis} for more explicit definitions.

\begin{defn}\label{defn:dblcatoffunctors}
Let $F,G\colon \C\to \D$ be double functors. A horizontal natural transformation $\mu:F\Rightarrow G$, which we henceforth call \textbf{m-natural transformation}, consists of 
\begin{itemize}
    \item an m-morphism $\mu_A\colon FA\mrto GA$ in $\D$ for each object $A\in \C$, and
    \item a square $$\csq{FA}{GA}{FB}{GB}{\mu_A}{Ff}{Gf}{\mu_B}$$ in $\D$ for each e-morphism $f:A\erto B$ in $\C$,
\end{itemize}  
such that the assignment of squares is functorial with respect to the composition of e-morphisms, and that these data satisfy a naturality condition with respect to m-morphisms and squares.

Dually, one defines a vertical natural transformation, which we call \textbf{e-natural transformation}.
\end{defn}

\begin{rmk}\label{rmk:nattr}
    Equivalently, an m-natural transformation can be described as a double functor $\C \times \mathbb{H}(\Delta^1) \to \D$, where $\mathbb{H}(\Delta^1)$ is the double category with a single non-identity m-morphism. There is an analogous description of e-natural transformations using the double category $\mathbb{V}(\Delta^1)$ with a single non-identity e-morphism.
\end{rmk}

\begin{defn}\label{modification}
Given m-natural transformations $\mu: F\Rightarrow G$, $\mu': F'\Rightarrow G'$ and e-natural transformations $\eta:F\Rightarrow F'$, $\eta':G\Rightarrow G'$ between double functors $\C\to\D$, a \textbf{modification} $\alpha$ shown below left
    \[\begin{inline-diagram}
    {F & G\\
    F' & G'\\};
    \to{1-1}{1-2}^\mu \to{2-1}{2-2}_{\mu'}
    \to{1-1}{2-1}_\eta \to{1-2}{2-2}^{\eta'}
    \diagArrow{-,white}{1-1}{2-2}!{\textcolor{black}{\alpha}}
    \end{inline-diagram}\qquad
    \begin{inline-diagram}
    {FA & GA\\
    F'A & G'A\\};
    \to{1-1}{1-2}^{\mu_A} \to{2-1}{2-2}_{\mu'_A}
    \to{1-1}{2-1}_{\eta_A} \to{1-2}{2-2}^{\eta'_A}
    \diagArrow{-,white}{1-1}{2-2}!{\textcolor{black}{\circlearrowleft\alpha_A}}
    \end{inline-diagram}\]
consists of a square in $\D$ as above right for each object $A\in \C$, satisfying horizontal and vertical coherence conditions with respect to the squares of the transformations $\mu$, $\mu'$, $\eta$, and $\eta'$.
\end{defn}

The double categories of interest to this paper arise from taking m- and e-morphisms to be certain classes of morphisms in some category,  
and squares from certain commuting squares in the ambient category. For these, it will be convenient for the two classes of maps in the double category to have a common class of isomorphisms. To that purpose, we introduce the following notion.

\begin{defn}\lbl{sharedisos}
A double category $\C$ has \textbf{shared isomorphisms} if: 
\begin{itemize}
    \item there is a groupoid $I$ with identity-on-objects functors $\M \leftarrow I \rightarrow \E$ which create isomorphisms. For a morphism $f$ in $I$, we write $f$ for both the corresponding m-isomorphism and e-isomorphism, which we distinguish in diagrams by the different arrow shapes
    \item for isomorphisms $f,f'$ and m-morphisms $g,g'$ there is a (unique) square as below left if and only if the square below right commutes in $\M$
    $$\csq{\bullet}{\bullet}{\bullet}{\bullet}{g}{f}{f'}{g'}\qquad\mcsq{\bullet}{\bullet}{\bullet}{\bullet}{g}{f}{f'}{g'}$$
    \item the analogous correspondence holds between squares in $\C$ and commuting squares in $\E$ for isomorphisms $f,f'$ and e-morphisms $g,g'$
\end{itemize}
\end{defn}

In our double categories of interest,  squares between fixed m- and e-morphisms will be unique when they exist, so the uniqueness of the  squares above will be inconsequential.

The unification of m- and e-isomorphisms extends to natural isomorphisms between double functors as well, which allows us to define a canonical notion of equivalence of double categories with shared isomorphisms.

\begin{defn}\lbl{naturaliso}
Let $F,G : \C \to \D$ be double functors, where $\D$ has shared isomorphisms. A \textbf{natural isomorphism} $\alpha : F \cong G$ consists of an isomorphism $\alpha_A : FA \cong GA$ for each object $A$ in $\C$, such that when we regard all $\alpha_A$ as m-morphisms (resp.\ e-morphisms), $\alpha$ is an m- (resp.\ e-) natural transformation.
    
\end{defn}

\begin{rmk}
Note that any natural isomorphism will be such that the component squares of the m- and e-natural transformation $\alpha$ are invertible (horizontally or vertically, as it corresponds), by the uniqueness of the squares in \cref{sharedisos}. \cref{sharedisos} also shows that the naturality condition can be reduced to checking that the components of $\alpha$ form a natural transformation in the 1-categorical sense between the underlying functors on m-morphisms and e-morphisms, so it is not necessary here to provide naturality squares in the data of $\alpha$.
\end{rmk}

We can use these natural isomorphisms to define a notion of equivalence between double categories with shared isomorphisms. A careful study of these equivalences is beyond the scope of this paper; our goal is simply to show that they induce homotopy equivalences of spaces after realization.

\begin{defn}\lbl{defn:equivalence}
Let $\C,\D$ be double categories with shared isomorphisms. An \textbf{equivalence} between $\C$ and $\D$ is a pair of double functors $F : \C \leftrightarrows \D : G$ equipped with natural isomorphisms $1_\C \cong GF$ and $FG \cong 1_\D$.
\end{defn}

A definition of this form is not possible for general double categories without making arbitrary choices for whether the natural isomorphisms are m- or e-transformations. 

This is appropriate for the double categories we consider which arise from categories, and has the following convenient characterization.

\begin{prop}\lbl{equiv:char}
Let $F:\C\to \D$ be a double functor between double categories with shared isomorphisms. Then, $F$ belongs to an equivalence if and only if it is fully faithful and essentially surjective.
\end{prop}

Here essentially surjective means that every object in $\D$ is isomorphic to $FC$ for some object $C$ in $\C$, just as for ordinary categories. 

\begin{proof}
Given an equivalence $F : \C \to \D$, $F$ is essentially surjective and fully faithful on m- and e-morphisms as the restrictions $F_\M : \M_\C \to \M_\D$ and to $F_\E : \E_\C \to \E_\D$ form equivalences of categories. Lastly, $F$ is fully faithful on squares relative to their boundaries, since $F$ is a double biequivalence (see \cite[Definition 3.7]{MSV})  by \cite[Proposition 5.14]{MSV}, and thus in particular fully faithful on squares.

Given a fully faithful and essentially surjective double functor $F : \C \to \D$, by the classical characterization of equivalences of categories, both $F_\M$ and $F_\E$ form equivalences of categories. Furthermore, as the objects and isomorphisms of $\M$ and $\E$ are the same for both $\C$ and $\D$ (in the sense of shared isomorphisms), the quasi-inverses $G_\M : \M_\C \to \M_\D$ and $G_\E : \E_\C \to \E_\D$ can be chosen to agree on objects and such that the isomorphisms $F_\M G_\M D \cong D$ and $F_\E G_\E D \cong D$ also agree, as in the classical construction of these quasi-inverses those choices are made arbitrarily (see, for example, \cite[Theorem 1.5.9]{emily}). It follows immediately from the proof in loc.\ cit.\ that under these choices, the isomorphisms $C \cong G_\M F_\M C$ and $C \cong G_\E F_\E C$ agree as well, by observing that any double functor between double categories with shared isomorphisms preserves the correspondence between m- and e-isomorphisms.

We can now define a double functor $G : \D \to \C$ which restricts to $G_\M$ on $\M_\D$ and $G_\E$ on $\E_\D$. It remains only to define how $G$ acts on squares; given a square $\alpha$ in $\D$ as below left, we construct the square below right in the image of $F$.
$$\csq{D_1}{D_2}{D_3}{D_4}{}{}{}{}\hspace{2cm}
\begin{inline-diagram}
{FGD_1 & FGD_1 & FGD_2 & FGD_2 \\
FGD_1 & D_1 & D_2 & FGD_2 \\
FGD_3 & D_3 & D_4 & FGD_4 \\
FGD_3 & FGD_3 & FGD_4 & FGD_4 \\};
\eqto{1-1}{1-2} \eqto{1-1}{2-1} \eqto{1-3}{1-4} \eqto{1-4}{2-4} \eqto{3-1}{4-1} \eqto{4-1}{4-2} \eqto{3-4}{4-4} \eqto{4-3}{4-4}
\mto{1-2}{1-3} \mto{2-2}{2-3} \mto{3-2}{3-3} \mto{4-2}{4-3}
\eto{2-1}{3-1} \eto{2-2}{3-2} \eto{2-3}{3-3} \eto{2-4}{3-4}
\mto{2-1}{2-2} \mto{2-3}{2-4} \mto{3-1}{3-2} \mto{3-3}{3-4}
\eto{1-2}{2-2} \eto{1-3}{2-3} \eto{3-2}{4-2} \eto{3-3}{4-3}
\comm{1-1}{2-2} \comm{1-2}{2-3} \comm{1-3}{2-4}
\comm{2-1}{3-2} \comm{2-2}{3-3} \comm{2-3}{3-4}
\comm{3-1}{4-2} \comm{3-2}{4-3} \comm{3-3}{4-4}
\end{inline-diagram}$$
The outer squares on the right above are squares in the double category by \cref{sharedisos} and by naturality of the isomorphisms $F_\M G_\M D \cong D$ and $F_\E G_\E D \cong D$ in $\M_\D$, $\E_\D$. As $F$ is fully faithful on squares, this composite square has a unique preimage in $\C$, which we define to be $G(\alpha)$.

It is then tedious but straightforward to check that $G$ respects identities and composites of squares, and that the isomorphisms $C \cong GFC$ and $FGD \cong D$ for $C$ in $\C$ and $D$ in $\D$ are natural, making $F,G$ into an equivalence of double categories.
\end{proof}

Finally, we recall that the process of constructing a space from a category by taking the geometric realization of its nerve has an analogue in double categories, as defined for example in \cite[Definition 2.14]{FP}. This is an especially important construction for us, as it will be used to define the $K$-theory space of our double categories of interest.

\begin{defn}
The double nerve, or \textbf{bisimplicial nerve}, of a double category $\C$ is the bisimplicial set $N_\square\C$ whose $(m,n)$-simplices are the $m \times n$-matrices of composable squares in $\C$.
\end{defn}

We let $|\C|$ denote the geometric realization of the bisimplicial set $N_\square \C$, or, equivalently, the geometric realization of its diagonal simplicial set $n\mapsto N_\square \C_{n,n}$. Going forward, we abuse notation and use these two spaces interchangeably.

\begin{lemma}\label{equivalent_spaces}
Let $\C,\D$ be double categories with shared isomorphisms. If there exists an equivalence between $\C$ and $\D$, then $|\C|$ and $|\D|$ are homotopy equivalent. Moreover, the same holds for any pair of double functors $F : \C \leftrightarrows \D : G$ equipped with natural transformations between $1_\C$ and $GF$, and between $FG$ and $1_\D$, in any combination of types (m- or e-) and directions.
\end{lemma}

\begin{proof}
This follows from the description of m-natural (resp. e-natural) transformations of \cref{rmk:nattr} as a double functor $\C \times \mathbb{H}(\Delta^1) \to \D$ (resp. $\C \times \mathbb{V}(\Delta^1\to \D$), where $\mathbb{H}(\Delta^1)$ and $\mathbb{V}(\Delta^1)$ both geometrically realize to the interval. 
\end{proof}


\section{$\fcgwa$ categories}\label{section:prefcgw}


In this section, we introduce $\fcgwa$ categories and establish the necessary categorical yoga. The purpose of these double categories is to capture the essential features of exact categories that make them so suitable for $K$-theory, while allowing for non-additive examples. $\fcgwa$ categories have two classes of maps that mimic the role of admissible monomorphisms and (the opposite of) admissible epimorphisms: these will be the m- and e-morphisms in the double category. They also have notions of (co)kernels and short exact sequences, but instead of defining these as certain (co)limits that would require an additive setting, their relevant features are axiomatized. This allows one to expand the classical intuition from exact categories to other settings such as sets and varieties, as done in \cite{CZ}. 

\subsection{A motivating example: exact categories}\lbl{subsec:exactcats}

Before jumping into our double-categorical framework, we start with a brief overview of exact categories, highlighting precisely the features that we will seek to capture later on. We will index our observations with labels that will correspond to each of our axioms. 

Let $\C$ be an exact category, and let $\M\subseteq\C$ (resp.\ $\E\subseteq\C$) denote the subcategory with the same objects as $\C$, and whose maps are the admissible monomorphisms (resp.\ epimorphisms). 

\begin{itemize}
    \item[(Z)] Note that the zero object $0\in\C$ is an initial object in $\M$, and a terminal object in $\E$.
    \item[(M)] Of course, by definition, all maps in $\M$ are monic, and all maps in $\E$ are epic.
    \item[(K)] Every map $i\colon A\mrto B$ in $\M$ has a cokernel $B\twoheadrightarrow\cok i$ in $\E$ yielding a short exact sequence 
\begin{diagram}
        {A  & B & \cok i\\};
        \mto{1-1}{1-2}^{i} \diagArrow{->>}{1-2}{1-3}
          \end{diagram}
     and dually every map $p\colon B\twoheadrightarrow C$ in $\E$ has a kernel $\ker p\mrto B$ in $\M$ yielding a short exact sequence. In the interest of expressing this in terms of squares, recall that having a short exact sequence as depicted above is equivalent to having a bicartesian square
        \begin{diagram-numbered}{diag:ses}
        {A & B\\
        0 & \cok i\\};
        \mto{1-1}{1-2}^{i} \mto{2-1}{2-2} \diagArrow{->>}{1-2}{2-2} \diagArrow{->>}{1-1}{2-1} 
    \end{diagram-numbered}
    \end{itemize}
    
    In fact, (co)kernels are functorial: given a commutative diagram in $\M$ as below left, if we take cokernels in the horizontal direction we get an induced morphism $\cok i\to\cok i'$ in $\C$.  
    \begin{diagram}
    {A & B & \cok i\\
    A' & B' & \cok i'\\};
    \mto{1-1}{1-2}^i \diagArrow{->>}{1-2}{1-3}
    \mto{1-1}{2-1} \mto{1-2}{2-2} \diagArrow{dashed,->}{1-3}{2-3}
    \mto{2-1}{2-2}^{i'} \diagArrow{->>}{2-2}{2-3}
    \end{diagram}
Since we only want to discuss monos and epis, we wish to know when the dashed map above is an admissible mono. One can check that this map will be a mono if and only if the commutative square on the left is cartesian. We claim that this mono will be admissible precisely when the induced morphism out of the pushout $B\cup_A A'\to B'$ is an admissible mono. Indeed, one can factor the diagram above as follows, where all rows are exact
\begin{diagram}
    {A & B & \cok i\\
    A' & B\cup_A A' & \cok i\\
    A' & B' & \cok i'\\};
    \mto{1-1}{1-2}^i \diagArrow{->>}{1-2}{1-3}
    \mto{1-1}{2-1} \mto{1-2}{2-2} \diagArrow{dashed,->}{2-3}{3-3}
    \mto{2-1}{2-2} \diagArrow{->>}{2-2}{2-3} \eq{1-3}{2-3}
    \mto{3-1}{3-2}^{i'} \diagArrow{->>}{3-2}{3-3} 
    \eq{2-1}{3-1} \to{2-2}{3-2}
    \end{diagram}
Applying the Snake Lemma to the bottom part of the diagram, we see that $$\cok (B\cup_A A'\to B')\cong \cok (\cok i\to\cok i');$$ thus, one of these monos is admissible if and only if the other one is. We conclude that for cokernels to have the appropriate functoriality, we must restrict ourselves to the subclass of cartesian squares in $\M$ with this additional pushout property; we call these \textbf{``good squares''}. The dual reasoning yields a class of good squares in $\E$.

On the other hand, if we start with a commutative square as below right and take kernels in the horizontal direction, we get an induced morphism $\ker p\to \ker p'$. 
\begin{diagram}
    {\ker p & B & C\\
    \ker p' & B' & C'\\};
    \mto{1-1}{1-2}^{i} \diagArrow{->>}{1-2}{1-3}^{p}
    \diagArrow{dashed,->}{1-1}{2-1}_{j} \mto{1-2}{2-2} \mto{1-3}{2-3}
    \mto{2-1}{2-2}^{i'} \diagArrow{->>}{2-2}{2-3}^{p'}
    \end{diagram}
This map will always be a mono, but not necessarily admissible. To ensure that it is, we must either ask for this directly (that is, restrict to the commutative squares that induce an admissible mono on kernels), or alternatively, ask that our exact category $\C$ be \emph{weakly idempotent complete}. In the latter case, the fact that the induced map $j$ is such that $i' j$ is an admissible mono implies that $j$ is also admissible, as proven in \cite[Proposition 7.6]{Buh10}. Every exact category admits an ``idempotent completion'' that does not affect their $K$-theory away from $K_0$, so this assumption is relatively harmless for $K$-theoretical purposes.

To summarize, if we do the appropriate restrictions we see that the kernel functor gives us an equivalence of categories \[k\colon \Ar_\circlearrowleft \E \to \Ar_{\g} \M\] 
\begin{diagram}
    {A & B &  & B & \ker p\\
    A & B &  & B & \ker p\\
    A' & B' & & B' & \ker p'\\};
    \diagArrow{->>}{1-1}{1-2}^{p} \mto{1-5}{1-4}
    \diagArrow{->>}{2-1}{2-2}^{p} \mto{2-5}{2-4} \diagArrow{->>}{3-1}{3-2}^{p'} \mto{3-5}{3-4} \mto{2-1}{3-1}_{i} \mto{2-2}{3-2}^{i'} \mto{2-4}{3-4}_{i'} \mto{2-5}{3-5} \labar{2-2}{3-4}{\mapsto} \labar{1-2}{1-4}{\mapsto} \labar{2-1}{3-2}{\circlearrowleft} \labar{2-4}{3-5}{\text{g}}
\end{diagram}
where morphisms are commutative and ``good'' squares, respectively.
Dually, the cokernel functor gives an equivalence of categories $c\colon \Ar_\circlearrowleft \M \to \Ar_{\g} \E$. 
 \begin{itemize}
     \item[(D)] Bicartesian squares play a special role when it comes to (co)kernels. Aside from being the squares that define short exact sequences as in \cref{diag:ses}, we have that a commutative square is bicartesian if and only if it induces an isomorphism on kernels, and this happens if and only if it induces an isomorphism on cokernels.
     \begin{diagram}
         {\ker p & \ker p' & \\
         A & B & \cok i\\
         A' & B' & \cok i'\\};
         \diagArrow{dashed, mmor}{1-1}{1-2}^{\cong} \mto{2-1}{2-2}^{i} \mto{3-1}{3-2}^{i'} \mto{1-1}{2-1} \mto{1-2}{2-2} \diagArrow{->>}{2-1}{3-1}_{p} \diagArrow{->>}{2-2}{3-2}^{p'} \diagArrow{->>}{2-2}{2-3} \diagArrow{->>}{3-2}{3-3} \diagArrow{dashed,->>}{2-3}{3-3}^{\cong}
     \end{diagram}
 \end{itemize}

 Another important feature of exact categories is the fact that admissible monos are closed under pushouts, and dually, admissible epis are closed under pullbacks. Let us mention some of these properties more explicitly:
\begin{itemize}
    \item[(PO)] Every span of admissible monos $C \mlto^{j} A \mrto^{i} B$ can be completed to a good square by taking its pushout
    \begin{diagram}
        {A & B\\
        C & B\cup_A C\\};
        \mto{1-1}{1-2}^{i} \mto{1-1}{2-1}_{j} \mto{1-2}{2-2}^{j'} \mto{2-1}{2-2}_{i'} \good{1-1}{2-2}
    \end{diagram} since pushouts of admissible monos are also pullbacks.
    \item[($\star$)] The above completion is initial among good squares (by the universal property of the pushout), and the induced maps $\cok i\mrto\cok i'$ and $\cok j\mrto\cok j'$ are isomorphisms.
    \item[(POL)] The interaction between good squares and pushouts gives a sort of ``pushout lemma'', in the sense that if the outer square in the commutative diagram below is good, then the right square is also good.
    \begin{diagram}
    {A & B & D\\
    C & B\cup_A C & E\\};
    \mto{1-1}{1-2} \mto{1-2}{1-3}
    \mto{2-1}{2-2} \mto{2-2}{2-3}
    \mto{1-1}{2-1} \mto{1-2}{2-2}\mto{1-3}{2-3}
    \good{1-1}{2-2}
    \end{diagram}
    To check this, we can see that the induced map $D\cup_B (B\cup_A C)\to E$ is (up to isomorphism) the map $D\cup_A C\to E$ which is an admissible mono by assumption. This implies that the square on the right is cartesian, since it factors as the composite
    \begin{diagram}
        {B & D\\
        B\cup_A C & D\cup_B (B\cup_A C)\\
        B\cup_A C & E\\};
        \mto{1-1}{1-2}\mto{1-1}{2-1}\mto{1-2}{2-2}\mto{2-1}{2-2}\eq{2-1}{3-1}\mto{3-1}{3-2}\mto{2-2}{3-2}
    \end{diagram} where the top square is a pushout of monos and hence also a pullback, and the bottom square is a pullback as $D\cup_B (B\cup_A C)\to E$ is monic.
    \item[(PBL)] Finally, we note that the type of commutative squares we consider here satisfies a sort of ``pullback lemma'': if the outer square and the right square below are commutative,  then so is the square on the left
    \begin{diagram}
    {A & B & C\\
    A' & B' & C'\\};
    \mto{1-1}{1-2} \mto{1-2}{1-3}
    \diagArrow{->>}{1-1}{2-1} \diagArrow{->>}{1-2}{2-2} \diagArrow{->>}{1-3}{2-3}
    \mto{2-1}{2-2} \mto{2-2}{2-3}
    \end{diagram} using the fact that $B'\mrto C'$ is a mono. 
\end{itemize}
We obtain dual properties if we focus on the admissible epis.

\subsection{$\fcgwa$ categories}\label{subsec:fcgwa}

In this section we define our main structures: $\fcgwa$ categories. Throughout this paper, we work with several categories with objects the m- or e-morphisms of $\C$, such as $\Ar_\circlearrowleft \M$, $\Ar_\circlearrowleft \E$ introduced in \cref{defn:dblcat}, so we begin by setting some notation. 

\begin{notation}[{\cite[Definition 2.4]{CZ}}]
Given a category $\A$, let $\Ar_\triangle \A$ denote the category whose objects are morphisms $A\to B$ in $\A$, and where 
\[\Hom_{\Ar_\triangle\A}(\makeshort{A \to^f B, A' \to^{f'} B'})  =
    \left\{
      \begin{tabular}{c}
        commutative \\ squares
      \end{tabular}
      \begin{inline-diagram}
      {A & B\\
      A'  & B'\\};
      \to{1-1}{1-2}^{f} \to{1-2}{2-2} \to{1-1}{2-1}_{\cong} \to{2-1}{2-2}^{f'}
      \end{inline-diagram}
      \right\}.\]
\end{notation}


\begin{defn}\label{defn:classofgoodsquares}
Let $\A$ be a category whose morphisms are all monic. A class of \textbf{good squares} is a class of commutative squares in $\A$ which are all cartesian and that contains all the squares in $\Ar_\triangle \A$. 
\end{defn}

\begin{notation}
    We denote by $\Ar_{\g} \A$ the category whose objects are maps in $\A$ and whose morphisms are good squares, and we often label diagrams in $\A$ that are good squares by
\begin{diagram}
    { \bullet & \bullet \\ \bullet & \bullet \\};
    \arrowsquare{}{}{}{}
    \good{1-1}{2-2}
\end{diagram}
\end{notation}


We now define $\fcgwa$ categories. The reader unfamiliar with (A)CGW categories is strongly encouraged to read each axiom together with its counterpart for exact categories, described in \cref{subsec:exactcats}.

\begin{defn}
\lbl{preFCGW}
An \textbf{$\fcgwa$ category} is a double category $\C=(\M,\E)$ with shared isomorphisms, equipped with
\begin{itemize}
    \item classes of good squares $\Ar_{\g} \M$ and $\Ar_{\g} \E$ 
    \item equivalences of categories 
\begin{flushleft}\hspace{0.7cm} $k\colon \Ar_\circlearrowleft \E \to \Ar_{\g} \M   \hspace{4.5cm} c\colon \Ar_\circlearrowleft \M \to \Ar_{\g} \E$ \end{flushleft}
\[\begin{inline-diagram}
    {A & B  & & A' & B'\\
    A' & B' & & \ker p & \ker p'\\};
    \mto{1-1}{1-2}^{i} \mto{1-4}{1-5}^{i'} \mto{2-1}{2-2}_{i'} \mto{2-4}{2-5} \eto{1-1}{2-1}_{p} \eto{1-2}{2-2}^{p'} \mto{2-4}{1-4}^{k(p)} \mto{2-5}{1-5}_{k(p')} \labar{1-2}{2-4}{\mapsto}  \labar{1-1}{2-2}{\circlearrowleft} \labar{1-4}{2-5}{\text{g}}
\end{inline-diagram} \hspace{0.5cm} \text{and} \hspace{0.5cm}
\begin{inline-diagram}
    {A & B  & & B & \cok i\\
    A' & B' & & B' & \cok i'\\};
    \mto{1-1}{1-2}^{i} \eto{1-5}{1-4}_{c(i)} \mto{2-1}{2-2}_{i'} \eto{2-5}{2-4}^{c(i')} \eto{1-1}{2-1}_{p} \eto{1-2}{2-2}^{p'} \eto{1-4}{2-4}_{p'} \eto{1-5}{2-5} \labar{1-2}{2-4}{\mapsto}  \labar{1-1}{2-2}{\circlearrowleft} \labar{1-4}{2-5}{\text{g}}
\end{inline-diagram}\]
\end{itemize}
such that 
\begin{enumerate}
    \item[(Z)]\label{Z} $\M,\E$ each have initial objects which agree.
    
    \item[(M)]\label{M} All morphisms in $\M,\E$ are monic.
    
    
    \item[(D)]\label{D} $k$ sends a  square to $Ar_\triangle \M \subset Ar_{\g} \M$ if and only if $c$ sends the square to $Ar_\triangle \E \subset Ar_{\g} \E$. In this case the square is called \textbf{distinguished} and denoted by
   \[\dsq{A}{B}{C}{D}{}{}{}{}\]
    
    \item[(K)]\label{K} For any m-morphism $f : A \mrto B$ there is a distinguished square as below left, and for any e-morphism $g : A \erto B$ there is a distinguished square as below right.
    \[\dsq{\varnothing}{B/A}{A}{B}{}{}{c(f)}{f}
    \qquad\qquad
    \dsq{\varnothing}{A}{B\bs A}{B}{}{}{g}{k(g)}\]
    The notation $B/A,B \bs A$ will be used when the defining maps $f$ and $g$ are clear from context. Otherwise, these objects will be denoted $\cok f, \ker g$ respectively.
    \item[($\star$)]\label{PS} For every diagram $C \mlto A \mrto B$, if the category of good squares as below left (with morphisms maps $D \mrto D'$ commuting under $B$ and $C$) is non-empty, then it has an initial object which we write $D = B \star_A C$.
    \begin{diagram}
    {A & B & \qquad & A & B & B/A\\
    C & D & & C & B\star_A C & B\star_A C/C\\};
    \mto{1-1}{1-2} \mto{1-4}{1-5} \eto{1-6}{1-5}
    \mto{1-1}{2-1} \mto{1-2}{2-2} \mto{1-4}{2-4} \mto{1-5}{2-5} \mto{1-6}{2-6}^\cong
    \mto{2-1}{2-2} \mto{2-4}{2-5} \eto{2-6}{2-5}
    \good{1-1}{2-2} \good{1-4}{2-5} \comm{1-5}{2-6}
    \end{diagram}
    Furthermore, the induced maps $B / A \mrto B \star_A C / C$ and $C / A \mrto B \star_A C / B$ are isomorphisms (above right). The dual statement holds for spans of e-morphisms.
    
    \item[(PO)]\label{star}  For every diagram $C \mlto A \mrto B$, the category of good squares as in axiom ($\star$) is non-empty. The dual statement need not hold for spans of e-morphisms.
    
    \item[(PBL)] Squares in $\C$ satisfy the ``pullback lemma'': if the right and outer diagrams below are squares in $\C$, then so is the diagram on the left.
    \begin{diagram}
    {A & B & C\\
    A' & B' & C'\\};
    \mto{1-1}{1-2} \mto{1-2}{1-3}
    \eto{1-1}{2-1} \eto{1-2}{2-2} \eto{1-3}{2-3}
    \mto{2-1}{2-2} \mto{2-2}{2-3}
    \end{diagram}
    The analogous statement holds for composites in the e-direction.
    
    \item[(POL)] \lbl{pushout:lemma}
If the outer square in the commutative diagram below is good, then so the right square.
    \begin{diagram}
    {A & B & D\\
    C & B\star_A C & E\\};
    \mto{1-1}{1-2} \mto{1-2}{1-3}
    \mto{2-1}{2-2} \mto{2-2}{2-3}
    \mto{1-1}{2-1} \mto{1-2}{2-2}\mto{1-3}{2-3}
    \good{1-1}{2-2}
    \end{diagram}
    The same property holds for e-morphisms when the $\star$-pushout exists.
\end{enumerate}
\end{defn}

\begin{rmk} 
Axiom (PO), which says that any span of morphisms in $\M$ admits a ``pushout'', is the only part of our framework that is not symmetric in the horizontal and vertical directions. This property is not required of the maps in $\E$, where instead we only expect a ``pushout'' if the given span is already known to be part of a good square. While this distinction is not necessary in an exact category where we have all pullbacks of admissible epis, the reader curious about this asymmetry is directed to \cref{ex:varieties}.

The need for these ``pushouts'' arises when studying the classical proofs of the Additivity Theorem (see, for example, \cite{McCarthy}, \cite[Section 1.4]{Waldhausen}, \cite[Chapter V, Theorem 1.3]{Kbook}). We will see that $\star$-pushouts are adequately functorial and allow for a construction of $\star$ in categories of diagrams; in particular, this will allow us to define an $S_\bullet$ construction that can be iterated. 
A more detailed study of the properties of the $\star$-pushout can be found in \cref{appendix}. In fact, a detailed study of the proof of the Additivity \cref{additivity} reveals the need for all of the above axioms in the definition of our $\fcgwa$ categories.
\end{rmk}

\begin{rmk}\lbl{goodsqpushout}
The good squares are meant to behave like the cofibrations in Waldhausen's category $F_1\C$. Recall that, given a Waldhausen category $\C$, $F_1\C$ is the subcategory of $\Ar\C$ whose objects are the cofibrations. Here, a morphism
    \begin{diagram}
    {A & B\\
    C & D\\};
    \mto{1-1}{1-2} \mto{2-1}{2-2}
    \to{1-1}{2-1} \to{1-2}{2-2}
    \end{diagram}
is a cofibration if the maps $A\to C$, $B\to D$ and $B\cup_A C\to D$ are cofibrations. 

In our setting, the pushout is replaced by the $\star$-pushout, whose universal property is limited to m-morphisms and good squares, and by axiom ($\star$) all good squares are such that there is an induced m-morphism $B\star_A C\mrto D$. Moreover, the converse also holds, and so this property characterizes good squares. Indeed, given a commutative square as below left
\[\begin{inline-diagram}
    {A & B\\
    C & D\\};
    \mto{1-1}{1-2} \mto{2-1}{2-2}
    \mto{1-1}{2-1} \mto{1-2}{2-2}
    \end{inline-diagram} \qquad
    \begin{inline-diagram}
    {A & B & B\\
    C & B\star_A C & D \\};
    \mto{1-1}{1-2} \mto{2-1}{2-2} \eq{1-2}{1-3}
    \mto{1-1}{2-1} \mto{1-2}{2-2} \mto{2-2}{2-3}
    \mto{1-3}{2-3}
    \end{inline-diagram}
    \]
together with an m-morphism $B\star_A C\mrto D$ over $D$, we can rewrite it as the composite above right, which implies the square is good.
\end{rmk}

As usual, $\fcgwa$ categories have natural notions of functors and subcategories.

\begin{defn}
An \textbf{$\fcgwa$ functor} is a double functor that preserves all of the relevant structure (the initial object, good squares, kernels, cokernels, and $\star$-pushouts) up to natural isomorphism.
\end{defn}

\begin{rmk}
    The fact that $\fcgwa$ functors preserve $\star$-pushouts is guaranteed by the preservation of the rest of the structure, and does not need to be checked separately. Indeed, given a double functor between $\fcgwa$ categories that preserves initial objects, good squares, kernels and cokernels, one can use the fact that, by \cref{pushout:uniqueness}, $\star$-pushouts can be characterized as the good squares which induce isomorphisms after applying $k^{-1}$ and $c^{-1}$.
\end{rmk}

\begin{defn}
A double subcategory $\D$ of a $\fcgwa$ category $\C$ is an \textbf{$\fcgwa$ subcategory} if it inherits a $\fcgwa$ structure from $\C$.
\end{defn}

For full double subcategories of an $\fcgwa$ category, many of the axioms are automatically preserved, so it is easy to check whether they are $\fcgwa$.

\begin{lemma}\lbl{fullsubFCGW}
A full double subcategory of an $\fcgwa$ category $\C$ is an $\fcgwa$ subcategory if it is closed under $k,c,\star$, and contains $\varnothing$.
\end{lemma}

We conclude this section with a few useful technical results. 



\begin{lemma}
The functors $k$ and $c$ are inverses on objects up to codomain-preserving isomorphism. 
\end{lemma}
\begin{proof}
   This is a direct consequence of axiom (K) in  \cref{preFCGW}.
\end{proof}

\begin{rmk}
    The above result invites us to consider distinguished squares of the form below as extensions of $A$ by $B$, which is exactly what they are in \cref{ex:exactcgw}.
$$\dsq{\varnothing}{B}{A}{C}{}{}{}{}$$
\end{rmk}




\begin{lemma}\lbl{isosemptycoker}
An m-morphism (resp. e-morphism) in a $\fcgwa$ category is an isomorphism if and only if its cokernel (resp. kernel) has initial domain.
\end{lemma}
\begin{proof}
This is a straightforward generalization of \cite[Lemma 2.8]{CZ}. 
\end{proof}

\begin{lemma}\lbl{mixedpullbackuniqueness}
In an $\fcgwa$ category, if there exists a square as below right completing the mixed cospan below left, then it is unique up to unique isomorphism.
    \begin{diagram}
    { & \bullet & \qquad & \bullet & \bullet\\
    \bullet & \bullet & & \bullet & \bullet\\};
    \mto{1-4}{1-5} 
    \eto{1-2}{2-2}^g \eto{1-4}{2-4} \eto{1-5}{2-5}^g
    \mto{2-1}{2-2}^f \mto{2-4}{2-5}^f
    \comm{1-4}{2-5}
    \end{diagram}
\end{lemma}
\begin{proof}
Given any such square, applying the inverse equivalence $c^{-1}$ yields the pullback square of $g$ and $c(f)$, as seen in the following diagram
\vspace{-0.3cm}
    \begin{diagram}
    {\bullet & \bullet & \bullet\\
    \bullet & \bullet & \bullet\\};
    \mto{1-1}{1-2} \eto{1-3}{1-2}
    \mto{2-1}{2-2}_f \eto{2-3}{2-2}^{c(f)}
    \eto{1-1}{2-1} \eto{1-2}{2-2}^g \eto{1-3}{2-3}
    \comm{1-1}{2-2} 
    \path[font=\scriptsize] (m-1-3) edge[-,white] node[pos=0.05]
  {\textcolor{black}{$\llcorner$}}  (m-2-2);
    \end{diagram}
    \vspace{-0.5cm}
Since pullbacks and kernel-cokernel pairs are unique up to unique isomorphism, the same must be true of this square.
\end{proof}

In particular, the above lemma implies that a square in an $\fcgwa$ category (if it exists) is unique relative to its boundary. Then,  for a given square of m- and e-morphisms, the existence of a  square filler can be treated as a property rather than data. This justifies our notation of squares in the double category by the symbol $\circlearrowleft$ instead of by additional data.

\subsection{Examples}

In this section we present some examples of $\fcgwa$ categories, detailing all the data that comprises their $\fcgwa$ structure. The cases of exact categories, finite sets, and varieties (\cref{ex:exactcgw,ex:setscgw,ex:varietiescgw}) already appear as examples of the framework in \cite{CZ}, and we highlight the new addition of extensive categories in \cref{ex:extensivecgw}, which as studied in \cref{subsec:extensive} include examples like finite $G$-sets and polytopes, among others.



\begin{ex}[Exact categories]\label{ex:exactcgw}
    The axioms of $\fcgwa$ categories were engineered so that any (weakly idempotent complete) exact category $\C$ would be an example; in particular, this includes all abelian categories. The double category to consider has the same objects as $\C$, the class $\M$ consists of the admissible monos, and for the symmetry to work we let 
    $$\E=\{\textrm{admissible epimorphisms}\}^{\op}.$$ The squares in the double category are commutative squares, and the good squares in $\M$ are the cartesian diagrams
     \begin{diagram}
    {A & B \\
    A' & B' \\};
    \mto{1-1}{1-2} 
    \mto{1-1}{2-1} \mto{1-2}{2-2} 
    \mto{2-1}{2-2}
    \end{diagram}
such that the induced map $B\cup_A A'\to B'$ is an admissible mono; dually, we define the good squares in $\E$. The functors $k$ and $c$ are defined from the usual kernel and cokernel functors. Distinguished squares are the bicartesian squares, and $\star$-pushouts are the pushouts. 

That this double category has shared isomorphisms follows immediately, as a map in an exact category is an isomorphism if and only if it is both an admissible mono and an admissible epi. A discussion of the rest of the axioms can be found in \cref{subsec:exactcats}.

    
\end{ex}

\begin{rmk}
If the exact category $\C$ in the previous example is abelian, then all monos and epis are admissible and the $\fcgwa$ structure is somewhat simplified: the good squares are precisely the pullbacks of monos or pushouts of epis.
\end{rmk}

\begin{rmk}
(Weakly idempotent complete) exact categories do not in general have all pullbacks, and so they are not examples of ACGW , or even pre-ACGW categories in \cite{CZ}.

Even when pullbacks exist, our restriction from pullback squares to good squares is not vacuous. To see this, let $\C$ denote the exact category of finitely generated projective (i.e., free) abelian groups. This category is idempotent complete, and thus weakly idempotent complete. If we consider the diagram
    \begin{diagram}
    {0 & \mathbb{Z}\\
    \mathbb{Z} & \mathbb{Z}\oplus \mathbb{Z}\\};
    \mto{1-1}{1-2} \mto{2-1}{2-2}_f
    \mto{1-1}{2-1} \mto{1-2}{2-2}^d
    \end{diagram}
where $d$ is the diagonal map $d(x)=(x,x)$ and $f$ is given by $f(x)=(x,-x)$, we see that this is a pullback square in $\C$ which is not good. Indeed, the map induced on cokernels is the mono $i:\mathbb{Z}\mrto\mathbb{Z}$ given by $i(x)=2x$, which is not admissible since its cokernel $\mathbb{Z}/2\mathbb{Z}$ is not free.
\end{rmk}

\begin{ex}[Finite sets]\lbl{ex:setscgw}
We can define a double category of finite sets 
by setting     \begin{center}
    $\M=\E=\{$injective functions$\}$ 
    \end{center}
and letting both good squares and squares in the double category be the pullback squares. Both  functors $k$ and $c$ take an injection $A\to B$ to the inclusion of the complement of its image $B\bs A\to B$. The initial object is the empty set,  the distinguished squares are the bicartesian squares, and $\star$-pushouts are pushouts; this is the same as its  ACGW category structure from \cite[Example 3.3]{CZ}. 

Axiom (PBL) holds as pullbacks satisfy the pullback lemma. Axiom (PO) follows from the existence of pushouts of injections, and axiom ($\star$) follows from the universal property of the pushout and the observation that a square of injections induces an injection from the pushout precisely when the original square is a pullback. Axiom (POL) can be deduced from the distributivity of intersections over unions among subsets. In this setting, the diagram in the axiom can be written as
\begin{diagram}
   {B \cap D & B & C\\
   D & B \cup D & E\\};
   \mto{1-1}{1-2} \mto{1-2}{1-3}
   \mto{2-1}{2-2} \mto{2-2}{2-3}
   \mto{1-1}{2-1} \mto{1-2}{2-2}\mto{1-3}{2-3}
   \end{diagram}
where the union and intersection are taken with respect to $E$. If the outer square is good (a pullback), we have $C \cap D = B \cap D$. It follows that 
$$C \cap (B \cup D) = (C \cap B) \cup (C \cap D) = B \cup (B \cap D) = B$$
so the right square is also a pullback.
\end{ex}

\begin{rmk}
    The examples of exact categories and finite sets illustrate why we cannot assume that the m- and e- morphisms both belong to some ambient ordinary category, as to model an exact category we must reverse the direction of the admissible epimorphisms. While one could instead rewrite the definition in a vertically-dual manner to avoid this, then in the case of finite sets the inclusions would have to be reversed to form the e-morphisms. The advantage of using double categories is that these directions need not agree, as in both cases the relationships between two types of morphisms are most meaningfully expressed using squares which are indifferent to this choice of direction.
\end{rmk}

The category of finite sets has the property that all injective functions are ``coproduct injections'' of the form $A \hookrightarrow A \sqcup B$ up to isomorphism. Coproduct injections come with a natural choice of complement, the opposing coproduct injection $A \sqcup B \hookleftarrow B$, which lets us construct an $\fcgwa$ category of coproduct injections in any extensive category.

\begin{ex}[Extensive categories]\lbl{ex:extensivecgw}
Any extensive category $\X$ defines an $\fcgwa$ category with 
    \begin{center}
    $\M=\E=\{$coproduct injections$\}$ 
    \end{center}
and letting both good squares and squares in the double category be the pullback squares. Both functors $k$ and $c$ take a coproduct injection to its complementary coproduct injection, and by the second axiom for extensive categories this assignment extends appropriately to squares as in the previous example. The distinguished squares are those of the form below up to isomorphism,
\[\dsq{A}{A \sqcup B}{C \sqcup A}{C \sqcup A \sqcup B}{}{}{}{}\]
and the $\star$-pushout of a span $C \sqcup A \mlto A \mrto A \sqcup B$ is given by the triple coproduct $C \sqcup A \sqcup B$. All axioms are deduced from the properties of extensive categories in a manner identical to \cref{ex:setscgw}.

As discussed in \cref{subsec:extensive}, these include among other examples finite $G$-sets and both types of categories of polytopes, where in the latter case the m- and e-morphisms are both piecewise functions $P \xleftarrow{\sim} \{A_1,...,A_k\} \to Q$ which are injective in the sense that each point in $Q$ is in the image of a point in at most one $A_i$.

The functors between extensive categories from \cref{extensiveprojection}, which preserved coproducts and pullbacks but were only defined on coproduct injections, are now precisely the data required to define an $\fcgwa$ functor between the associated $\fcgwa$ categories ((co)kernels and pushouts are also preserved under these conditions). For instance, the functor $\mathsf{int} \colon \cGh_n \to \fGh_n$ on coproduct injections between polytopes sending a closed $n$-polytope to its interior is an $\fcgwa$ functor.
\end{ex}

Finally, we consider examples which behave similarly to extensive categories but in which the m- and e-morphisms are not restricted to coproduct inclusions.

\begin{ex}[Varieties]\lbl{ex:varietiescgw}\lbl{ex:varieties}
We can define a double category $\Var$ whose objects are varieties, with m- and e-morphisms given by
  \[\M =\{\mathrm{closed\ immersions}\} \qqand \E = \{\mathrm{open\ immersions}.\}\] 
Like the examples above, good squares and squares in the double category are the pullback squares (as varieties are closed under pullbacks), and the functors $k$ and $c$ take a morphism to the inclusion of its complement. Axioms (Z) and (M) are easily checked, and this is clearly a double category with shared isomorphisms. For axiom (D), one can verify that the distinguished squares \[\dsq ABCD{}{}{f}{g}\] are the pullback squares in which $\im f \cup \im g = D$.  Then, axiom (K) holds directly as well.

The structure in this example so far is almost identical to \cite[Example 3.4]{CZ}, except we swap open and closed immersions when defining m- and e-morphisms. The reason for this is that $\star$-pushouts in $\Var$ are given by pushouts of varieties; then, axiom (PO) holds since pushouts of closed immersions exist, and the resulting square is a pullback. We note that this does not hold for e-morphisms, as the pushout of open immersions need not exist. However, it does when the span of open immersions is known to belong to a pullback square, and thus $\star$-pushouts of both m- and e-spans satisfy axiom ($\star$). Axiom (PB) is satisfied as the squares in the double category are pullbacks. Finally, axiom (POL) can be verified in a similar manner as the previous examples.
\end{ex}

\begin{rmk}
Just as \cref{ex:exactcgw}, varieties give another example that fits our axioms, and not those of ACGW categories (although, unlike exact categories, varieties are pre-ACGW). In this case, this is due to the fact that our $\star$-pushouts need not exist in the case of e-morphisms, while $\star$-pushouts of both classes of morphisms are required in axiom (PP) of \cite[Definition 5.4]{CZ}.
\end{rmk}

\begin{ex}[Spaces]
In a similar manner we can define $\Top$ with objects topological spaces and m- and e-morphisms given by open and closed inclusions. The axioms hold for reasons analogous to the previous example, and in fact $\Var$ includes into $\Top$ as an ECGW subcategory.
\end{ex}

\subsection{$\fcgwa$ categories of functors}

A particularly useful way for us to construct new $\fcgwa$ categories from familiar ones is through functor categories. Given an $\fcgwa$ category $\C$ and any double category $\D$, we wish to describe an $\fcgwa$ structure on a double subcategory of the double category $[\D,\C]$ of double functors described in \cref{defn:dblcatoffunctors}.

\begin{defn}\lbl{functordblcat}
For $\C$ an $\fcgwa$ category and $\D$ any double category, we define the double subcategory $\C^\D \subset [\D,\C]$ as follows:
\begin{itemize}
    \item objects are all double functors $\D \to \C$
    \item $\M$ consists of the ``good'' m-natural transformations: these are the ones whose naturality squares of m-morphisms are good
    \item $\E$ is given by the ``good'' e-natural transformations: these are the ones whose naturality squares of e-morphisms are good
    \item squares consist of all modifications between the m- and e-morphisms; in particular, these are pointwise squares in $\C$.
\end{itemize}
Note that $\M$ and $\E$ here are categories, as good squares are closed under identities and composition and there are no restrictions placed on the mixed naturality squares of these transformations.
\end{defn}

As we saw in \cref{ex:exactcgw}, in order for (co)kernels to have the desired functoriality it is not enough to consider diagrams whose maps are all in $\M$, and instead we need to work with a more well-behaved notion of good square. Similarly, when working with m-natural transformations, it will not suffice to ask that all the squares involved are good, but instead we need a stronger notion of ``good cube''. In order to do this, we present the following definition, which adapts the good cubes of \cite[Definition 2.3]{stupidpaper} to our setting.

\begin{defn}\lbl{goodsquares}
Let $\C$ be an $\fcgwa$ category. A commutative cube of morphisms in $\M$ is a \textbf{good cube} if each face is a good square, and if the induced m-morphism between $\star$-pushouts\footnote{Such a morphism always exists; see \cref{lemma2.14}.} is such that the square below right is good.
    \[\begin{goodcubeinlinediagram}
    {A & & & A' & & \\
    & & B & & & B'\\
    & B\star_A C & & & B'\star_{A'} C' & \\
    C & & & C' & & \\
    & & D & & & D'\\};
    \mto{1-1}{4-1}  \mto{1-4}{4-4} \mto{2-6}{5-6}
    \mto{2-3}{3-2} \mto{4-1}{3-2} 
    \over{2-3}{2-6}
    \mto{1-1}{1-4} \mto{2-3}{2-6} \mto{4-1}{4-4} \mto{5-3}{5-6}
    \over{3-2}{3-5}
    \diagArrow{mmor,blue}{3-2}{3-5}
    \mto{1-1}{2-3} \mto{1-4}{2-6} \mto{4-1}{5-3} \mto{4-4}{5-6}
    \over{2-3}{5-3}
    \mto{2-3}{5-3}
    \mto{2-6}{3-5} \mto{4-4}{3-5} \mto{3-5}{5-6}
    \over{3-2}{5-3}
    \mto{3-2}{5-3}
    \end{goodcubeinlinediagram}\qquad
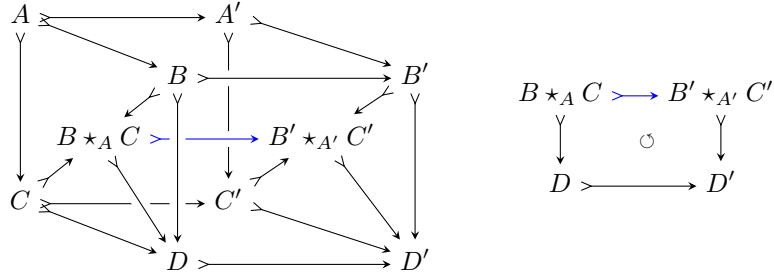
    \begin{inline-diagram}
    {B\star_A C & B'\star_{A'} C'\\
    D & D'\\};
    \mto{1-1}{2-1} \mto{1-2}{2-2}
    \mto{2-1}{2-2} \diagArrow{mmor,blue}{1-1}{1-2}
    \comm{1-1}{2-2}
    \end{inline-diagram}\]
    We call this the ``southern square''. Good cubes in $\E$ are defined in the same way.
\end{defn}

\begin{rmk}\lbl{goodcubessymmetric:claim}
A priori, it seems as if our definition of good cube is subject to a choice of direction. Indeed, we could have taken $\star$-pushouts of the back and front faces, instead of the left and right faces, and induced a different southern square. However, as explained in \cref{goodcubesymmetric}, if any of these induced squares are good, then all of them are. Moreover, it is possible to define a ``southern arrow'' m-morphism of the entire cube as in \cite[Definition 2.3]{stupidpaper} and show that any of the southern squares of a cube is good if and only if this southern arrow exists.
\end{rmk}

\begin{thm}\lbl{functorfcgw}
For $\C$ an $\fcgwa$ category and $\D$ any double category, the functor double category $\C^\D$ admits the structure of an $\fcgwa$ category as follows:
\begin{itemize}
    \item  $\Ar_{\g} \M$ are the commutative squares of m-natural transformations whose component cubes of naturality squares between m-morphisms are good cubes. $\Ar_{\g} \E$ is defined dually.
    \item The functors $k$ and $c$ are defined pointwise from those of $\C$, as is $\star$ in the sense that the $\star$-pushout of a span of $\D$-shaped diagrams in $\C$ is the $\D$-shaped diagram of pointwise $\star$-pushouts. Distinguished squares are also given pointwise.
\end{itemize}
\end{thm}

Showing that this defines an $\fcgwa$ structure is nontrivial, especially for $\star$-pushouts, but the axioms of $\fcgwa$ categories were designed to enable this kind of construction. As the technical details of this proof are not needed to describe our main results, we refer the reader to \cref{appendixpart2}, which also contains several helpful corollaries providing $\fcgwa$ structures on more specialized subcategories of $\C^\D$.

\subsection{Comparison with (A)CGW categories}

Since the purpose of $\fcgwa$ categories is to strengthen the axiomatic framework of the CGW  and ACGW categories of \cite{CZ} to allow for more examples and a better foundational behavior, it is important to compare these different structures and highlight which features have changed and which are still present from one framework to the other.

\subsubsection{$\fcgwa$ categories vs.\ CGW categories}\lbl{FCGWareCGW}\lbl{trivialextension}

Given any $\fcgwa$ category $\C$, we can always extract a CGW category by considering the double subcategory of $\C$ with the same objects, m-morphisms, and e-morphisms, and whose squares consist only of the distinguished squares of $\C$, and restricting the functors $c$ and $k$ to this subcategory. Axioms (Z), (M) and (K) of \cite[Definition 2.5]{CZ} are immediate, and axiom (I) follows from the properties of shared isomorphisms in \cref{sharedisos}. Finally, $\fcgwa$ categories also have a stronger, functorial version of axiom (A) in \cite{CZ}, which is intended to encode the existence of a trivial extension. Indeed, this can be recovered from our axioms by taking the $\star$-pushout of the span below 
    \begin{diagram}
    {A & \varnothing & B\\};
    \mto{1-2}{1-3} \mto{1-2}{1-1}
    \end{diagram}
and using our axiom (PO) together with the shared isomorphisms condition.

\subsubsection{$\fcgwa$ categories vs.\ ACGW categories}

Our $\fcgwa$ categories are very similar in nature to the ACGW categories of \cite[Definition 5.6]{CZ}. A first key distinction is that we replace the role played by pullbacks in ACGW categories by the more flexible notion of ``good square''. This change propagates to several distinctions between the axioms: we do not require m- or e-morphisms to be closed under pullback---this is axiom (P) of ACGW---and notably,  we do not require all pullback squares to  participate in the equivalences $k$ and $c$ as in ACGW, but rather consider the class of good squares---this results in the inclusion of exact categories as an example of $\fcgwa$, and not ACGW. 

A second key distinction is in our requirements of $\star$-pushouts. On the one hand, in terms of existence, these are more relaxed than axioms (S) and (PP) of \cite{CZ}, reducing the necessary $\star$-pushouts. In particular, unlike in ACGW we do not require $\star$-pushouts of e-morphisms to exist, which allows for the inclusion of varieties as $\fcgwa$, and not ACGW. We note that the extra functoriality properties asserted in axioms (S) and (PP) of ACGW can be obtained as consequences of ours when $\star$-pushouts exist; for details, see \cref{lemma2.15,southern:cokernel}. On the other hand, we have the additional axioms (POL) and (PBL) which need not be satisfied by an ACGW category. All these distinctions turn out to be crucial to obtain a spectrum in \cref{section:spectra} and for our proof of the Additivity Theorem in \cref{section:additivity}.

As a result of these differences, we do not get an ACGW category from an $\fcgwa$ category, or vice versa.

One can also highlight some less impactful differences. For 
instance, our description of the behavior of the equivalences $c$ and $k$ is more explicit, asking that $c$ acts on squares by 
\begin{diagram}
    {A & B  & & B & \cok i\\
    A' & B' & & B' & \cok i'\\};
    \mto{1-1}{1-2}^{i} \eto{1-5}{1-4}_{c(i)} \mto{2-1}{2-2}_{i'} \eto{2-5}{2-4}^{c(i')} \eto{1-1}{2-1}_{p} \eto{1-2}{2-2}^{p'} \eto{1-4}{2-4}_{p'} \eto{1-5}{2-5} \labar{1-2}{2-4}{\mapsto}  \labar{1-1}{2-2}{\circlearrowleft} \labar{1-4}{2-5}{\text{g}}
\end{diagram} and similarly for $k$. Here, not only do we ask that the codomain of the morphism $c(i)$ agrees with the codomain of $i$ (as in axiom (K) of (A)CGW), but also that the squares share their codomain vertical morphism $p'$. We believe this was an omission in \cite{CZ} as this fact is certainly used throught their paper; in fact, it was also omitted in our framework until the latest revision of this work.

\subsubsection{Theorems and constructions available}

A more important question in practice, beyond which axioms we do or do not have in each framework, is: which results translate from one framework to another? We now explain this by summarizing the results from \cite{CZ} and giving a preview of the results from our next sections.

\begin{itemize}
    \item \textbf{Q-construction.} As $\fcgwa$ categories are in particular CGW categories, all of \\ $\fcgwa$ , ACGW , and CGW categories admit a $Q$-construction defined in \cite[Definition 4.1]{CZ}. 
    \item \textbf{$S_\bullet$-construction.} All of these structures also admit an $S_\bullet$-construction, given in \cite[Definitions 7.2 and 7.12]{CZ} for ACGW categories and in \cref{defnSn} for $\fcgwa$ categories. Both constructions produce simplicial double categories with the same objects. The $S_\bullet$-construction of an ACGW category is pointwise a CGW category; hence, this can be iterated only once, and its realization produces a space. The $S_\bullet$-construction of an $\fcgwa$ category is pointwise $\fcgwa$, which allows us to iterate it and produce an infinite loop space; see \cref{section:spectra}.
    \item \textbf{D\'evissage.}
    The $Q$-construction for pre-ACGW categories satisfies a version of the D\'evissage theorem \cite[Theorem 6.2]{CZ}. This result will not hold in general for any $\fcgwa$ category (just as Quillen's D\'evissage holds for abelian categories but not exact categories). However, a detailed study of the proof reveals that this theorem holds for an $\fcgwa$ category in which $\M$ has pullbacks and where good squares in $\M$ are all the pullbacks. Indeed, the use of axiom (U) of pre-ACGW in the proof is now guaranteed by our equivalence of categories $c$, and axiom (S) is obtained by using \cref{southern:cokernel} twice. In particular, note that the asymmetry in the $\star$-pushouts between $\M$ and $\E$ in $\fcgwa$ categories is not an issue for this proof.  
\item \textbf{Additivity.} A central result is the additivity theorem, which describes the splitting behavior of $K$-theory on short exact sequences. In \cite[Proposition 7.15]{CZ}, they prove the additivity theorem for the $Q$-construction of a CGW category arising from a subtractive category, as introduced by Campbell \cite{Campbell}. In \cref{additivity0}, we prove the additivity theorem for the $S_\bullet$-construction of any relative $\fcgwa$ category. 
\item \textbf{Fibration.} $\fcgwa$ categories satsify a version of the Fibration theorem, which relates the $K$-theory of an $\fcgwa$ category equipped with two classes of weak equivalences; see \cref{fibrationthm}. Our proof relies on the Additivity theorem; hence it does not hold for ACGW categories in general. 
\item \textbf{Localization.} ACGW  and $\fcgwa$ categories each have their own version of the Localization theorem; see \cite[Theorem 8.6]{CZ} and \cref{localization}. These differ quite substantially: the version for $\fcgwa$ categories has minimal conditions, and produces a relative $\fcgwa$ category as a cofiber. On the other hand, the result for ACGW categories has several technical conditions, but produces an ACGW category as a cofiber (although the fact that the constructed cofiber is ACGW is one of the hypotheses that must be verified). 
\end{itemize}

\section{Adding weak equivalences}\label{section:fcgwa}


One of the benefits of Waldhausen's $S_\bullet$-construction is that it allows us to incorporate homotopical data in the form of weak equivalences. In practice, when a Waldhausen category has additional algebraic structure (such as that of an exact or abelian category), the weak equivalences often interact nicely with that structure. In particular, one often finds that the class of weak equivalences can be completely determined by the acyclic monomorphisms and epimorphisms, and that in turn, these can be characterized by having acyclic (co)kernels. Such is the case, for example, in the category of bounded chain complexes over an exact category, with quasi-isomorphisms as weak equivalences. 

In this section, we borrow this intuition and define \textit{m- and e-quivalences} on an $\fcgwa$ category, constructed from a given class of acyclic objects. 

\begin{defn}\lbl{fcgwa}
A class of \textbf{acyclic objects} in an $\fcgwa$ category $\C$ is a class of objects $\W$ of $\C$ such that:
\begin{enumerate}
    \item[(IA)]\label{initial_acyclic} any initial object is in $\W$
    \item[(A23)]\label{acyclic_2of3} for any kernel-cokernel pair $A \mrto B \elto C$ in $\C$, if any two of $A,B,C$ are in $\W$ then so is the third.
\end{enumerate}
\end{defn}

\begin{defn}\label{FCGWA_we}
Let $\C$ be an $\fcgwa$ category together with a class $\W$ of acyclic objects. An m-morphism (resp.\ e-morphism) in $\C$ is a \textbf{weak equivalence} if its cokernel (resp.\ kernel) is in $\W$. 

If we abuse notation and let $\W$ denote the full double subcategory of $\C$ on the class of acyclic objects, we refer to the pair $(\C,\W)$ as a \textbf{relative $\fcgwa$ category}.
\end{defn}

\begin{defn}
A \textbf{relative $\fcgwa$ functor} $(\C,\W) \to (\C',\W')$ between relative $\fcgwa$ categories is an $\fcgwa$ functor $\C \to \C'$ that preserves acyclic objects.
\end{defn}

\begin{notation}
Given a relative $\fcgwa$ category $(\C,\W)$, we will refer to the m-morphisms (resp.\ e-morphisms) which are weak equivalences as m-equivalences (resp.\ e-equivalences), and denote them by $\mrto^\sim$ (resp.\ $\erto^\sim$). When it is not relevant whether the weak equivalence is horizontal or vertical, we denote them by $\to^\sim$.
\end{notation} 

\begin{ex}
In any $\fcgwa$ category $\C$, acyclic objects can be chosen to be the initial objects. By \cref{isosemptycoker}, they satisfy \cref{acyclic_2of3} and weak equivalences are precisely the isomorphisms. The $K$-theory of this relative $\fcgwa$ category as defined in \cref{section:Kth} is the same as that of the underlying CGW category of $\C$ defined in \cite{CZ} (for more details, see \cref{ktheorycgw}). 
\end{ex}

\begin{ex}\label{intersection_FCGWA}
For any relative $\fcgwa$ category $(\C,\W)$ and $\C' \subset \C$ an $\fcgwa$ subcategory, we have that $(\C',\W \cap \C')$ forms a relative $\fcgwa$ category. 
\end{ex}

\begin{ex}
As explained in \cref{ex:exactcgw}, weakly idempotent complete exact categories can be given the structure of an $\fcgwa$ category. Let $\C$ be such a category, which in addition has a Waldhausen structure. If we denote by $\W$ the class of objects $X\in\C$ such that $0\to X$ is a weak equivalence, then $(\C, \W)$ will be a relative $\fcgwa$ category whenever $\W$ has 2-out-of-3.

For instance, this will be the case when $\C$ is a Waldhausen category constructed from a cotorsion pair and any such class $\W$ of acyclic objects as in \cite{cotorsion}, when $\C$ is a biWaldhausen category satisfying the extension and saturation axioms (such as the complicial biWaldhausen categories of \cite[1.2.11]{TT}), and when $\C$ satisfies the saturation axiom and is both left and right proper (like the complicial exact categories with weak equivalences of \cite[Definition 3.2.9]{Sch11}). In particular, our construction recovers the classical (epi and mono) quasi-isomorphisms for the case of chain complexes on an exact category, where acyclic objects are given by the exact complexes. 
\end{ex}

\begin{ex}\label{modn}
    Recall that an $\fcgwa$ category $\C$ has a sum operation where $A + B$ is given by the $\star$-pushout of the span $A \mlto \varnothing \mrto B$. In an extensive category this is the coproduct and in an exact category this is the direct sum. We can then define the object $nA$ as the $n$-fold sum of an object $A$ with itself, and the full double subcategory $n\C$ containing all objects of the form $nA$.

    The initial object is always in $n\C$, and in many examples of $\C$ such as finite sets and free $R$-modules, $n\C$ will be closed under kernels, cokernels, and extensions in the ambient ECGW category $\C$. This can often be shown using an invariant such as cardinality (for finite sets) or dimension (for free $R$-modules) with respect to which extensions are additive. In this case $(\C,n\C)$ will form a relative $\fcgwa$ category.
\end{ex}

\begin{ex}\label{relativegsets}
    The objects of any full extensive subcategory $\Y$ of an extensive category $\X$ closed under complements forms a class of acyclic objects. We will denote this relative $\fcgwa$ category by $(\X,\Y)$. This condition is weaker than the Serre condition as it only requires $B$ to be in $\Y$ if both $A$ and $A \sqcup B$ are, rather than merely $A \sqcup B$. The complementary subcategory $\X - \Y$, however, will not be closed under coproducts unless $\Y$ is Serre.
\end{ex}

\begin{ex}\label{relativepolytopes}
    In the $\fcgwa$ category $\fGh_n$ of open heterogeneous $n$-dimensional polytopes from \cref{polytopes}, the objects in the full double subcategory $\fGh_{n-1}$ form a class of acyclic objects, as unions of simplices of dimension at most $n-1$ contain the empty set and are closed under complements and disjoint unions. In the resulting relative $\fcgwa$ category, a weak equivalence is a piecewise inclusion of polytopes whose complement is at most $(n-1)$-dimensional. 

    This relative $\fcgwa$ category is closely related to the assembler $\fG_n \backslash \fG_{n-1}$ of \cite[Definition 2.9]{assemblers}, and in fact under suitable conditions any assembler whose underlying category is a subcategory of an extensive category gives rise to a class of acyclic objects under a similar construction. In \cite[Proposition 5.5]{assemblers} Zakharevich shows that the $K$-theory of $\fG_n \backslash \fG_{n-1}$ agrees with that of $\cG_n$, a result we will also prove using abstract principles of relative $\fcgwa$ categories rather than arguments involving ``covering families'' as in the original proof (see \cref{polytopecompare,polytopecofiber}).
\end{ex}

As we now show, the properties of weak equivalences are easily expressed in terms of their defining acyclic objects. This is reminiscent of the construction of Waldhausen structures on exact categories via cotorsion pairs of \cite{cotorsion}. Much of the theory we develop holds equally well in a more general setting in which weak equivalences are not determined by acyclic objects, but this complicates the proofs significantly and we do not pursue this here.

The following results can be easily deduced for any relative $\fcgwa$ category from \cref{fcgwa}.

\begin{lemma}\label{contains_isos}
All isomorphisms are weak equivalences.
\end{lemma}


\begin{lemma}\label{we_acyclic} 
Given a weak equivalence $X\to^\sim Y$, if either $X$ or $Y$ is acyclic, then both are.
\end{lemma}


\begin{lemma}\lbl{acyclic:we}
Any map between acyclic objects is a weak equivalence.
\end{lemma}


In particular, all morphisms in the full double subcategory $\W$ are weak equivalences, and an object in $\C$ is acyclic if and only if both the m- and e-morphisms from $\varnothing$ are weak equivalences.  

Additionally, we can prove the following.

\begin{lemma}\label{we_2of3}
m- and e-equivalences each satisfy 2-out-of-3. In particular, they form subcategories of $\M$ and $\E$.
\end{lemma}

\begin{proof}
We prove this for m-morphisms, the argument for e-morphisms is dual.

Given m-morphisms $f : A \mrto B$ and $g : B \mrto C$, consider the following diagram
    \begin{diagram}
    {\cok f & \cok gf & D\\
    B & C & \cok g\\
    A & A\\};
    \mto{1-1}{1-2} \eto{1-3}{1-2}
    \mto{2-1}{2-2}_g \eto{2-3}{2-2}
    \eq{3-1}{3-2}
    \mto{3-1}{2-1}^f \mto{3-2}{2-2}_{gf}
    \eto{1-1}{2-1} \eto{1-2}{2-2} \eto{1-3}{2-3}^\cong
    \dist{1-1}{2-2} 
    \end{diagram}
By \cref{we_acyclic}, $D$ is acyclic if and only if $\cok g$ is, so if any two of $f,g,gf$ are weak equivalences, then two of $\cok f, \cok g, \cok gf$ are acyclic, and hence so is the third by \cref{acyclic_2of3}. Together with \cref{contains_isos}, this shows that weak equivalences form subcategories of $\M$ and $\E$.
\end{proof}



\begin{lemma}\lbl{parallel:2of3}
In a kernel-cokernel pair of squares, if any two of the three parallel maps are weak equivalences then so is the third. 
\end{lemma}

\begin{proof}
Consider the kernel-cokernel pair of squares depicted in the left column of the diagram below, with parallel m-morphisms $f,g,h$
    \begin{diagram}
    {A & B & \cok f\\
    C & D & \cok g\\
    E & F & \cok h\\};
    \mto{1-1}{1-2}^f \eto{1-3}{1-2}
    \mto{2-1}{2-2}^g \eto{2-3}{2-2}
    \mto{3-1}{3-2}^h \eto{3-3}{3-2}
    \eto{1-1}{2-1} \eto{1-2}{2-2} \eto{1-3}{2-3}
    \mto{3-1}{2-1} \mto{3-2}{2-2} \mto{3-3}{2-3}
    \comm{1-1}{2-2} \good{1-2}{2-3}
    \comm{2-2}{3-3} \good{2-1}{3-2}
    \end{diagram}
Taking cokernels of both squares, we get a kernel-cokernel sequence $$\cok f \erto \cok g \mlto \cok h$$ as shown in the diagram, so by \cref{acyclic_2of3} if any two of $f,g,h$ are weak equivalences then so is the third.
\end{proof}

\begin{lemma}\lbl{acyclic:pushout}
Acyclic objects are closed under $\star$-pushouts (when these exist).
\end{lemma}

\begin{proof}
Consider a span of m-morphisms $B \mlto A \mrto C$ where $A,B,C$ are acyclic. By \cref{acyclic:we} these morphisms are weak equivalences, hence $B/ A$ is acyclic. By axiom ($\star$), $(B \star_A C)/ C \cong B / A$, so the map $C \mrto B \star_A C$ is a weak equivalence. Therefore, $B \star_A C$ is acyclic by \cref{acyclic:we}. The same argument holds for spans of e-morphisms whose $\star$-pushout exists.
\end{proof}

\begin{rmk}\lbl{rmk:acyclicisfcgw}
\cref{acyclic_2of3}, along with \cref{acyclic:pushout}, imply that any class of acyclic objects $\W$ forms an $\fcgwa$ category by \cref{fullsubFCGW}. Conversely, given an $\fcgwa$ category $\C$, any full $\fcgwa$ subcategory that is closed under extensions provides a class of acyclic objects.
\end{rmk} 

\begin{defn}
An $\fcgwa$ subcategory $\C'$ of $\C$ is \textbf{closed under extensions} if, for any kernel-cokernel sequence $$A\mrto B \elto C$$ in $\C$ such that $A,C$ are in $\C'$, $B$ is also in $\C'$.
\end{defn}

An $\fcgwa$ category often admits more than one natural choice of acyclic objects. In fact, \cref{section:fibration} provides a tool for comparing the two resulting $\fcgwa$ categories with weak equivalences when one is a refinement of the other.

\begin{defn}
A \textbf{refinement} of a relative $\fcgwa$ category $(\C,\W)$ is a subclass $\V \subseteq \W$ such that $\V$ is also a class of acyclic objects.
\end{defn}

\begin{ex}
The poset of refinements of $(\C,\W)$ ordered by inclusion has both minimal and maximal elements, given by initial objects in $\C$ and $\W$ itself, respectively.
\end{ex}

The following is immediate from our definitions, along with \cref{rmk:acyclicisfcgw}.

\begin{lemma}\label{W_fcgw}
For any refinement $(\C,\V)$ of a relative $\fcgwa$ category $(\C,\W)$, the pair $(\W,\V)$ is also a relative $\fcgwa$ category. 
\end{lemma}


\section{The $S_\bullet$-construction}\lbl{section:Kth}

We are now equipped to define the $K$-theory of an $\fcgwa$ category by translating the $S_\bullet$-construction into our setting. The construction is similar to that of \cite[Definition 7.10]{CZ}, but we also accommodate weak equivalences, and more importantly, the variants in our construction allow for this process to be iterated. In other words, given an $\fcgwa$ category $\C$, we construct a simplicial double category $S_\bullet \C$ which is furthermore a simplicial $\fcgwa$ category.

The following double category will be useful for defining our $S_\bullet$-construction.

\begin{defn}\label{defn:sn}
For each $n$, let $\mathcal{S}_n$ denote the double category generated by the following objects, horizontal morphisms, vertical morphisms, and squares.
    \begin{general-diagram}{1.4em}{1em}
    {A_{0,0} & A_{0,1} & A_{0,2} & \cdots & A_{0,n}\\
    & A_{1,1} & A_{1,2} & \cdots  & A_{1,n}\\
    & &   A_{2,2} & \cdots  & A_{2,n}\\
    & &  & \ddots &\vdots\\
    & &  &   & A_{n,n}\\};
    \mto{1-1}{1-2} \mto{1-2}{1-3} \mto{1-3}{1-4} \mto{1-4}{1-5}
     \mto{2-2}{2-3} \mto{2-3}{2-4} \mto{2-4}{2-5}
     \mto{3-3}{3-4} \mto{3-4}{3-5}
     \eto{2-2}{1-2} \eto{2-3}{1-3} \eto{2-5}{1-5} 
     \eto{3-3}{2-3} \eto{3-5}{2-5} 
    \eto{4-5}{3-5} \eto{5-5}{4-5}
    \diagArrow{-,white}{1-2}{2-3}!{\textcolor{black}{\circlearrowleft}}
     \diagArrow{-,white}{1-3}{2-4}!{\textcolor{black}{\circlearrowleft}}
      \diagArrow{-,white}{2-3}{3-4}!{\textcolor{black}{\circlearrowleft}}
       \diagArrow{-,white}{1-5}{2-4}!{\textcolor{black}{\circlearrowleft}}
        \diagArrow{-,white}{2-5}{3-4}!{\textcolor{black}{\circlearrowleft}}
    \end{general-diagram}
\end{defn}

\begin{rmk}
    If $\mathrm{Ar}([n])$ denotes the arrow category of the poset $[n]$, used by Waldhausen in \cite[\S 1.3]{Waldhausen} to define the $S_\bullet$-construction of a Waldhausen category, then note that the double category $\mathcal{S}_n$ of \cref{defn:sn} above is simply the flat double category of commutative squares of $\mathrm{Ar}([n])$. This can be obtained, for instance, as the double category of quintets $Q\mathrm{Ar}([n])$ where we view $\mathrm{Ar}([n])$ as a 2-category with trivial 2-cells; see \cite[\S 3.1.4]{Grandis}.
\end{rmk}

\begin{defn}\lbl{defnSn}
Given an $\fcgwa$ category $\C$, we define a simplicial double category $S_\bullet \C$ as follows:
\begin{itemize}
    \item for each $n$, $S_n\C$ is the full double subcategory of $\C^{\mathcal{S}_n}$ given by the functors $F$ such that $F(A_{i,i})=\varnothing$ for all $i$, and that $F$ sends all squares in $\mathcal{S}_n$ to distinguished squares in $\C$.
    \item for the simplicial structure, the face map $d_i : S_n\C \to S_{n-1}\C$, $0 \leq i \leq n$, deletes the objects $F(A_{j,i})$ and $F(A_{i,j})$ for all $j$, where what remains after discarding or composing the affected squares is a diagram of shape $\mathcal{S}_{n-1}$; the degeneracy map $s_i : S_n\C \to S_{n+1}\C$ inserts a row and column of identity morphisms above and to the right of $F(A_{i,i})$
\end{itemize}
We will often refer to the objects of $S_n \C$ as ``staircases''.
\end{defn}

\begin{thm}\lbl{statementSnfcgwa}
$S_n \C$ is a relative $\fcgwa$ category, with $\fcgwa$ structure inherited from that of $\C^{\mathcal{S}_n}$ as described in \cref{functorfcgw}, and acyclic objects defined as the pointwise acyclics in $\C$.
\end{thm}

\begin{proof}
\cref{Snfcgw} shows that $S_n \C$ is a $\fcgwa$ subcategory of $\C^{\mathcal{S}_n}$, and pointwise acyclic diagrams clearly form a class of acyclic objects.
\end{proof}

\begin{defn}
 For  a relative $\fcgwa$ category $(\C,\W)$, define $$K(\C,\W) = \Omega|wS_\bullet \C| \ \ \text{ and } \ \ K_n(\C,\W) = \pi_n K(\C,\W),$$ where $wS_\bullet\C$ is the simplicial double category obtained by restricting the m-morphisms and e-morphisms in $S_\bullet\C$ to the m-equivalences and e-equivalences.
 \end{defn}
 
As usual, we start by studying $K_0$ and showing that it agrees with the intuitive Grothendieck group. Similarly to \cite[Theorem 4.3]{CZ}, most of the relations will be given by the distinguished squares, except that we get additional relations induced by the weak equivalences.
 
 \begin{prop}\label{kzero}
 For any relative $\fcgwa$ category $(\C, \W)$, $K_0(\C,\W)$ is the free abelian group generated by the objects of $\C$, modulo the relations that, for any distinguished square 
 \[\dsq{A}{B}{C}{D}{}{}{}{}\] we have $[A]+[D]=[B]+[C]$, and that for any horizontal or vertical weak equivalence $A\xrightarrow{\sim} B$ we have $[A]=[B]$.
 \end{prop}
 \begin{proof}
 By definition, $K_0(\C,\W)=\pi_0\Omega|wS_\bullet \C|=\pi_1|wS_\bullet \C|$. Since $|wS_\bullet \C|$ is path-connected (as $|wS_0 \C|=\ast$), it follows from the Van-Kampen Theorem that $\pi_1|wS_\bullet \C|$ is the free group on $\pi_0|wS_1 \C|$, modulo the relations $\delta_1(x)=\delta_2(x)\delta_0(x)$ for each $x\in\pi_0|wS_2 \C|$.
  
  Let us describe what these conditions entail. The points of $|wS_1 \C|$ are the objects of $\C$, and  connected components are determined by zig-zags of squares as below left, which we think of as a square from $A$ to $B$.
  
    \begin{diagram}
    {A & \bullet & & A & A\\
    \bullet & B & & B & B\\};
    \wmto{1-1}{1-2} \wmto{2-1}{2-2}
    \weto{1-1}{2-1} \weto{1-2}{2-2}
    \comm{1-1}{2-2}
    \eq{1-4}{1-5} \eq{2-4}{2-5}
     \weto{1-4}{2-4} \weto{1-5}{2-5}
     \comm{1-4}{2-5}
    \end{diagram}
     
 These squares include those above right, which shows that weakly equivalent objects are identified. Conversely, the relation that weakly equivalent objects are identified implies that objects $A$ and $B$ in the same connected component of $|wS_1 \C|$ are identified, so these two relations are equivalent.
     
 Points of $|wS_2 \C|$ are kernel-cokernel sequences in $\C$ represented by distinguished squares $x$ as below left, where $\delta_1(x)=B$, $\delta_0(x)=B/A$ and $\delta_2(x)=A$.
 \[\dsq{\varnothing}{B/A}{A}{B}{}{}{}{}\qquad\qquad\dsq{A}{B}{C}{D}{}{}{}{}\]
 To show that our distinguished square relation is always satisfied in $K_0(\C,\W)$, we recall that distinguished squares induce isomorphisms on cokernels. We can then see that for any distinguished square as above right we have
 \begin{align*}
     [B] & = [A]+[B/A]\\
         & = [A]+[D/C]\\
         & = [A]+[D]-[C]
 \end{align*} which yields the desired relation. Conversely, if we start from the relation on generic squares, we obtain the splitting $[B] = [A] + [B/A]$ by restricting to the points of $|wS_2 \C|$, so that relation is equivalent to ours.
 
 Finally, note that $K_0(\C)$ is abelian because, as explained in \cref{trivialextension}, we have trivial extensions $$A\mrto A+B \elto B, \qquad B\mrto A+B \elto A$$ and so $[A] + [B]=[A+B]=[B] + [A]$. 
 \end{proof}
 
Having established a new $K$-theory machinery, we now wish to show that it agrees with the existing ones for all the relevant examples. We start by stating the following, analogous to \cite[1.4.1 Corollary (2)]{Waldhausen}.

\begin{defn}
Given an $\fcgwa$ category $\C$, let $s_\bullet\C$ denote the simplicial set given by $s_n\C=\ob S_n\C$. 
\end{defn}
 
\begin{lemma}\lbl{iS:equals:s}
For an $\fcgwa$ category $\C$, we have $iS_\bullet \C \simeq s_\bullet \C$, where $i$ denotes the class of isomorphisms in $\C$.
\end{lemma}
 
\begin{proof}
Since $\C$ has shared isomorphisms, as does each $S_n\C$ by \cref{statementSnfcgwa}, the double subcategory $iS_n\C$ is isomorphic to the double category of commutative squares in the groupoid $I(S_n\C)$ of isomorphisms in $S_n\C$. By Waldhausen's Swallowing Lemma (\cite[1.5.6]{Waldhausen}), $iS_n\C$ is then homotopy equivalent to the groupoid $I(S_n\C)$ itself, and from this point the proof proceeds exactly as in \cite[1.4.1]{Waldhausen}.
\end{proof}

Using this lemma, we see that the $K$-theory of an $\fcgwa$ category with isomorphisms as weak equivalences agrees with its $K$-theory as constructed in \cite{CZ}. 

\begin{prop}\lbl{ktheorycgw}
For an $\fcgwa$ category $\C$, $K(\C,\varnothing)$ agrees with the $K$-theory of its underlying CGW category as defined in \cite{CZ}. 
\end{prop}
\begin{proof}
By \cref{iS:equals:s}, $K(\C,\varnothing)$ is homotopy equivalent to $\Omega|s_\bullet \C|$, which is precisely $K^S$ of the underlying CGW category of $\C$ as defined in \cite[Definition 7.4]{CZ}.
\end{proof}

\begin{rmk}\lbl{Kthagree}
In particular, this implies that the $K$-theory of the $\fcgwa$ categories given by exact categories, finite sets, and varieties of \crefrange{ex:exactcgw}{ex:varieties} agree with their existing counterparts in the literature. 
\end{rmk}

\begin{rmk}\lbl{idempotentcompletion}
The only caveat if one wishes to model the $K$-theory of exact categories through our formalism is that, as explained in \cref{ex:exactcgw}, they need to be weakly idempotent complete. However, this does not present a real obstruction, for two reasons. First, any exact category $\C$ satisfies 
$K(\C)\simeq K(\bar{\C})$, where $\bar{\C}$ denotes the full exact subcategory of the idempotent completion of $\C$ consisting of the objects $A$ such that $[A]\in K_0(\C)$; in particular, $\bar{\C}$ is weakly idempotent complete. And second, a detailed study of \cref{ex:exactcgw} reveals that the only axiom that potentially fails when $\C$ is not weakly idempotent complete is axiom (PBL), which is not needed to \emph{construct} the $K$-theory space of an $\fcgwa$ category, but rather to prove the foundational properties we will focus on in later sections.
\end{rmk} 

Our definition also agrees with the existing notion of $K$-theory for an extensive category $\X$, namely its symmetric monoidal $K$-theory with the coproduct monoidal structure. Recall that the $K$-theory of a symmetric monoidal category is given by the group completion of the realization of its underlying monoidal groupoid, which is a topological monoid.

\begin{prop}\label{extensiveKtheory}
For an extensive category $\X$, $K(\X,\varnothing)$ is homotopy equivalent to the symmetric monoidal $K$-theory of $(\X,\sqcup)$.
\end{prop}

\begin{proof}
Unwinding the definition of $S_n\X$, we find that its objects agree precisely with those of the category $N_n\X$ in Waldhausen's description of the $K$-theory of a category with finite coproducts at the beginning of \cite[Section 1.8]{Waldhausen}: a choice of objects $A_1,...,A_n$ and ``appropriate sum diagrams'' of exactly the form given by our staircase diagram of coproduct injections. By Waldhausen's Swallowing Lemma (\cite[1.5.6]{Waldhausen}), our double category $iS_n\X$ will be homotopy equivalent to the category $iN_n\X$. This description agrees with Segal's construction of $K$-theory using $\Gamma$-categories, and hence all of the equivalent models for symmetric monoidal $K$-theory. 

\end{proof}

\begin{rmk}
    Note that Waldhausen in \cite[Section 1.8]{Waldhausen} describes this construction for categories with sums and weak equivalences, and this proof applies just as well to our setting with weak equivalences, where the weak equivalences in an extensive category are taken to be the subcategory of coproduct injections whose complements are acyclic.
\end{rmk}

In the case of polytopes, Zakharevich's $K$-theory of the assemblers $\cG_n$ and $\fG_n$ is defined in \cite[Definition 2.12]{assemblers} and equivalently in \cite[Definition 1.7, Theorem 2.1]{k1assemblers}.

\begin{prop}
    The $K$-theories $K(\cGh_n,\varnothing)$ and $K(\fGh_n,\varnothing)$ of $\fcgwa$ categories are equivalent to the $K$-theories of the assemblers $\cG_n$ and $\fG_n$.
\end{prop}

The proof of this equivalence is somewhat involved, so we defer it to \cref{appendixpolytopes}.

\begin{rmk}
It is natural to ask whether our notion of $K$-theory also agrees with the existing ones when working with an exact category with weak equivalences, such as chain complexes with quasi-isomorphisms. Due to the way it was constructed, our $K$-theory machinery is only designed to take as input a category whose weak equivalences are defined through a class of acyclics. That is, if there is any hope of a comparison, the exact category must be such that an admissible monomorphism (resp.\ epimorphism) is a weak equivalence if and only if its cokernel (resp.\ kernel) is weakly equivalent to $0$. 

Furthermore, since our double-categorical perspective only deals with admissible monomorphisms and epimorphisms, it must be the case that m- and e-equivalences encode the data of all weak equivalences. This is the case, for example, when weak equivalences can be expressed as composites of admissible monomorphisms and epimorphisms which are themselves weak equivalences.

Fortunately, this seems to be the case for the vast majority of exact categories with weak equivalences that arise in practice. A comparison with the existing $K$-theory constructions in these settings will crucially require several of the tools that we develop later in the paper, and for that reason we delay it to \cref{prop:comparisonexact}. 
\end{rmk}

Finally, we provide a simple tool for comparing the $K$-theories of two relative $\fcgwa$ categories under favorable circumstances, such as two equivalent relative $\fcgwa$ categories.

\begin{lemma}\label{kcomparison}
A pair of relative $\fcgwa$ functors $F : (\C,\W) \leftrightarrows (\D,\V) : G$ equipped with cartesian natural weak equivalences between $1_\C$ and $GF$, and between $FG$ and $1_\D$ in any combination of types and directions, induces a homotopy equivalence $K(\C,\W) \simeq K(\D,\V)$.
\end{lemma}

\begin{proof}
This follows from applying \cref{equivalent_spaces} to the induced pair of double functors $wS_n\C \leftrightarrows vS_n\D$, whose existence and associated natural transformations rely on $F,G$ being relative $\fcgwa$ functors and their associated natural transformations being cartesian and componentwise weak equivalences.
\end{proof}

\begin{ex}\label{polytopecompare}
Consider the relative $\fcgwa$ categories $(\fGh_n,\fGh_{n-1})$ from \cref{relativepolytopes} and $(\cGh_n,\varnothing)$ from \cref{polytopes}. There are $\fcgwa$ functors $\mathsf{clint} : \fGh_n \leftrightarrows \cGh_n : \mathsf{int}$ where $\mathsf{clint}$ sends an open heterogeneous $n$-polytope to the closure of its interior in $n$-dimensional space, and $\mathsf{int}$ sends a closed homogeneous $n$-polytope to its interior in $n$-dimensional space. Both are relative $\fcgwa$ functors with respect to the given choices of acyclic objects, trivially for $\mathsf{int}$ and for $\mathsf{clint}$ because the interior of a less-than-$n$-dimensional polytope in $n$-dimensional space is empty. The composite $\mathsf{clint}\circ\mathsf{int}$ is isomorphic to the identity on $\cGh_n$, while the composite $\mathsf{int} \circ \mathsf{clint}$ sends a polytope in $\fGh_n$ to its interior. This latter composite relative $\fcgwa$ functor is then naturally weakly equivalent to the identity, as there is a natural inclusion of the interior of a polytope to the original whose complement is less-than-$n$-dimensional. Therefore we have $K(\fGh_n,\fGh_{n-1}) \simeq K(\cGh_n)$. In \cref{polytopecofiber}, we conclude from this that $K(\cGh_n)$ is the homotopy cofiber of the map $K(\fGh_{n-1}) \to K(\cGh_n)$.
\end{ex}

\section{The Additivity Theorem}\label{section:additivity}

The purpose of this section is to show that our $K$-theory construction satisfies the Additivity Theorem. Aside from being a fundamental result that any $K$-theory machinery is expected to satisfy, it will be useful in the next section when we establish our version of the Fibration Theorem.

In order to state the Additivity Theorem, we define extension categories in our setting.

\begin{defn}\label{defnE}
 Let $\A,\B\subseteq\C$ be full $\fcgwa$ subcategories of an $\fcgwa$ category $\C$. We define the \textbf{extension double category} $E(\A,\C,\B)$ as the full double category of $S_2 (\C)$ whose objects are determined by kernel-cokernel sequences in $\C$ of the form $$A\mrto^f C\elto^g B$$ with $A\in \A$, $B\in\B$ and $C\in\C$. Explicitly, an m-morphism in $E(\A,\C,\B)$ is a triple of pointwise m-morphisms in $\A,\C,\B$ respectively, related by good squares and squares in the double category $\C$ as follows
     \begin{diagram}
     {A & C & B\\
     A' & C' & B'\\};
     \mto{1-1}{1-2}^f \mto{2-1}{2-2}_{f'}
     \eto{1-3}{1-2}_g \eto{2-3}{2-2}^{g'}
     \mto{1-1}{2-1}_{h_A} \mto{1-2}{2-2}^{h_C} \mto{1-3}{2-3}^{h_B}
     \good{1-1}{2-2} \comm{1-2}{2-3}
     \end{diagram} and e-morphisms are defined analogously. Squares in $E(\A,\C,\B)$ are given by triples of  squares in $\A,\C,\B$ respectively, natural in the appropriate sense.  
 \end{defn}
 
Interestingly, any commuting square as above left is automatically good by an argument that recurs throughout the rest of the paper.

\begin{lemma}\label{goodkernelsquare}
For any two extensions related by a  commuting square and a square in the double category as above, the commuting square is always the (co)kernel of the square in the double category, and thus a good square. This holds when the vertical maps are either m- or e-morphisms.
\end{lemma}

\begin{proof}
As the top and bottom row in the diagram are kernel-cokernel pairs and the square on the right is a square in the double category, there exists a good square in $\M$ which agrees with the left square everywhere except possibly $h_A$. However, as both squares commute and $f'$ is a monomorphism, the remaining map in the good square must indeed be $h_A$, so the square agrees with the good square kernel of the right square. The argument for e-morphisms is entirely dual.
\end{proof}
 
 \begin{lemma}\label{extensionsFCGWA}
 $E(\A,\C,\B)$ is an $\fcgwa$ category with the structure inherited from $S_2 (\C)$ of \cref{statementSnfcgwa}, which is moreover a relative $\fcgwa$ category whenever $\C$ is.
 \end{lemma}
 \begin{proof}
 When $\A=\B=\C$, we have that $E(\C,\C,\C)=S_2 (\C)$ and the result is shown in \cref{statementSnfcgwa}. It is then straightforward to check that $E(\A,\C,\B) \subseteq E(\C,\C,\C)$ is an $\fcgwa$ subcategory by \cref{fullsubFCGW}, as $\A,\B$ are $\fcgwa$ subcategories, and that pointwise acyclic objects form a class of acyclic objects in $E(\A,\C,\B)$.
 \end{proof}

We now return to the goal of this section: to prove the Additivity Theorem stated below.

 \begin{thm}[Additivity]\label{additivity0}
Let $\C$ be a relative $\fcgwa$ category. Then, the map
$$wS_\bullet E(\C,\C,\C)\to wS_\bullet \C\times wS_\bullet\C$$ induced by $$(A\mrto C\elto B)\mapsto (A,B)$$ is a homotopy equivalence.
 \end{thm}
 
The proof of Additivity proceeds in a manner almost identical to McCarthy's \cite{McCarthy}. Just as in \cite[Theorem 1.4.2]{Waldhausen}, the first step is to reduce the proof of Additivity to the case when the equivalences considered are isomorphisms. In the classical case, this is done by showing that the bisimplicial set $(m,n)\mapsto s_n \C(m, w)$  is equivalent to the bisimplicial set $(m,n)\mapsto w_m S_n \C$, or, in other words, that staircases of sequences of weak equivalences in $\C$ are the same as sequences of weak equivalences of staircases in $\C$. We now introduce the double categorical version of this statement.

\begin{defn}\lbl{defngrids}
Let $(\C,\W)$ be a relative $\fcgwa$ category, and let $\D$ denote the free double category on an $l \times m$ grid of squares. The \textbf{double category of w-grids} $w_{l,m}\C$ is the full double subcategory of $\C^\D$ consisting of the grids whose morphisms are all weak equivalences. 
\end{defn}

\begin{prop}\lbl{gridsfcgw:statement}
Let $(\C,\W)$ be a relative $\fcgwa$ category. Then $w_{l,m}\C$ is an $\fcgwa$ category with structure inherited from that of $\C^\D$ in \cref{functorfcgw}.  Moreover, if $\V$ a refinement of $\W$, then the double subcategory of grids in $\V$ forms a class of acyclic objects on $w_{l,m}\C$.
\end{prop}

The proof of this proposition can be found in \cref{gridsfcgw}. With this structure in hand, we can see the following.

\begin{lemma}\label{grid_commute}
There is an isomorphism of simplicial sets 
 $$s_\bullet w_{l,m} \C\cong w_{l,m} S_\bullet \C ,$$ 
 simplicial in both $l$ and $m$. More generally, for any refinement $\V \subseteq \W$, 
 $$vS_\bullet w_{l,m} \C \cong vw_{l,m} S_\bullet \C.$$
 \end{lemma}
 
\begin{proof}
This follows immediately from the definitions, and it amounts to saying that staircases of w-grids in $\C$ are the same as w-grids of staircases in $\C$.
\end{proof}

Like in the classical case, this allows us to show that weak equivalences are not an integral part of the Additivity Theorem.

\begin{prop}\lbl{lemmaAdditivity}
If the map $$s_\bullet E(\A,\A,\A)\to s_\bullet\A\times s_\bullet\A$$ is a homotopy equivalence for every $\fcgwa$ category $\A$, then the map $$wS_\bullet E(\C,\C,\C)\to wS_\bullet\C\times wS_\bullet\C$$ is a homotopy equivalence for every relative $\fcgwa$ category $(\C,\W)$.
\end{prop}
\begin{proof}
Let $(\C,\W)$ be a relative $\fcgwa$ category, and consider the $\fcgwa$ category of w-grids $w_{l,m}\C$ of \cref{gridsfcgw:statement}. Note that for each $l,m,n$, we have by \cref{grid_commute} an isomorphism $$ s_n w_{l,m}\C\cong w_{l,m} S_n \C.$$ 
Moreover, there is a homotopy equivalence 
$$s_\bullet w_{l,m}E(\C,\C,\C) \simeq s_\bullet E(w_{l,m}\C, w_{l,m}\C,w_{l,m}\C)$$ 
for each $l,m$ arising via \cref{equivalent_spaces} from the evident equivalence of double categories. Applying the assumption of the lemma to each $\A=w_{l,m}\C$ gives homotopy equivalences of simplicial sets
$$w_{l,m}S_\bullet E(\C,\C,\C) \simeq s_\bullet w_{l,m} E(\C,\C,\C) \simeq s_\bullet E(w_{l,m}\C,w_{l,m}\C,w_{l,m}\C)$$ $$\to s_\bullet w_{l,m}\C \times s_\bullet w_{l,m}\C \simeq w_{l,m}S_\bullet\C \times w_{l,m}S_\bullet\C$$
which assemble into a levelwise homotopy equivalence of trisimplicial sets, and thus a homotopy equivalence $$wS_\bullet E(\C,\C,\C)\to wS_\bullet\C\times wS_\bullet\C.$$
\end{proof}

We are now ready to prove Additivity. Our proof is nearly identical to \cite[Section 4]{Campbell}, which in turn follows McCarthy \cite{McCarthy}; we briefly sketch the outline of the proof, focusing only on the parts that do not trivially translate to our setting.

\begin{proof}[Proof of \cref{additivity0}]
By \cref{lemmaAdditivity}, it suffices to show that Additivity holds for $\fcgwa$ categories (with isomorphisms as weak equivalences). We start with a brief recollection of McCarthy's proof strategy for the Additivity theorem \cite{McCarthy}, adapted to our setting. Let $F\colon E(\C,\C,\C)\to \C\times\C$ denote the additivity functor; we wish to prove that $$S_\bullet F\colon S_\bullet E(\C,\C,\C)\to S_\bullet \C\times S_\bullet \C$$ is a homotopy equivalence. Consider the bisimplicial set $S_\cdot F|\C^2$ whose set of $(m,n)$-simplices is given by the collection of all diagrams of the form 
\begin{toolongdiagram}
    {\varnothing=A_0 & A_1 & \dots & A_m & &  & &\\
    \varnothing=C_0 & C_1 & \dots & C_m & & & &\\
    \varnothing=B_0 & B_1 & \dots & B_m & & & &\\
    \varnothing=A_0 & A_1 & \dots & A_m & S_0 & S_1 & \dots & S_n\\
    \varnothing=B_0 & B_1 & \dots & B_m & T_0 & T_1 & \dots & T_n\\};
    \mto{1-1}{1-2} \mto{1-2}{1-3} \mto{1-3}{1-4}
    \mto{1-1}{2-1} \mto{1-2}{2-2} \mto{1-4}{2-4}
    \mto{2-1}{2-2} \mto{2-2}{2-3} \mto{2-3}{2-4}
    \eto{3-1}{2-1} \eto{3-2}{2-2} \eto{3-4}{2-4}
    \mto{3-1}{3-2} \mto{3-2}{3-3} \mto{3-3}{3-4}
    \mto{4-1}{4-2} \mto{4-2}{4-3} \mto{4-3}{4-4} \mto{4-4}{4-5} \mto{4-5}{4-6} \mto{4-6}{4-7} \mto{4-7}{4-8}
    \mto{5-1}{5-2} \mto{5-2}{5-3} \mto{5-3}{5-4} \mto{5-4}{5-5} \mto{5-5}{5-6} \mto{5-6}{5-7} \mto{5-7}{5-8}
    \good{1-1}{2-2} \comm{2-1}{3-2}
    \end{toolongdiagram}
    Here, the first (top) diagram is an $m$-simplex in $S_\bullet E(\C,\C,\C)$, and the two bottom sequences denote an $(m+n+1)$-simplex in $S_\bullet \C\times S_\bullet \C$. We omit the choices of the rest of the staircases in our notation, but remark that they are indeed part of the data of these simplices.

    Now, for each $n$, let $E_n\colon S_\cdot F|\C^2(-,n)\to S_\cdot F|\C^2(-,n)$ denote the simplicial map that sends a simplex in $S_\cdot F|\C^2(m,n)$ as above to the following 
\begin{toolongdiagram}
    {\varnothing & \varnothing & \dots & \varnothing & &  & &\\
    \varnothing & \varnothing & \dots & \varnothing & & & &\\
    \varnothing & \varnothing & \dots & \varnothing & & & &\\
    \varnothing & \varnothing & \dots & \varnothing & S_0/S_0 & S_1/S_0 & \dots & S_n/S_0\\
    \varnothing & \varnothing & \dots & \varnothing & T_0/T_0 & T_1/T_0 & \dots & T_n/T_0\\};
    \eq{1-1}{1-2} \eq{1-2}{1-3} \eq{1-3}{1-4}
    \mto{1-1}{2-1} \mto{1-2}{2-2} \mto{1-4}{2-4}
    \eq{2-1}{2-2} \eq{2-2}{2-3} \eq{2-3}{2-4}
    \eto{3-1}{2-1} \eto{3-2}{2-2} \eto{3-4}{2-4}
    \eq{3-1}{3-2} \eq{3-2}{3-3} \eq{3-3}{3-4}
    \eq{4-1}{4-2} \eq{4-2}{4-3} \eq{4-3}{4-4} \mto{4-4}{4-5} \mto{4-5}{4-6} \mto{4-6}{4-7} \mto{4-7}{4-8}
    \eq{5-1}{5-2} \eq{5-2}{5-3} \eq{5-3}{5-4} \mto{5-4}{5-5} \mto{5-5}{5-6} \mto{5-6}{5-7} \mto{5-7}{5-8}
    \good{1-1}{2-2} \comm{2-1}{3-2}
    \end{toolongdiagram}
McCarthy shows in \cite[Page 326]{McCarthy} that if the simplicial maps $E_n$ are homotopy equivalences for all $n$, then the map $S_\bullet F\colon S_\bullet E(\C,\C,\C)\to S_\bullet \C\times S_\bullet \C$ is a homotopy equivalence, as desired. In turn, to prove this claim about  $E_n$, he defines a simplicial map  $$\Gamma_n\colon S_\cdot F|\C^2(-,n)\to S_\cdot F|\C^2(-,n)$$ sending a simplex in $S_\cdot F|\C^2(m,n)$ to the following 
\begin{toolongdiagram}
    {\varnothing & \varnothing & \dots & \varnothing & &  & &\\
    \varnothing=B_0 & B_1 & \dots & B_m & & & &\\
   \varnothing=B_0 & B_1 & \dots & B_m & & & &\\
    \varnothing & \varnothing & \dots & \varnothing & S_0/S_0 & S_1/S_0 & \dots & S_n/S_0\\
     \varnothing=B_0 & B_1 & \dots & B_m & T_0 & T_1 & \dots & T_n\\};
    \eq{1-1}{1-2} \eq{1-2}{1-3} \eq{1-3}{1-4}
    \mto{1-1}{2-1} \mto{1-2}{2-2} \mto{1-4}{2-4}
    \eq{2-1}{2-2} \eq{2-2}{2-3} \eq{2-3}{2-4}
    \eq{3-1}{2-1} \eq{3-2}{2-2} \eq{3-4}{2-4}
    \eq{3-1}{3-2} \eq{3-2}{3-3} \eq{3-3}{3-4}
    \eq{4-1}{4-2} \eq{4-2}{4-3} \eq{4-3}{4-4} \mto{4-4}{4-5} \mto{4-5}{4-6} \mto{4-6}{4-7} \mto{4-7}{4-8}
    \mto{5-1}{5-2} \mto{5-2}{5-3} \mto{5-3}{5-4} \mto{5-4}{5-5} \mto{5-5}{5-6} \mto{5-6}{5-7} \mto{5-7}{5-8}
    \good{1-1}{2-2} \comm{2-1}{3-2}
    \end{toolongdiagram}
One can easily verify that $\Gamma_n$ is a retraction onto the subspace $X\subseteq S_\cdot F|\C^2(-,n)$ in which all the $A_i$'s are $\varnothing$, and that $E_n\Gamma_n=E_n$. Moreover, one can construct a homotopy equivalence $E_n\vert_X\simeq \id_X$ using a similar argument to the classical one that contracts a category with a terminal object. Hence, to show that $E_n$ is a homotopy equivalence, it suffices to construct a homotopy equivalence $\Gamma_n\simeq\id$. 

McCarthy's simplicial homotopy admits an analogue in our setting, but this is significantly less evident than the claims we've made in the proof so far. Indeed, proving that all the required simplices exist makes extensive use of our $\fcgwa$ axioms, which in turn permit the functorial constructions of \cref{appendix}. We now elaborate on this part of the proof.
    
 We define a simplicial homotopy $h$ as follows: for each $m$, and each $0\leq i\leq m$, the map $$h_i\colon S_\cdot F|\C^2(m,n)\to S_\cdot F|\C^2(m+1,n)$$ takes a generic element $e\in S_\cdot F|\C^2(m,n)$ 
to the element  $h_i(e)\in S_\cdot F|C^2(m+1,n)$ given by
    \begin{toolongdiagram}
    {\varnothing=A_0 & A_1 & \dots & A_i & S_0  & \dots & S_0 & & & &\\
    \varnothing=C_0 & C_1 & \dots & C_i & C_i\star_{A_i}S_0  & \dots & C_m\star_{A_m}S_0 & & & &\\
    \varnothing=B_0 & B_1 & \dots & B_i & B_i & \dots & B_m & & & &\\
    \varnothing=A_0 & A_1 & \dots & A_i & S_0  & \dots & S_0 & S_0 & S_1 & \dots & S_n\\
    \varnothing=B_0 & B_1 & \dots & B_i & B_i  & \dots & B_m & T_0 & T_1 & \dots & T_m\\}; 
    \mto{1-1}{1-2} \mto{1-2}{1-3} \mto{1-3}{1-4} \mto{1-4}{1-5} \eq{1-5}{1-6}  \eq{1-6}{1-7}
    \mto{1-1}{2-1} \mto{1-2}{2-2} \mto{1-4}{2-4} \mto{1-5}{2-5} \mto{1-7}{2-7} 
    \mto{2-1}{2-2} \mto{2-2}{2-3} \mto{2-3}{2-4} \mto{2-4}{2-5} \mto{2-5}{2-6} \mto{2-6}{2-7} 
    \eto{3-1}{2-1} \eto{3-2}{2-2} \eto{3-4}{2-4} \eto{3-5}{2-5} \eto{3-7}{2-7} 
    \mto{3-1}{3-2} \mto{3-2}{3-3} \mto{3-3}{3-4} \eq{3-4}{3-5} \mto{3-5}{3-6} \mto{3-6}{3-7} 
    \mto{4-1}{4-2} \mto{4-2}{4-3} \mto{4-3}{4-4} \mto{4-4}{4-5} \eq{4-5}{4-6}  \eq{4-6}{4-7} \eq{4-7}{4-8} \mto{4-8}{4-9} \mto{4-9}{4-10} \mto{4-10}{4-11}
    \mto{5-1}{5-2} \mto{5-2}{5-3} \mto{5-3}{5-4} \eq{5-4}{5-5} \mto{5-5}{5-6} \mto{5-6}{5-7} \mto{5-7}{5-8} \mto{5-8}{5-9} \mto{5-9}{5-10} \mto{5-10}{5-11} 
    \good{1-1}{2-2} \good{1-4}{2-5} 
    \comm{3-1}{2-2} \comm{3-4}{2-5} 
    \end{toolongdiagram}
    where the maps and squares between $\star$-pushouts are given by \cref{lemma2.14}; note that the $\star$-pushouts exist by axiom (PO), and that the cited result makes additional use of axioms ($\star$) and (POL). 
    
Even though they are not pictured in the above diagrams, we must make choices of staircases, and verify that the maps pictured above truly give  kernel-cokernel pairs in $\C$. Let $A_{k,l}, B_{k,l}$ and $C_{k,l}$ denote the objects in the (non-depicted) staircases of the top, bottom, and middle rows of the extension in $e\in S_\bullet F|\C^2(m,n)$. Similarly, denote by $h_i(e)^A_{k,l}$, $h_i(e)^B_{k,l}$ and $h_i(e)^C_{k,l}$ the objects in the staircases of the top, bottom, and middle rows of the extension in $h_i(e)$. Then, we let
\begin{align*}
    h_i(e)^A_{k,l}&=\begin{cases}
A_{k,l} & k,l\leq i\\
S_0/ A_{0,k} & k\leq i, l>i\\
\varnothing & \text{otherwise}
\end{cases} \\
h_i(e)^C_{k,l}&=\begin{cases}
C_{k,l} & k,l\leq i\\
h_i(e)^A_{k,l} & k=i,l=i+1\\
h_i(e)^B_{k,l} & l,k\geq i+1\\
C_{k,l-1}\star_{A_{k,l-1}} h_i(e)^A_{k,l} & \text{otherwise}
\end{cases}\\
h_i(e)^B_{k,l}&=\begin{cases}
B_{k,l} & k,l\leq i\\
B_{k,l-1} & k\leq i, l\geq i+1\\
B_{k-1,l-1} & k\geq i+1, l\geq i+1
\end{cases}\end{align*}

First, we must make sure that the data of $h_i(e)^A, h_i(e)^B$ and $h_i(e)^C$ actually form staircases. The first two are immediate, as all the squares involved are squares already present in $e$. The fact that $h_i(e)^C$ is a staircase is due to the existence of 
distinguished squares
\vspace{-0.3cm}
    \begin{diagram}
    {C_{k,l}\star_{A_{k,l}} S_{k,0} & C_{k,l+1}\star_{A_{k,l+1}} S_{k,0} &  & C_{k,l} & C_{k,l}\star_{A_{k,l}} S_{k,0} \\
    B_{k,l} & B_{k,l+1} & & C_{k+1,l} & C_{k+1,l}\star_{A_{k+1,l}} S_{k+1,0}\\};
    \mto{1-1}{1-2} \mto{2-1}{2-2} \mto{1-4}{1-5} \mto{2-4}{2-5}
    \eto{2-1}{1-1} \eto{2-2}{1-2} \eto{2-4}{1-4} \eto{2-5}{1-5}
    \dist{1-1}{2-2} \dist{1-4}{2-5}
    \end{diagram}
    \vspace{-0.5cm}
    \vspace{-0.3cm}
    \begin{diagram}
    {C_{k,l}\star_{A_{k,l}} S_{k,0} & C_{k,l+1}\star_{A_{k,l+1}} S_{k,0}\\
    C_{k+1,l}\star_{A_{k+1,l}} S_{k+1,0} & C_{k+1,l+1}\star_{A_{k+1,l+1}} S_{k+1,0}\\};
    \mto{1-1}{1-2} \mto{2-1}{2-2}
    \eto{2-1}{1-1} \eto{2-2}{1-2}
    \dist{1-1}{2-2}
    \end{diagram}
    \vspace{-0.5cm}
arising from \cref{lemma2.14}, \cref{lemma2.15}, and \cref{lemma2.16distinguished} respectively, where we abbreviate $S_{k,0}\coloneqq S_0/ A_{0,k}$. Note that these results collectively require axioms (K), ($\star$), (PBL), 
 and (POL). 

For each $k,l$, we have evident choices of maps  $$h_i(e)^A_{k,l}\mrto h_i(e)^C_{k,l}\elto h_i(e)^B_{k,l}$$ which form kernel-cokernel sequences. It remains to check that these assemble into maps $$h_i(e)^A\mrto h_i(e)^C\elto h_i(e)^B;$$ that is, that all the squares between the staircases are of the correct form. A careful study reveals that this is ensured by the aforementioned properties of the $\star$-pushout, together with the fact that by \cref{lemma2.15},  we have squares
\vspace{-0.3cm}
    \begin{diagram}
    {S_{k,0} & C_{k,l}\star_{A_{k,l}} S_{k,0}\\
    S_{k+1,0} & C_{k+1,l}\star_{A_{k+1,l}} S_{k+1,0}\\};
    \mto{1-1}{1-2} \mto{2-1}{2-2}
    \eto{2-1}{1-1} \eto{2-2}{1-2}
    \comm{1-1}{2-2}
    \end{diagram}
    \vspace{-0.5cm}
for each $k,l$, whose induced map on cokernels is  the map $B_{k+1,l}\erto B_{k,l}$ found in $e$. 

Just as in \cite[Proposition 4.17]{Campbell}, one can check that $h$ defines a simplicial homotopy from $\Gamma_n$ to $\id$.
\end{proof}

It will also be useful to have a more general version of the Additivity Theorem at hand.

\begin{thm}\label{additivity}
Let $\A,\B\subseteq\C$ be full $\fcgwa$ subcategories of a relative $\fcgwa$ category $(\C,\W)$. Then, the map
$$wS_\bullet E(\A,\C,\B)\to wS_\bullet \A\times wS_\bullet\B$$ induced by $$(A\mrto C\elto B)\mapsto (A,B)$$ is a homotopy equivalence.
\end{thm}
\begin{proof}
The proof is identical to the relevant part of \cite[Proposition 1.3.2]{Waldhausen}, since by \cref{trivialextension} our $\fcgwa$ categories always admit trivial extensions of the form \[A\mrto A\star_\varnothing B \elto B\]
\end{proof}

In several instances, it will be useful to recognize when a certain $\fcgwa$ category is equivalent (in the sense of \cref{defn:equivalence}) to an extension category. We study this in the following lemma.

\begin{lemma}\lbl{extensioncatsequiv1}
Let $\A,\B\subseteq\C$ be full $\fcgwa$ subcategories of an $\fcgwa$ category $\C$ with inclusion functors $i_\A,i_\B$. $\C$ is equivalent to $E(\A,\C,\B)$ if we have the following:
\begin{itemize}
    \item $\fcgwa$ functors $F : \C \to \A$, $G : \C \to \B$,
    \item an m-natural transformation $\phi : i_\A F \mrto 1_\C$,
    \item an e-natural transformation $\psi : i_\B G \erto 1_\C$,
    \item for each object $C$ in $\C$, $\begin{inline-diagram} {FC & C & GC\\};
    \mto{1-1}{1-2}^{\phi_C} \eto{1-3}{1-2}_{\varphi_C}\end{inline-diagram}$ is a kernel-cokernel pair,
    \item every extension $A\mrto C\elto B$ in $\C$, with $A\in \A, B\in\B, C\in\C,$ is isomorphic to one of the above form.
\end{itemize}
\end{lemma}

\begin{proof}
The data above, excluding the last property, determine a $\fcgwa$ functor $\C \to E(\A,\C,\B)$ left inverse to the forgetful functor $E(\A,\C,\B) \to \C$ so long as $\phi,\psi$ are good natural transformations, which is automatic by \cref{goodkernelsquare}. 

It remains then to show that this functor is an equivalence by checking the conditions of \cref{equiv:char}. Essential surjectivity holds by our last assumption. Fullness and faithfulness for m-morphisms follows from \cref{mixedpullbackuniqueness} and the analogous uniqueness of pullback squares, as any m-morphism in $E(\A,\C,\B)$ as above is uniquely determined by its source and target extensions and the map $f$. The same properties follow dually for e-morphisms and similarly for  squares, which are uniquely determined by their boundaries.
\end{proof}

\begin{cor}\label{extensioncatsequiv}
Let $(\C,\W)$ be a relative $\fcgwa$-category under the conditions of \cref{extensioncatsequiv1}, such that the functors $F$ and $G$ preserve acyclic objects. Then, we have  homotopy equivalences $$wS_\bullet\C\simeq wS_\bullet E(\A,\C,\B)\simeq wS_\bullet\A\times wS_\bullet\B.$$
\end{cor}

\begin{proof}
It is tedious but straightforward to check that an equivalence $L:\C\to E(\A,\C,\B)$ in the sense of \cref{defn:equivalence} induces an equivalence $S_n\C\to S_nE(\A,\C,\B)$ for each $n$. Moreover, these restrict to equivalences $wS_n\C\to wS_nE(\A,\C,\B)$, since the fact that $F$ and $G$ preserve acyclic objects implies that a map $f$ in $\C$ is an m-equivalence (resp.\ e-equivalence) if and only if $Lf$ is an m-equivalence (resp.\ e-equivalence). The second homotopy equivalence is then a consequence of \cref{additivity}.
\end{proof}

As an example of \cref{extensioncatsequiv}, we can describe product decompositions of the $K$-theory of extensive categories.

\begin{ex}\label{extensiveadditivity}
    Given a Serre subcategory $\Y$ of an extensive category $\X$, the extensions $\pi_\Y X \hookrightarrow X \hookleftarrow \pi_{-\Y} X$ from \cref{extensiveprojection} are cartesian-natural in $X$ and satisfy all of the conditions of \cref{extensioncatsequiv1}: the only one that is interesting to check is the final condition. If we are given an extension $B \hookrightarrow A \hookleftarrow C$ with $B$ in $\Y$ and $C$ in $\X - \Y$ but which is not of this form, then we have $A \cong B \sqcup D \sqcup C'$, where $D$ is the complement of $B \hookrightarrow \pi_Y A$ and $C \cong D \sqcup C'$; but $D$ is in $\Y$, so $C$ cannot be in $\X - \Y$, a contradiction. 
    Then, by \cref{extensioncatsequiv}, we have that $K(\X) \simeq K(\Y) \times K(\X-\Y)$.
\end{ex}

\begin{ex}\label{gsetadditivity}
    In the case of finite $G$-sets for a finite group $G$ and the Serre subcategory of free finite $G$-sets, \cref{extensiveadditivity} and \cref{gsets} shows that $K(\finset_G)$ is equivalent to the product of the $K$-theory of finitely generated free $G$-sets and the $K$-theory of finite coproducts of non-free transitive $G$-sets.
\end{ex}

\section{Delooping}\lbl{section:spectra}

In this section, we show that $K(\C,\W)$ is a spectrum, for any relative $\fcgwa$ category $(\C,\W)$. This is done by defining a notion of relative $K$-theory and following the same outline as in \cite[Section 1.5]{Waldhausen}; we include the proofs here for completeness.

\begin{defn}
Let $F: \A\to\B$ be an $\fcgwa$ functor between $\fcgwa$ categories. For each $n$, we define the double category $S_n(F)$ as the pullback
    \begin{diagram}
    {S_n (F) & S_{n+1}\B\\
    S_n \A & S_n \B\\};
    \to{1-1}{1-2} \to{1-1}{2-1} \to{1-2}{2-2}^{d_0} \to{2-1}{2-2}_F
    \path[font=\scriptsize] (m-1-1) edge[-,white] node[pos=0.05]
  {\textcolor{black}{$\lrcorner$}}  (m-2-2);
    \end{diagram}
\end{defn}

$S_n(F)$ is then the double category of staircases in $S_{n+1}\B$ which are equipped with a lift of all but the top row to $S_n\A$ along $F$. Note that $S_{\bullet +1}\B$ is what is typically called the simplicial path space, or d\'ecalage, of $S_\bullet \B$.

\begin{lemma}
$S_\bullet(F)$ is a simplicial $\fcgwa$ category.
\end{lemma}
\begin{proof}
The fact that each $S_n(F)$ is an $\fcgwa$ category follows directly from the $\fcgwa$ structures on $S_{n+1}\B$ and $S_n\A$ given by \cref{statementSnfcgwa}. The face and degeneracy maps are given by shifting those of $S_{n}\B$; that is, $d_i^{S_\bullet(F)}\coloneqq d_{i+1}^{S_\bullet\B}$, and $s_i^{S_\bullet(F)}\coloneqq s_{i+1}^{S_\bullet\B}$.
\end{proof}

Just as in \cite[Chapter IV, 8.5.4]{Kbook}, we have the following.

\begin{lemma}\lbl{relativecontractible}
If $\A=\B$, then $wS_\bullet S_\bullet (\id_\B)$ is contractible.
\end{lemma}
\begin{proof}
Note that in this case, $S_n(\id_\B)$ is defined via the pullback
    \begin{diagram}
    {S_n(\id_\B) & S_{n+1}\B\\
    S_n \B& S_n\B \\};
    \to{1-1}{1-2} \to{1-1}{2-1} \eq{2-1}{2-2} \to{1-2}{2-2}^{d_0}
    \path[font=\scriptsize] (m-1-1) edge[-,white] node[pos=0.05]
  {\textcolor{black}{$\lrcorner$}}  (m-2-2);
    \end{diagram}
 and thus $S_n(\id_\B)\cong S_{n+1}\B$; in other words, $S_\bullet (\id_\B)$ is the simplicial path space of $S_\bullet \B$. Similarly, one can see that $wS_n S_\bullet (\id_\B)$ is the simplicial path space of $w S_n S_\bullet \B$ for each $n$. Then, we have a homotopy equivalence $wS_n S_\bullet (\id_\B)\simeq w S_n S_0 \B\simeq\ast$ for each $n$, from which we conclude our result.
\end{proof}

Given an $\fcgwa$ functor $F: \A\to\B$ we have $\fcgwa$ functors $$\B\to S_n(F)$$ taking $B\in\B$ to $\varnothing\mrto B\req \dots\req B\in S_n(F)$, and
  $$S_n (F)\to S_n \A$$ given by one of the legs of the pullback. These functors satisfy the following proposition, 
  analogous to \cite[Proposition 1.5.5]{Waldhausen}.

\begin{prop}\lbl{propspectrum}
Let $F:\A\to\B$ be an $\fcgwa$ functor. Then, we have a homotopy fiber sequence $$wS_\bullet \B\to wS_\bullet S_\bullet (F)\to wS_\bullet S_\bullet \A$$
\end{prop}
\begin{proof}
First, we have a homotopy equivalence $wS_\bullet S_n (F)\simeq wS_\bullet E(\B,S_n(F),S_n\A)$, as the conditions in \cref{extensioncatsequiv} are easily checked. Then, by the Additivity \cref{additivity}, we have a homotopy equivalence $$wS_\bullet S_n (F)\simeq wS_\bullet \B\times wS_\bullet S_n\A$$ for each $n$, from which we deduce the existence of the homotopy fiber sequence in the statement, as each term $wS_\bullet S_n\A$ is connected.
\end{proof}



We can finally deduce the main result in this section.

\begin{thm}\lbl{spectra}
Let $(\C,\W)$ be a relative $\fcgwa$ category. Then, iterating the $S_\bullet$-construction exhibits  $K(\C,\W)=\Omega|wS_\bullet\C|$ as an $\Omega$-spectrum.
\end{thm}
\begin{proof}
Using \cref{propspectrum} for $\A=\B=\C$ yields a homotopy fiber sequence $$wS_\bullet \C\to wS_\bullet S_\bullet (\id_\C)\to wS_\bullet S_\bullet \C$$ But $wS_\bullet S_\bullet (\id_\C)$ is contractible by \cref{relativecontractible}, and so we conclude that there exists a homotopy equivalence $|wS_\bullet \C|\simeq \Omega|wS_\bullet S_\bullet\C|$. Iterating this process yields the desired delooping $$|wS_\bullet \C|\simeq \Omega|wS_\bullet S_\bullet\C|\simeq \Omega \Omega|wS_\bullet S_\bullet S_\bullet\C|\simeq\dots \simeq\Omega^n|wS^{n+1}_\bullet\C|\simeq\dots $$
\end{proof}

\section{The Fibration Theorem}\lbl{section:fibration}

This section is dedicated to our primary tool for comparing $\fcgwa$ categories: the analogue of Waldhausen's Fibration Theorem, which relates the $K$-theory spectra of an $\fcgwa$ category equipped with two comparable classes of weak equivalences. The statement is as follows.

\begin{thm}[Fibration]\lbl{fibrationthm}
Let $\V$ and $\W$ be two classes of acyclic objects on an $\fcgwa$ category $\C$, such that $\V\subseteq\W$. Then, there exists a homotopy fiber sequence
$$K(\W,\V) \to K(\C,\V) \to K(\C,\W)$$ which is moreover a fiber sequence of spectra.
\end{thm}

For the sake of clarity, we note that this is not a generalization of Waldhausen's Fibration Theorem, in the sense that it does not contain the classical result as a special case; indeed, it is not the case that any Waldhausen category gives rise to an $\fcgwa$ category. 

Our proof largely follows that of Waldhausen, but avoids the rather burdensome assumptions that go into proving that the category of weak equivalences is homotopy equivalent to that of trivial cofibrations. Indeed, the reader might have noticed  we do not require any additional conditions on our structures in order for our Fibration Theorem to hold. In contrast, the classical version due to Waldhausen (see  \cite[Theorem 1.6.4]{Waldhausen}) asks for the saturation and extension axioms, and for the existence of a cylinder functor satisfying the cylinder axiom. More modern accounts have managed to eliminate some of these conditions: a more relaxed version of Waldhausen's Fibration due to Schlichting (see \cite[Theorem A.3]{Sch06})  replaces cylinders by factorizations, asking that every map factors as a cofibration followed by a weak equivalence. Raptis takes this one step further and removes the extension axiom with a clever application of Additivity (see \cite[Theorem 2.8]{raptis}); however, the factorization and saturation conditions remain. 

The reason behind this apparent clash is that our $\fcgwa$ categories were, in a way, constructed so that all of these properties are already incorporated. Namely, the saturation axiom (in our case, the fact that m- and e-equivalences satisfy 2-out-of-3) is an easy consequence of the definition of m- and e-equivalences, as seen in \cref{we_2of3}. Similarly, the extension axiom is required in the classical setting in order to prove that trivial cofibrations can be characterized by having acyclic cokernels; this is precisely how all our m-equivalences are defined in \cref{FCGWA_we}.

As for the absence of a cylinder or factorization requirement, the reason is that all of the maps that our constructions see are already ``simple enough'' and do not need to be decomposed any further; this is a feature of the double-categorical approach. Concretely, this amounts to considering only admissible monomorphisms and epimorphisms in an exact category as opposed to working with arbitrary morphisms. 

As a consequence, our proof departs from Waldhausen's in that it does not need to go through the subcategory of trivial cofibrations, which he denotes $\overline{w}S_\bullet \C$. Instead, we rely on the following result, which exploits the symmetry of our setting, where vertical maps have equally convenient properties to horizontal ones.

\begin{prop}\lbl{grid:additivity} 
For any refinement $(\C,\V)$ of $(\C,\W)$ and any $l,m$, we have homotopy equivalences of simplicial double categories $$vS_\bullet w_{l,m}\C \simeq vS_\bullet w_{0,m}\C \times vS_\bullet w_{l-1,m}\W$$ 
and $$vS_\bullet w_{l,m} \C \simeq vS_\bullet w_{l,0}\C \times vS_\bullet w_{l,m-1}\W$$
\end{prop}
\begin{proof}
We prove the first statement; the second is entirely dual. The strategy will be to show that $w_{l,m}\C$ is equivalent (in the sense of \cref{extensioncatsequiv1}) to the extension $\fcgwa$ category $E(w_{l-1,m} \W, w_{l,m} \C, w_{0,m} \C)$; then, we deduce the desired statement from \cref{extensioncatsequiv} and the Additivity \cref{additivity}.

For this, consider an object $A$ in $w_{l,m} \C$ pictured below left, and associate to it the object in $E(w_{l-1,m} \W, w_{l,m} \C, w_{0,m} \C)$ pictured below right (where $l,m$ are pictured as 2 and 1 respectively for convenience). We henceforth abuse notation and identify $w_{0,m} \C$ with its image under the inclusion $w_{0,m} \C\hookrightarrow w_{l,m} \C$, and similarly for $w_{l-1,m} \W$.
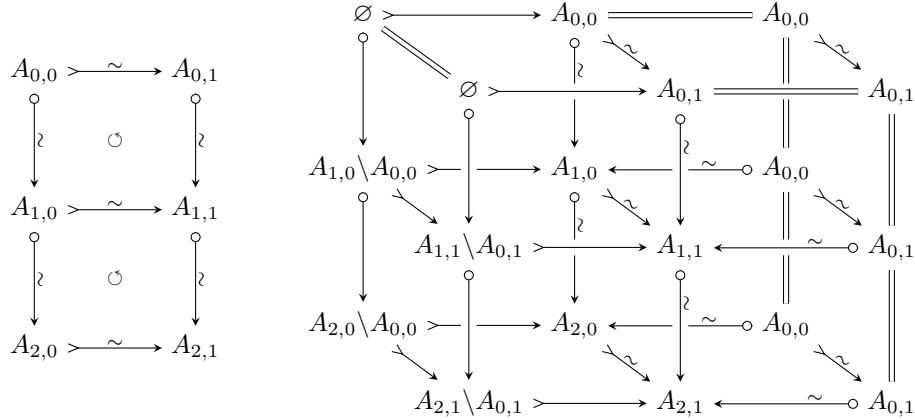
    \[\begin{inline-diagram}
    {A_{0,0} & & A_{0,1}\\
    \\
    A_{1,0} & & A_{1,1}\\
    \\
    A_{2,0}& &  A_{2,1}\\};
    \wmto{1-1}{1-3} \wmto{3-1}{3-3} \wmto{5-1}{5-3}
    \weto{1-1}{3-1} 
    \weto{1-3}{3-3}
    \weto{3-1}{5-1} 
    \weto{3-3}{5-3}
    \comm{1-1}{3-3} \comm{3-1}{5-3}
    \end{inline-diagram} \qquad
    \begin{squishedinlinediagram}[column sep={4em,between origins}]
    {\varnothing  & & A_{0,0} & & A_{0,0}& \\
    & \varnothing & & A_{0,1}&  & A_{0,1}\\
   A_{1,0}\bs A_{0,0} & & A_{1,0}&  & A_{0,0} & \\
    & A_{1,1}\bs A_{0,1} & & A_{1,1} & &  A_{0,1} \\
    A_{2,0}\bs A_{0,0} & & A_{2,0}&  & A_{0,0} & \\
     & A_{2,1}\bs A_{0,1} & & A_{2,1} & &  A_{0,1} \\};
    \eq{1-1}{2-2} \wmto{1-3}{2-4} \wmto{1-5}{2-6}
    \mto{3-1}{4-2} \wmto{3-3}{4-4} \wmto{3-5}{4-6}
    \mto{1-1}{1-3} \mto{3-1}{3-3}
    \eq{1-5}{1-3} 
    \path[font=\scriptsize] (m-3-5) edge[emor] node[pos=0.3,above=-0.7ex,sloped] {$\sim$} (m-3-3);
    \eto{1-1}{3-1}  
    \path[font=\scriptsize] (m-1-3) edge[emor] node[pos=0.3,above=-0.7ex,sloped] {$\sim$} (m-3-3);
    \eq{1-5}{3-5}
    \over{2-2}{4-2} \over{2-4}{4-4}
    \eto{2-2}{4-2} 
    \path[font=\scriptsize] (m-2-4) edge[emor] node[pos=0.3,above=-0.7ex,sloped] {$\sim$} (m-4-4);
    \eq{2-6}{4-6}
    \over{2-6}{2-4}
    \eq{2-6}{2-4} 
    \over{2-2}{2-4}
    \mto{2-2}{2-4} 
    \eto{3-1}{5-1} 
    \path[font=\scriptsize] (m-3-3) edge[emor] node[pos=0.3,above=-0.7ex,sloped] {$\sim$} (m-5-3);
    \eq{3-5}{5-5}
    \mto{5-1}{5-3} 
    \path[font=\scriptsize] (m-5-5) edge[emor] node[pos=0.3,above=-0.7ex,sloped] {$\sim$} (m-5-3);
    \mto{6-2}{6-4} 
    \path[font=\scriptsize] (m-6-6) edge[emor] node[pos=0.3,above=-0.7ex,sloped] {$\sim$} (m-6-4);
    \mto{5-1}{6-2} \wmto{5-3}{6-4} \wmto{5-5}{6-6}
    \over{4-2}{6-2} \over{4-4}{6-4}
    \eto{4-2}{6-2} 
    \path[font=\scriptsize] (m-4-4) edge[emor] node[pos=0.3,above=-0.7ex,sloped] {$\sim$} (m-6-4);
    \eq{4-6}{6-6}
    \over{4-2}{4-4} \mto{4-2}{4-4}
    \over{4-6}{4-4} 
    \path[font=\scriptsize] (m-4-6) edge[emor] node[pos=0.3,above=-0.7ex,sloped] {$\sim$} (m-4-4);
    \end{squishedinlinediagram}\]

First of all, we check that the diagram above right truly is an object of \linebreak $E(w_{l-1,m} \W, w_{l,m} \C, w_{0,m} \C)$. Indeed, all of the squares are either good squares or squares in the double category, it is clearly a kernel-cokernel pair since these are constructed pointwise, and the grid on the right is an element of $w_{0,m} \C$. Lastly, the grid on the left is comprised of objects in $\W$ since they are all kernels of e-equivalences, and then the maps between them must be w- and e-equivalences by \cref{acyclic:we}; thus, this grid is an object of $w_{l-1,m} \W$.

Now, to use \cref{extensioncatsequiv1}, we need to define $\fcgwa$ functors $R:w_{l,m}\C \to w_{0,m} \C$ and \linebreak $L:w_{l,m}\C \to w_{l-1,m} \W$ together with an e-natural transformation $\eta:R\Rightarrow \id$ and an m-natural transformation $\mu:R\Rightarrow \id$. Let $R$ and $L$ respectively send an object $A$ as above left to the right and left grids in the pictured extension, and let the components of $\eta$ and $\mu$ be given by the horizontally depicted e- and m-morphisms. The mixed naturality squares of $\eta$ are given by composites of the squares in the grid $A$, and $\mu$ is the kernel transformation of $\eta$. 

$R$ is evidently an $\fcgwa$ functor and $\eta$ an e-natural transformation whose component squares are good.
To see that $L$ is an $\fcgwa$ functor, we must check that it preserves the remaining relevant structure. The fact that $L$ preserves good squares is ensured by the converse in \cref{cubecorrespondence}, and it also preserves $\star$-pushouts, since by \cref{rmkstarkercoker} the $\star$-pushout of the kernels is the kernel of the $\star$-pushouts. To see that $L$ preserves cokernels, let $A\mrto B$ be an m-morphism in $w_{l,m}\C$ and construct the following diagram
    
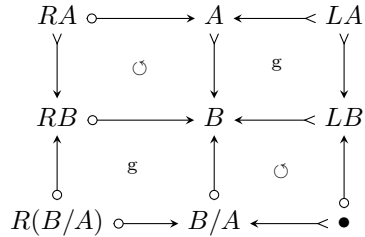
\begin{diagram}
    {RA & A & LA\\
    RB & B & LB\\
    R(B/A) & B/A & \bullet\\};
    \eto{1-1}{1-2} \eto{2-1}{2-2} \eto{3-1}{3-2}
    \mto{1-3}{1-2} \mto{2-3}{2-2} \mto{3-3}{3-2}
    \mto{1-1}{2-1} \mto{1-2}{2-2} \mto{1-3}{2-3}
    \eto{3-1}{2-1} \eto{3-2}{2-2} \eto{3-3}{2-3}
    \comm{1-1}{2-2} \comm{2-2}{3-3}
    \good{1-2}{2-3} \good{2-1}{3-2}
    \end{diagram}
where all columns and rows are kernel-cokernel pairs. Then, we have that $\bullet$ must be both the kernel of $R(B/A)\erto B/A$ (which is by definition $L(B/A)$) and the cokernel of $LA\mrto LB$ (which is $LB/LA$). This shows that $L$ preserves cokernels; the proof for kernels is analogous. Lastly, $L$ preserves acyclic objects, as $\V$ is closed under kernels.

As to the last condition of \cref{extensioncatsequiv1}, in order to see that every object \linebreak $B\mrto A\elto C$ in $E(w_{l-1,m} \W, w_{l,m} \C, w_{0,m} \C)$ is of the form $LA\mrto A \elto RA$ up to isomorphism, note that as $B\in w_{l-1,m} \W$, it has initial objects in the top row, and so the top components of $C\erto A$ are necessarily isomorphisms by \cref{isosemptycoker}. Hence, up to isomorphism, each row of $C$ must agree with the top row of $A$, and we get that $C\cong RA$. As $k$ preserves isomorphisms, this implies that $B\cong LA$, completing the proof.
\end{proof}

We can now proceed to the proof of the Fibration Theorem.

\begin{proof}(\cref{fibrationthm})
To obtain the desired homotopy fiber sequence on $K$-theory, it is enough to show that $$vS_\bullet \W  \to vS_\bullet  \C \to wS_\bullet \C$$ is a homotopy fiber sequence. For this, let $vwS_\bullet \C$ denote the simplicial quadruple category which has
$w$-m-equivalences and $w$-e-equivalences in the first and second directions, $v$-m-equivalences and $v$-e-equivalences in the third and fourth directions, with either squares in the double category or commuting squares between them as appropriate for the higher cells. 

 Note that we can include $vS_\bullet \C$ into $vwS_\bullet \C$ by considering identities in the $w$-directions. Similarly, we have an inclusion of $wS_\bullet \C$ into $vwS_\bullet \C$ which, as $\V\subseteq \W$, is furthermore a homotopy equivalence by the 2-dimensional analogue of Waldhausen's Swallowing Lemma (\cite[Lemma 1.6.5]{Waldhausen}), proven easily by applying the original Lemma twice. 
We will abuse notation and write $vwS_\bullet \C\to wS_\bullet \C$ for the homotopy inverse, which formally only exists at the level of spaces.

In order to show that the sequence pictured above is a homotopy fiber sequence, it suffices to prove that the outer rectangle below is a homotopy pullback, as each category $w_{-,m}S_n \W$ has an initial object and so $wS_\bullet \W$ is contractible. 
    \begin{diagram}
    {vS_\bullet \W  & vwS_\bullet \W & wS_\bullet \W \\
    vS_\bullet \C  & vwS_\bullet \C & wS_\bullet \C\\};
    \to{1-1}{1-2} \to{1-2}{1-3} \to{2-1}{2-2} \to{2-2}{2-3}
    \to{1-1}{2-1} \to{1-2}{2-2} \to{1-3}{2-3}
    \end{diagram}
Since the horizontal maps in the square above right are homotopy equivalences by the Swallowing Lemma, this is equivalent to showing that the square above left is a homotopy pullback. 

Up to this point, our proof is virtually identical (albeit higher-dimensional) to \cite[Theorem 1.6.4]{Waldhausen}. The conclusion, however, diverges from Waldhausen's approach and instead exploits the symmetry in our $\fcgwa$ categories.

Recall that we have homotopy equivalences
    \begin{align*}
        vw_{l,m}S_\bullet \C &  \simeq vS_\bullet w_{l,m}\C\\
                            &  \simeq (vS_\bullet w_{0,m}\C) \times (vS_\bullet w_{l-1,m}\W) \\
                            & \simeq (vS_\bullet w_{0,0}\C \times vS_\bullet w_{0,m-1}\W) \times (vS_\bullet w_{l-1,m}\W)\\
    \end{align*}
where the first equivalence (in fact, isomorphism) is due to \cref{grid_commute}, and the others are obtained from \cref{grid:additivity}. Then, we have $$vw_{l,m}S_\bullet \C \simeq vS_\bullet \C \times vS_\bullet w_{0,m-1}\W \times vS_\bullet w_{l-1,m}\W,$$ and using the same reasoning for the $\fcgwa$ category $\W$ in place of $\C$, we see that $$vw_{l,m}S_\bullet \W \simeq vS_\bullet \W \times vS_\bullet w_{0,m-1}\W \times vS_\bullet w_{l-1,m}\W.$$

Writing $X$ for the trisimplicial double category with $$X_{\bullet,l,m} = vS_\bullet w_{0,m-1}\W \times vS_\bullet w_{l-1,m}\W,$$ the argument above shows that the relevant square is homotopy equivalent 
to the following:
    \begin{diagram}
    {vS_\bullet \W  & vS_\bullet \W \times X  \\
    vS_\bullet \C  & vS_\bullet \C \times X\\};
    \to{1-1}{1-2} \to{1-1}{2-1} \to{1-2}{2-2}\to{2-1}{2-2}
    \end{diagram}
which is a homotopy pullback, as the homotopy fibers of the vertical maps agree.

To show that this in fact gives a fiber sequence at the level of spectra, recall by \cref{spectra} that the delooping of $\vert w S_\bullet \C\vert$ is given by $\vert w S^n_\bullet\C\vert$, where $S^n_\bullet\C$ is a multi-simplicial $\fcgwa$ category. Then, it suffices to show that for each $n$ we have a homotopy fiber sequence  
$$vS^n_\bullet \W  \to vS^n_\bullet  \C \to wS^n_\bullet \C.$$ In turn, in order to prove this it suffices to show that for each $k$, fixing one simplicial dimension yields a homotopy fiber sequence
$$vS^{n-1}_\bullet (S_k \W)  \to vS^{n-1}_\bullet (S_k \C) \to wS^{n-1}_\bullet (S_k \C),$$ as all three multi-simplicial objects above are connected. 
This can now be proven inductively using our argument above, as $S_k\C$ is an $\fcgwa$ category in which acyclic objects are defined pointwise, by \cref{statementSnfcgwa}.
\end{proof} 

As an application, we can use \cref{fibrationthm} to show that for any exact category with weak equivalences that gives rise to a relative $\fcgwa$ category, the $K$-theory as a Waldhausen category and as a relative $\fcgwa$ category agree. This concludes the series of comparisons started in \cref{section:Kth}.

\begin{prop}\label{prop:comparisonexact}
Let $\C$ be an exact category with a class of weak equivalences $w$, and let $\W$ be the class of objects $X\in\C$ such that $0\to X$ is in $w$. If $(\C,w)$ is either
\begin{itemize}
    \item a complicial exact category with weak equivalences as in \cite[Definition 3.2.9]{Sch11},
    \item a complicial biWaldhausen category as in \cite[1.2.11]{TT} closed under canonical homotopy pushouts and pullbacks (\cite[1.1.2]{TT}), or
    \item an exact category with weak equivalences constructed from a cotorsion pair as in \cite{cotorsion} and such that $\W$ has 2-out-of-3
\end{itemize}  
then the $K$-theory of $(\C,w)$ as a Waldhausen category is homotopy equivalent to the $K$-theory  of $(\C,\W)$ as an $\fcgwa$ category.
\end{prop}
\begin{proof}
In all the specified cases, there exists a homotopy fiber sequence of $K$-theory spectra of Waldhausen categories $$K(\W,i)\to K(\C,i)\to K(\C,w)$$ where $i$ denotes the corresponding class of isomorphisms. On the other hand, by \cref{fibrationthm}, there exists a homotopy fiber sequence of $K$-theory spectra of $\fcgwa$ categories $$K(\W,\varnothing)\to K(\C,\varnothing)\to K(\C,\W).$$ It then suffices to show that $K(\W,\varnothing)\simeq K(\W,i)$ and $K(\C,\varnothing)\simeq K(\C,i)$. However, recall that the horizontal and vertical weak equivalences determined by the class of acyclics $\varnothing$ are precisely the isomorphisms, and so by \cref{iS:equals:s} we have that $K(\C,\varnothing)\simeq \Omega\vert s_\bullet\C \vert$ for the $\fcgwa$ category $\C$. This is in turn equivalent to the classical construction $\Omega\vert s_\bullet \C\vert$ for the exact category $\C$, which agrees with $K(\C,i)$ by Corollary (2) following \cite[Lemma 1.4.1]{Waldhausen}. The same argument also applies to $\W$, concluding the result. 
\end{proof}

\section{The Localization Theorem}\label{section:localization}

In the previous section, we saw how the Fibration \cref{fibrationthm} allows us to compare the $K$-theory spectra $K(\C, \W)$ and $K(\C, \V)$ of an $\fcgwa$ category $\C$ with two classes of weak equivalences when $\V\subseteq\W$; namely, they differ by a homotopy fiber $K(\W, \V)$. Interestingly, as an immediate consequence of our Fibration Theorem, we obtain a Localization Theorem that allows us to compare the $K$-theory spectra of two different $\fcgwa$ categories $\A\subseteq\B$ by finding a homotopy cofiber. Unlike most previous localization theorems, ours requires only that $\A$ is closed under kernels, cokernels, and extensions in $\B$, compared to the much stronger classical Serre condition.

\begin{thm}[Localization]\lbl{localization}
Let $\A\subseteq\B$ be a full inclusion of $\fcgwa$ categories, such that $\A$ is closed under cokernels of m-morphisms, kernels of e-morphisms, and extensions in $\B$. Then, there exists a relative $\fcgwa$ category $(\B,\A)$ such that
$$K(\A)\to K(\B) \to K(\B,\A)$$
is a homotopy fiber sequence of spectra.
\end{thm}
\begin{proof}
This is a direct application of \cref{fibrationthm} for $\C=\B$, $\W=\A$, $\V=\varnothing$, as any full $\fcgwa$ subcategory $\A\subseteq\B$ which is closed under extensions forms a class of acyclic objects in $\B$.
\end{proof}

This generalizes many Localization Theorems in the literature when restricted to $\fcgwa$ categories arising from exact categories. For example, any inclusion of abelian categories $\A\subseteq\B$ satisfying the hypotheses of Quillen's original Localization Theorem \cite[Theorem 5]{Qui73}, or any inclusion of exact categories $\A\subseteq\B$ satisfying the conditions of either Schlichting's Localization Theorem \cite[Theorem 2.1]{Sch}, C\'ardenas' Localization Theorem \cite{Cardenas}, or the first author's Localization Theorem \cite[Theorem 6.1]{cotorsion} will satisfy the conditions of \cref{localization}. 

Notably, all of these localization theorems except for \cite[Theorem 6.1]{cotorsion} that the subcategory $\A$ be Serre, which in that context means $\A$ must be closed under subobjects and quotients in $\B$. By contrast, \cref{localization} only requires that $\A$ has 2-out-of-3 for short exact sequences in $\B$, and thus provides a wider field for applications than the previously existing results. In particular, as in \cite[Section 8]{cotorsion}, it can be used to compare $K(R)$ and $G(R)$ for certain classes of rings, but unlike \cite[Theorem 6.1]{cotorsion} it does not include a condition on injective objects. 

In each case, passing to a more general setting broadens the scope of the theorem but also reduces the tractability of the resulting cofiber. The inclusion of an exact subcategory closed under extensions, subobjects, and quotients (among other technical conditions) has an exact category as its cofiber. If it is only closed under extensions, kernels, and cokernels but also has enough injective objects, the cofiber is a Waldhausen category \cite{cotorsion}. In turn, without enough injectives, our more general theorem produces an $\fcgwa$ category whose weak equivalences may not satisfy the axioms of a Waldhausen category. 



In the non-additive setting, we can compare our result to the Localization Theorem of Campbell and Zakharevich: any inclusion of ACGW categories $\A\subseteq\B$ satisfying the conditions of \cite[Theorem 8.6]{CZ} will be under the hypotheses of \cref{localization}. In this case, the $\fcgwa$ perspective of adding weak equivalences as additional structure in a cofiber $(\B,\A)$, as opposed to the ACGW perspective of strictly inverting them in a cofiber $\B\bs\A$, lets us avoid several of their conditions including the often tedious process of checking that the double category $\B\bs\A$ is CGW. 

We now illustrate the utility of \cref{localization} with several examples. 

\begin{ex}
    The $K$-theory groups of the relative $\fcgwa$ category $(\C,n\C)$ for $\C$ the $\fcgwa$ category of finite sets or free $R$-modules can now be seen to fit into a long exact sequence 
    \[
    \cdots \to K_1(n\C) \to K_1(\C) \to K_1(\C,n\C) \to K_0(n\C) \xrightarrow{\cdot n} K_0(\C) \to K_0(\C,n\C) \to 0.
    \]
    The map $K_0(n\C) \xrightarrow{\cdot n} K_0(\C)$ in each of these cases is the inclusion of $n\mathbb{Z}$ into $\mathbb{Z}$, so we have $K_0(\C,n\C) \cong \mathbb{Z}/n$. 
    
    More generally, for any group $G$ and surjective homomorphism $\pi \colon K_0(\C) \to G$ we can define $\ker\pi\C$ as the full double subcategory of $\C$ containing the objects whose isomorphism class is sent to 0 by $\pi$. By \cref{kzero}, $\ker\pi\C$ will be closed under kernels, cokernels, and extensions, and by the argument above $K_0(\C,\ker\pi\C) \cong G$. For $G$ any abelian group with a generating set of cardinality $n$, there is a surjective homomorphism $K_0(\finset^n) \cong \mathbb{Z}^n \to G$, so this construction applies quite generally.
    
    In the $K$-theory of Waldhausen categories, this construction appears as the Cofinality Theorem (see for instance \cite[Theorem V.2.3]{Kbook}), which shows that $K_i(\C,\ker\pi\C) \cong 0$ for $i > 0$ and hence $K(\C,\ker\pi\C) \simeq G$ where the group $G$ is regarded as a discrete space. In fact, the proof of Cofinality given in \cite[Theorem V.2.3]{Kbook} applies verbatim in our context as well: the cylinder conditions are not relevant as they are merely the conditions for applying the localization theorem, and the argument in the proof of \cite[Theorem IV.8.10]{Kbook} that $K(\C,\ker\pi\C) \cong G$ can be applied verbatim for our definition of the $S_\bullet$ construction.
\end{ex}

\begin{cor}\label{extensivelocalization}
    Let $\X$ be an extensive category in which all objects are finitary and $\Y$ a Serre subcategory. Then there is a relative ECGW category $(\X,\Y)$ such that $K(\X,\Y) \simeq K(\X - \Y)$, where the homotopy equivalence is induced by the functor $\pi_{\X-\Y} \colon \X \to \X - \Y$ with the inclusion of $\X-\Y$ into $\X$ inducing its homotopy inverse.
\end{cor}

This result shows that for a Serre subcategory of an extensive category, the localization on $K$-theory is equivalent to an extensive category without weak equivalences. This is in perfect analogy with the classical localization theorem for exact categories, and just as every extension in an extensive category is split so too is the homotopy fiber sequence $K(\Y) \to K(\X) \to K(\X - \Y)$.

\begin{proof}
     First note that as $\Y$ is Serre, it is full and closed under coproducts and complements, so $(\X,\Y)$ forms a relative ECGW category. By \cref{localization}, we then have a homotopy fiber sequence of spectra 
     \[K(\Y)\to K(\X)\to K(\X,\Y)\] induced by the ECGW inclusion $\Y \to \X$. But in \cref{extensiveadditivity} we showed that under the same conditions on $\Y$ we have a homotopy equivalence $K(\X) \simeq K(\Y) \times K(\X - \Y)$, under which the map $K(\Y) \to K(\X)$ corresponds to the map $K(\Y) \to K(\Y) \times K(\X - \Y)$ induced by the $\fcgwa$ functor $\Y\to \Y\times (\X-\Y)$ mapping $Y\mapsto (Y,\varnothing)$. 
     
     We further claim that this is an equivalence at the level of spectra. To prove this, it suffices to show that $iS^n_\bullet \X$ is homotopy equivalent to $iS^n_\bullet\Y\times iS^n_\bullet (\X-\Y)$ for all $n$, as the delooping of $K(\X)$ is given by $\vert iS^n_\bullet\X\vert$ using the multi-simplicial $\fcgwa$ category $S^n_\bullet\X$; see \cref{spectra}. In turn, for this it is enough to show that for every $k$, fixing one simplicial direction yields a homotopy equivalence between $iS^{n-1}_\bullet(S_k\X)$ and $iS^{n-1}_\bullet(S_k\Y)\times iS^{n-1}_\bullet(S_k(\X-\Y))$. This now follows inductively from \cref{extensiveadditivity}, as one can verify that the fact that $\X$ is extensive and finitary and $\Y$ is Serre imply that the $\fcgwa$ category $S_k\X$ is itself extensive and finitary and that $S_k\Y$ is Serre.
     
    We thus get a second homotopy fiber sequence of spectra together with maps
    \[\begin{tikzcd}
        K(\Y)\dar[equal]\rar & K(\X)\dar["\simeq"] \rar & K(\X,\Y)\\
        K(\Y)\rar & K(\Y)\times K(\X-\Y)\rar & K(\X-\Y)
    \end{tikzcd}\]
   which gives a homotopy equivalence on cofibers, induced by $\pi_{\X-\Y} \colon \X \to \X-\Y$; note that this is homotopy inverse to the map on $K$-theory induced by the inclusion of the subcategory $\X - \Y$.
\end{proof}

\begin{ex}\label{polytopecofiber}
As a consequence of \cref{extensivelocalization}, we now have that the relative $\fcgwa$ category $K(\fGh_n,\fGh_{n-1})$ from \cref{relativepolytopes} is the homotopy cofiber of the map $K(\fGh_{n-1}) \to K(\fGh_n)$ induced by the inclusion, completing the proof that the cofiber of this map is $K(\cGh_n)$ (via \cref{polytopecompare}). In particular, note that $\fGh_{n-1}$ is Serre as coproduct inclusions are nondecreasing in dimensionality.
\end{ex}

\begin{ex}
In the theory of $R$-modules, $K$-theory classically refers to finitely generated projective modules (which agrees with the $K$-theory of finitely generated free modules away from $K_0$) while $G$-theory refers to the $K$-theory of all finitely generated modules. Similarly, the $K$-theory of all finite $H$-sets for a finite group $H$ could be considered its ``$G$-theory'' while that of finitely generated free $H$-sets the $K$-theory of $H$. By \cref{extensivelocalization} and \cref{gsetadditivity}, the cofiber of the map $K(H) \to G(H)$ is given by the $K$-theory of the full subcategory of $\finset_H$ generated under finite coproducts by the non-free transitive $H$-sets.
\end{ex}

\appendix

\section{Functoriality Constructions}\label{appendix}

In this appendix, we prove a number of technical results that are mostly unenlightening but, unfortunately, necessary. The main goal is to prove \cref{gridsfcgw} and \cref{Snfcgw} which say that $S_n\C$ and the w-grids $w_{l,m}\C$ of \cref{defnSn} and \cref{defngrids} are $\fcgwa$ categories.

\subsection{Properties of $\star$-pushouts}

We establish some technical results concerning $\star$-pushouts. All of the results in this subsection assume an $\fcgwa$ category.

\begin{lemma}\label{pushout:uniqueness}
For any good square in $\M$ as below inducing an isomorphism on cokernels, the induced map $B \star_A C \mrto D$ is an isomorphism.
    \begin{diagram}
    {A & B & B/A\\
    C & D & D/C\\};
    \mto{1-1}{1-2} \eto{1-3}{1-2}
    \mto{2-1}{2-2} \eto{2-3}{2-2}
    \mto{1-1}{2-1} \mto{1-2}{2-2} \mto{1-3}{2-3}^\cong
    \good{1-1}{2-2} \comm{1-2}{2-3}
    \end{diagram}
\end{lemma}

\begin{proof}
By the definition of $\star$-pushouts, we have the following diagram 
    \begin{diagram}
    {A & B & B/A\\
    C & B\star_A C & B\star_A C/C\\
    C & D & D/C\\};
    \mto{1-1}{1-2} \eto{1-3}{1-2}
    \mto{2-1}{2-2} \eto{2-3}{2-2}
    \mto{3-1}{3-2} \eto{3-3}{3-2}
    \mto{1-1}{2-1} \mto{1-2}{2-2} \mto{1-3}{2-3}^\cong \diagArrow{mmor,bend left, out=60,in=120}{1-3}{3-3}^\cong
    \eq{2-1}{3-1} \mto{2-2}{3-2} \mto{2-3}{3-3}
    \good{1-1}{2-2} \comm{1-2}{2-3}
    \good{2-1}{3-2} \dist{2-2}{3-3}
    \end{diagram}
where the map $B\star_A C/C \mrto D/C$ is an isomorphism as the composite $B / A \cong B \star_A C/C \mrto D / C$ is an isomorphism.  Then, since distinguished squares induce isomorphisms on cokernels, \cref{isosemptycoker} implies that the map $B \star_A C \mrto D$ is an isomorphism.
\end{proof}


\begin{cor}\label{composition:pushout:functoriality}
Given a diagram $C \mlto A \mrto B \mrto B'$, we have $B' \star_B (B \star_A C) \cong B' \star_A C$.  In other words, the composite of $\star$-pushouts below is the $\star$-pushout of the outer span.
    \begin{diagram}
    {A & B & B'\\
    C & B\star_A C & B'\star_B (B\star_A C)\\};
    \mto{1-1}{1-2} \mto{1-2}{1-3}
    \mto{2-1}{2-2} \mto{2-2}{2-3} 
    \mto{1-1}{2-1} \mto{1-2}{2-2}\mto{1-3}{2-3}
    \end{diagram}
\end{cor}
\begin{proof}
The induced map on cokernels of the vertical m-morphisms is a composite of isomorphisms, so by \cref{pushout:uniqueness} the composite is a $\star$-pushout.
\end{proof}

\begin{prop}\label{lemma2.14}
Given a black commutative diagram as below, where the top face is a good square, there exists an induced blue m-morphism between $\star$-pushouts such that the two squares created commute, and the bottom one is a good square
    \begin{squisheddiagram}
    {A & & A' & \\
     & B &  & B'\\
    C &  & C'  & \\
    &  B\star_A C & & B'\star_{A'} C'\\};
    \mto{1-1}{1-3}  \mto{3-1}{3-3} \mto{4-2}{4-4}
    \mto{1-1}{3-1}  \mto{1-3}{3-3} \mto{2-4}{4-4}
    \mto{1-1}{2-2} \mto{1-3}{2-4} \mto{3-1}{4-2} \mto{3-3}{4-4}
    \over{2-2}{2-4} \over{2-2}{4-2}
    \mto{2-2}{2-4} \mto{2-2}{4-2}
    \diagArrow{mmor,blue}{4-2}{4-4}
    \good{1-1}{2-4} 
    \diagArrow{-,color=white}{3-3}{4-2}!{\textcolor{blue}{\g}}
    \end{squisheddiagram}
Moreover, this assignment is functorial, and if all the original faces are good squares then the two squares created are good, and this is a good cube. The analogous statement for e-morphisms also holds, if both $\star$-pushouts exist.
\end{prop}
\begin{proof}
In order to obtain the desired blue m-morphism such that the two squares created commute, it suffices to note that the square 
    \begin{diagram}
    {A & & C\\
    A & A' & C'\\
    B & B' & B'\star_{A'} C'\\};
    \mto{1-1}{1-3}  
    \mto{2-1}{2-2} \mto{2-2}{2-3}
    \mto{3-1}{3-2} \mto{3-2}{3-3}
    \eq{1-1}{2-1} \mto{1-3}{2-3}
    \mto{2-1}{3-1} \mto{2-2}{3-2} \mto{2-3}{3-3}
    \good{1-1}{2-3} \good{2-2}{3-3} \good{2-1}{3-2}
    \end{diagram}
is good, and invoke the universal property of the $\star$-pushout $B\star_A C$. The bottom square is good by axiom (POL), and functoriality follows from uniqueness of the maps induced by the $\star$-pushout. Finally, if all faces are good, then this is a good cube, since the southern square is an identity square.
\end{proof}

\begin{prop}\label{lemma2.15}
Given a black diagram as below left, where all faces are either good squares or squares in the double category, there exists an induced blue e-morphism between $\star$-pushouts such that the two diagrams created are squares in the double category. 
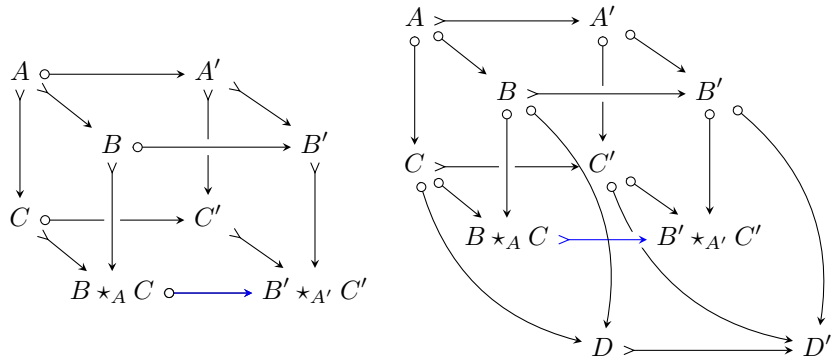
    \[\begin{squishedinlinediagram}
    {A & & A' & \\
     & B &  & B'\\
    C &  & C'  & \\
    &  B\star_A C & & B'\star_{A'} C'\\};
    \eto{1-1}{1-3}  \eto{3-1}{3-3} \eto{4-2}{4-4}
    \mto{1-1}{3-1}  \mto{1-3}{3-3} \mto{2-4}{4-4}
    \mto{1-1}{2-2} \mto{1-3}{2-4} \mto{3-1}{4-2} \mto{3-3}{4-4}
    \over{2-2}{4-2} \over{2-2}{2-4}
    \mto{2-2}{4-2} \eto{2-2}{2-4}
    \diagArrow{emor,blue}{4-2}{4-4}
    \end{squishedinlinediagram}
    \begin{squishedinlinediagram}
    {A & & A' &  &\\
     & B &  & B' &\\
    C &  & C'  & &\\
    &  B\star_A C & & B'\star_{A'} C' &\\
    & & & &\\
    & & D & & D'\\};
    \mto{1-1}{1-3}  \mto{3-1}{3-3} \mto{4-2}{4-4}
    \eto{1-1}{3-1}  \eto{1-3}{3-3} \eto{2-4}{4-4}
    \eto{1-1}{2-2} \eto{1-3}{2-4} \eto{3-1}{4-2} \eto{3-3}{4-4}
    \over{2-2}{2-4} \over{2-2}{4-2}
    \mto{2-2}{2-4} \eto{2-2}{4-2}
    \mto{6-3}{6-5} \diagArrow{emor, bend right}{3-1}{6-3} \diagArrow{emor, bend right}{3-3}{6-5}
    \over{4-2}{4-4}
    \diagArrow{mmor,blue}{4-2}{4-4}
    \diagArrow{emor, bend left}{2-2}{6-3} \diagArrow{emor, bend left}{2-4}{6-5}
    \end{squishedinlinediagram}\]
Moreover, this assignment is functorial, and if one of the  squares is distinguished, then so is the parallel new square. The analogous statement for e-morphisms also holds, if we start from a black diagram as above right.
\end{prop}
\begin{proof}
The constructions necessary for the proof are represented in the diagram below, where the black arrows are given in the data, and the ones we construct are dashed. We proceed to explain the steps in order.
    \begin{squisheddiagram}
    {A &  & A' & & A'\bs A &\\
     & B &  & B'& & B'\bs B\\
    C & &  C'  & & C'\bs C &\\
    &   \cok f & & B'\star_{A'} C' & & (B'\bs B)\star_{(A'\bs A)} (C'\bs C)\\
    C/A &  & C'/A' \\
    & \cok f /B & & B'\star_{A'} C'/B'\\};
    \eto{1-1}{1-3}  \eto{3-1}{3-3} 
    \mto{1-1}{3-1}  \mto{1-3}{3-3} 
    \mto{1-1}{2-2} \mto{1-3}{2-4}  \mto{3-3}{4-4}
    \over{2-2}{2-4}
    \eto{2-2}{2-4}
    \diagArrow{mmor,densely dashed}{1-5}{1-3}  \diagArrow{mmor,densely dashed}{3-5}{3-3}
    \diagArrow{mmor,densely dashed}{1-5}{2-6} \diagArrow{mmor,densely dashed}{3-5}{4-6} \diagArrow{mmor,densely dashed}{1-5}{3-5} \diagArrow{mmor,densely dashed}{2-6}{4-6}
    \over{2-6}{2-4}
    \diagArrow{mmor,densely dashed}{2-6}{2-4}
    \diagArrow{mmor,densely dashed}{4-6}{4-4}^f
    \over{2-2}{4-2}
    \diagArrow{mmor, densely dashed}{3-1}{4-2} \diagArrow{mmor,densely dashed}{2-2}{4-2} 
    \diagArrow{emor, densely dashed}{5-1}{3-1} \diagArrow{emor, densely dashed}{5-3}{3-3}
     \diagArrow{emor, densely dashed}{6-4}{4-4}
    \diagArrow{emor, densely dashed}{5-1}{5-3} \diagArrow{emor, densely dashed}{6-2}{6-4}
    \over{6-2}{4-2}
    \diagArrow{emor, densely dashed}{6-2}{4-2}
    \diagArrow{mmor, densely dashed}{5-1}{6-2} \diagArrow{mmor, densely dashed}{5-3}{6-4} 
    \over{4-2}{4-4}
    \diagArrow{emor, densely dashed}{4-2}{4-4}
    \over{2-4}{4-4}
    \mto{2-4}{4-4}
    \end{squisheddiagram}
    
First, consider the kernels of the given horizontal e-morphisms, and construct the $\star$-pushout of the induced span between them. By \cref{lemma2.14}, there exists an m-morphism \[(B'\bs B)\star_{(A'\bs A)} (C'\bs C)\mrto^f B'\star_{A'} C'\] such that all squares on the top right cube are good. 

We can now consider $\cok f$ and form the cube on the top left, which uses all of the original data except for $B \star_A C$, placing $\cok f$ in its stead. Note that all the squares in this cube are either good squares or squares in the double category (by construction, together with axiom (PBL)). 

Taking cokernels of the vertical m-morphisms yields the bottom left cube, where all squares are either good squares or squares in the double category (again by construction, together with axiom (PBL)). By definition of $B'\star_{A'} C'$, the map $C'/A'\mrto B'\star_{A'} C'/B'$ is an isomorphism. Then, by \cref{isosemptycoker}, the map $C/A\mrto \cok f/B$ is an isomorphism as well, and by \cref{pushout:uniqueness} we get that the induced m-morphism $B\star_A C\mrto \cok f$ must also be an isomorphism, which concludes the proof of the first statement.

Now suppose the given top square is distinguished. This implies that the map $A'\bs A\mrto B'\bs B$ is an isomorphism; then, so is $C'\bs C\mrto (B'\bs B)\star_{(A'\bs A)} (C'\bs C)$, and thus the bottom square of the top left cube must be distinguished as well.
\end{proof}

\begin{rmk}\label{rmkstarkercoker}
From the kernel-cokernel sequence 
\[B\star_{A} C\cong \cok f \erto B'\star_{A'} C' \mlto^f (B'\bs B)\star_{(A'\bs A)} (C'\bs C)\] constructed in the proof above, we see that the kernel of the induced e-morphism is precisely the $\star$-pushout of the kernels of the three given e-morphisms in the data.
\end{rmk}

\begin{lemma}\label{southern:cokernel}
Given a good square between objects $A,B,C,D$ as in the diagram below, where $\star$ denotes $B \star_A C$, the maps in blue form a kernel-cokernel pair. 
    \begin{general-diagram}{0.7em}{0.7em}
    {A & & & B & &\hspace{0.2cm} & B / A \\ 
    \\ 
    & & \hspace{0.1cm} \star \\
    C & & & D & & & D / C \\ 
    \vspace{.2cm}\\ 
    \vspace{.2cm}\\
    C \bs A & & & D \bs B & & & \bullet \\};
    \mto{1-1}{1-4} \mto{4-1}{4-4} \mto{7-1}{7-4}
    \mto{1-1}{4-1} \mto{1-4}{4-4} \mto{1-7}{4-7}
    \eto{7-1}{4-1} \eto{7-4}{4-4} \eto{7-7}{4-7}
    \eto{1-7}{1-4} \eto{4-7}{4-4} \eto{7-7}{7-4}
    \mto{4-1}{3-3} \mto{1-4}{3-3}
    \diagArrow{mmor, blue}{3-3}{4-4} \diagArrow{emor, blue}{7-7}{4-4}
    \comm{7-1}{4-4} \comm{1-7}{4-4}
    \good{1-1}{4-4} \over{4-7}{7-4} \good{4-7}{7-4}
    \end{general-diagram}
\end{lemma}

\begin{proof}
First, note that both maps are unique, as the blue m-morphism is the unique map from the $\star$-pushout from axiom (PO), and the blue e-morphism is the composite of the good square formed by applying $k^{-1}$ followed by $c$ to the original good square, equivalently in either direction by \cref{mixedpullbackuniqueness}.

Now, we can factor the left column of the diagram above as below left:
    \[\begin{inline-diagram}
    {A &  B &  B &  B / A \\ 
    C &  \star &  D &   D / C \\ 
    C/ A  & \star / B  & D/ B  & \bullet \\};
    \mto{1-1}{1-2} \eqto{1-2}{1-3} \eto{1-4}{1-3}
    \mto{2-1}{2-2} \diagArrow{mmor, blue}{2-2}{2-3} \eto{2-4}{2-3}
    \mto{3-1}{3-2}_\cong \mto{3-2}{3-3}  \eto{3-4}{3-3}
    \mto{1-1}{2-1} \mto{1-2}{2-2} \mto{1-3}{2-3} \mto{1-4}{2-4}
    \eto{3-1}{2-1} \eto{3-2}{2-2}\eto{3-3}{2-3} \eto{3-4}{2-4}
    \diagArrow{emor, blue}{3-4}{2-3}
    \good{1-1}{2-2} \good{1-2}{2-3} \comm{1-3}{2-4}
    \comm{2-1}{3-2} \dist{2-2}{3-3} \good{3-4}{2-3}
    \end{inline-diagram} \qquad
    \begin{inline-diagram}
    {\star & D & D / \star \\
    \star / B & D / B & (D / B) / (\star / B) \\
    C / A & D / B & \bullet \\};
    \diagArrow{mmor, blue}{1-1}{1-2} \mto{2-1}{2-2} \mto{3-1}{3-2}
    \eto{1-3}{1-2} \eto{2-3}{2-2} \diagArrow{emor, blue}{3-3}{3-2}
    \eto{2-1}{1-1} \diagArrow{emor, blue}{2-2}{1-2} \eto{2-3}{1-3}_\cong
    \eto{3-1}{2-1}^\cong \diagArrow{implies,-, blue}{3-2}{2-2} \eto{3-3}{2-3}_\cong
    \dist{2-1}{1-2} \dist{3-1}{2-2} \good{3-3}{2-2} \good{2-3}{1-2}
    \end{inline-diagram}\]
We then have the diagram of horizontal kernel-cokernel pairs above right, where the lower square is a square in the double category by \cref{sharedisos} and distinguished by \cref{isosemptycoker}. Therefore, $D/ \star \cong \bullet$, so by 
$\star \mrto D \elto \bullet$ is a kernel-cokernel sequence.
\end{proof}

Let us say a cube is an m-m-e cube if it has m-morphisms in two directions and e-morphisms in the remaining direction; similarly, we have e-e-m cubes, m-m-m cubes, etc.  

\begin{prop}\label{cubecorrespondence}
Given a good m-m-m cube, taking cokernels of the m-morphisms and squares in any of the three directions produces an m-m-e cube whose faces are all good squares or squares in the double category. Conversely, given such an m-m-e cube, taking kernels produces a good m-m-m cube. The same is also true with the roles of m- and e-morphisms reversed.
\end{prop}

\begin{proof}
Consider a good m-m-m cube, whose faces and a choice of southern square 
are all good squares, and let $\star,\star'$ denote the $\star$-pushouts of the relevant spans. We first take cokernels in the direction of the southern square, as pictured below. 
    \begin{general-diagram}{.225em}{.04em}
    {A &\hspace{0.2cm}&&&&&& B \\
    &&&&&&& \vspace{0.5cm}&&&&\\
	\vspace{0.5cm}&&&&&&& &&&&\\
	&&&&&&&\vspace{0.5cm} &&&&\\
	&&&& \star  &\hspace{0.2cm}&\hspace{0.2cm}&&&&&\\
	&& A' &&&&&&& B' \\
	&&&&&&&\vspace{0.5cm} &&&&\\
	C &&&&&&& D &&&&\\
	&&&&&&& &&&&\\
    &&&\hspace{0.2cm}&&& \star' &&&&&\\
	&&&& A' / A &&&\vspace{0.5cm} &&&& B' / B \\
	&&&&&&&\vspace{0.5cm} &&&&\\
	&& C' &&&&&&& D' &&\\
	&&&&&&&&&&&\\
	&&&&&&&& \star' / \star &&&\\
	&&&&&&& \vspace{0.5cm}&&&&\\
	&&&&&&& \vspace{0.5cm}&&&&\\
	&&&& C' / C &&&\vspace{0.5cm} &&&\hspace{0.2cm}& D' / D\\}; 
	\mto{1-1}{1-8} \mto{8-1}{8-8} 
	\mto{1-1}{6-3} \mto{1-8}{6-10} 
	\mto{8-1}{13-3} \mto{8-8}{13-10}
	\mto{1-1}{8-1} \over{6-3}{13-3}\mto{6-3}{13-3} \mto{1-8}{8-8} 
    \mto{1-8}{5-5} \diagArrow{mmor, dash pattern=on 18pt off 18pt}{8-1}{5-5} \diagArrow{mmor, blue}{5-5}{8-8} \mto{6-10}{10-7} \diagArrow{mmor, dash pattern=on 25pt off 18pt}{13-3}{10-7} \over{10-7}{13-10}\diagArrow{mmor, blue}{10-7}{13-10}
    \over{5-5}{10-7} \diagArrow{mmor, blue}{5-5}{10-7}
    \over{6-3}{6-10} \mto{6-3}{6-10} 
    \over{6-10}{13-10}\mto{6-10}{13-10}
    \over{13-3}{13-10}\mto{13-3}{13-10}
    \over{15-9}{10-7} \diagArrow{emor, blue}{15-9}{10-7} \diagArrow{mmor, blue}{15-9}{18-12}
    \over{11-5}{11-12} \over{11-5}{18-5}
    \mto{11-5}{11-12} \mto{11-5}{18-5}
    \mto{18-5}{18-12} \mto{11-12}{18-12}
    \over{11-5}{6-3} \eto{11-5}{6-3} \eto{11-12}{6-10}
    \eto{18-5}{13-3} \eto{18-12}{13-10}
    \mto{11-12}{15-9} \mto{18-5}{15-9}
    \end{general-diagram}
By \cref{rmkstarkercoker}, $\star' / \star$ is the $\star$-pushout of $B' / B \mlto A' / A \mrto C' / C$, so  \cref{goodsqpushout} ensures that the square involving $A'/A, \ B'/B, \ C'/C, \ D'/D$ is good. As all of the mixed squares in this m-m-e cube are squares in the double category by construction, we have showed that the cokernel cube in this direction is of the desired form.

We now take cokernels of the m-m-m cube in the remaining two directions, as depicted below. This diagram can be further completed by taking cokernels of the m-m-e cubes and producing the black dashed e-morphisms; note that both squares of e-morphisms created are good.
    \begin{megasquisheddiagram}
    {A &&&&&&& B &&&&\hspace{1.5cm}&&& B/A &&\\
    &&&&&&& \vspace{0.5cm}&&&&&&& &&\\
	\vspace{0.5cm}&&&&&&& &&&&&&& &&\\
	&&&&&&&\vspace{0.5cm} &&&&&&& &&\\
	&&&& \star  &\hspace{0.2cm}&&&&&&&&&\\
	&& A' &&&&&&& B' &&&&&&& B'/A' \\
	&&&&&&& &&&&&&& &&\\
	C &&&&&&& D &&&&&&& D/C &&\\
	&&&&&&& &&&&&&& &&\\
    &&&\hspace{0.2cm}&&& \star' &&&&&&&&&&\\
	&&&&&&&\vspace{0.5cm} &&&&&&& && \\
	&&&&&&&\vspace{0.5cm} &&&&&&& &&\\
	&& C' &&&&&&& D' &&&&&&& D'/C' \\
	&&&&&&& &&&&&&& &&\\
	C/A &&&&&&& D/B &&&&&&& \bullet &&\\
	&&&&&&& \vspace{0.5cm}&&&&&&& &&\\
	&&&&&&& \vspace{0.5cm}&&&&&&& &&\\
	&&&&&&&\vspace{0.5cm} &&&&&&& &&\\
	&&&&&&& \vspace{0.5cm}&&&&&&& &&\\
	&& C'/A' &&&&&&& D'/B' &&&&&&& \bullet'\\}; 
	\mto{1-1}{1-8} \eto{1-15}{1-8}
	\mto{8-1}{8-8} \eto{8-15}{8-8}
	\mto{1-1}{6-3} \mto{1-8}{6-10} \mto{1-15}{6-17}
	\mto{8-1}{13-3} \mto{8-8}{13-10} \mto{8-15}{13-17}
	\mto{15-1}{15-8} \diagArrow{emor, densely dashed}{15-15}{15-8}
	\mto{20-3}{20-10} \diagArrow{emor, densely dashed}{20-17}{20-10}
	\mto{1-1}{8-1} \over{6-3}{13-3}\mto{6-3}{13-3} \mto{1-8}{8-8} \mto{1-15}{8-15}  \mto{6-17}{13-17}
    \mto{1-8}{5-5} \diagArrow{mmor, dash pattern=on 18pt off 18pt}{8-1}{5-5} \diagArrow{mmor, blue}{5-5}{8-8} \diagArrow{emor, blue}{15-15}{8-8} 
    \eto{15-1}{8-1}  \eto{15-8}{8-8} \diagArrow{emor, densely dashed}{15-15}{8-15} \diagArrow{emor, densely dashed}{20-17}{13-17}
    \mto{15-1}{20-3} \mto{15-8}{20-10} \diagArrow{mmor, densely dashed, blue}{15-15}{20-17}
    \mto{6-10}{10-7} \mto{13-3}{10-7} \over{10-7}{13-10}\diagArrow{mmor, blue}{10-7}{13-10}
    \over{5-5}{10-7} \diagArrow{mmor, blue}{5-5}{10-7}
    \over{6-3}{6-10} \mto{6-3}{6-10} \over{6-17}{6-10} \eto{6-17}{6-10}
    \over{6-10}{13-10}\mto{6-10}{13-10}
    \over{13-3}{13-10}\mto{13-3}{13-10} \over{13-17}{13-10}\eto{13-17}{13-10}
    \over{20-3}{13-3}\eto{20-3}{13-3} \over{20-10}{13-10} \eto{20-10}{13-10}
    \over{20-17}{13-10} \diagArrow{emor, blue}{20-17}{13-10};
    \end{megasquisheddiagram}

Now,  these m-m-e cubes are such that their remaining face is a good square if and only if there exists an induced dashed blue m-morphism as in the picture such that the diagram $$\bullet, \ \bullet', \ D, \ D'$$ is a square in the double category. Indeed, the diagram with vertices $$B/A, \ B'/A', \ D/C, \ D'/C'$$ is a good square if and only if  taking its cokernel produces the induced dashed blue m-morphism such that the diagram $$\bullet,\ \bullet', \ D/C, \ D'/C'$$ is a square in the double category. This, by axiom (PBL), is equivalent to the diagram $$\bullet, \ \bullet', \ D, \ D'$$ being a square, which again by axiom (PBL) is equivalent to the diagram $$\bullet,\ \bullet', \ C'/A', \ D'/B'$$ being a square. But that,  in turn, happens if and only if its kernel square $$C/A, \ D/B, \ C'/A',\  D'/B'$$ is good.

Finally, as $\star$ denotes $B\star_A C$ and $\star'$ denotes $B'\star_{A'} C'$, the existence of the induced dashed blue m-morphism such that the diagram $$\bullet, \ \bullet', \ D, \ D'$$ is a square in the double category is equivalent to the southern square of the m-m-m cube being good, since these squares form a kernel-cokernel pair by \cref{southern:cokernel}.

For the converse, to show that the kernel of an m-m-e cube with all faces good squares or squares in the double category is always good, first observe that given such an m-m-e cube pictured as the lower left cube in the diagram above, taking cokernels we get the lower right cube with all faces good squares or squares in the double category, either by construction or in the case of the rightmost face by axiom (PBL). This shows, by \cref{southern:cokernel}, that in the kernel m-m-m cube pictured as the top left cube in the diagram, the southern square is good. 

It then remains only to show that the topmost square of the m-m-m cube is good. This follows by constructing the top right m-m-e cube as the kernel of the bottom right cube. Its topmost diagram is a square in the double category  by axiom (PBL), and forms a kernel-cokernel pair with the topmost square of the m-m-m cube, which is therefore good.
\end{proof}

\begin{rmk}\label{goodcubesymmetric}
In particular, this implies that there is no need to specify a direction for the good southern square when dealing with good cubes, as claimed in \cref{goodcubessymmetric:claim}, since the ``goodness'' of an m-m-m cube can be equivalently determined from any of its m-m-e cokernel cubes.
\end{rmk}

We can further deduce the following, which can be interpreted as the statement that all m-m-e and e-e-m cubes whose faces are good squares and squares in the double category are ``good cubes''.

\begin{cor}\label{MDhhv}
Consider an m-m-e cube whose faces are either good squares or squares in the double category, together with the induced cube to the $\star$-pushouts as constructed in \cref{lemma2.15}, depicted below left. Then the diagram below right is a square.
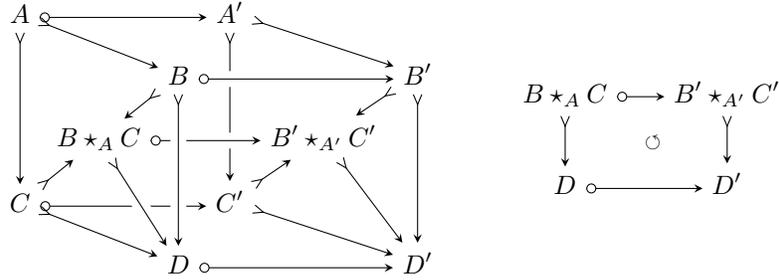
   \[\begin{goodcubeinlinediagram}
    {A & & & A' & & \\
    & & B & & & B'\\
    & B\star_A C & & & B'\star_{A'} C' & \\
    C & & & C' & & \\
    & & D & & & D'\\};
    \mto{1-1}{4-1}  \mto{1-4}{4-4} \mto{2-6}{5-6}
    \mto{2-3}{3-2} \mto{4-1}{3-2} 
    \over{2-3}{2-6}
    \eto{1-1}{1-4} \eto{2-3}{2-6} \eto{4-1}{4-4} \eto{5-3}{5-6}
    \over{3-2}{3-5}
    \diagArrow{emor}{3-2}{3-5}
    \mto{1-1}{2-3} \mto{1-4}{2-6} \mto{4-1}{5-3} \mto{4-4}{5-6}
    \over{2-3}{5-3}
    \mto{2-3}{5-3}
    \mto{2-6}{3-5} \mto{4-4}{3-5} \mto{3-5}{5-6}
    \over{3-2}{5-3}
    \mto{3-2}{5-3}
    \end{goodcubeinlinediagram}\qquad
    \begin{inline-diagram}
    {B\star_A C & B'\star_{A'} C'\\
    D & D'\\};
    \mto{1-1}{2-1} \mto{1-2}{2-2}
    \eto{2-1}{2-2} \diagArrow{emor}{1-1}{1-2}
    \comm{1-1}{2-2}
    \end{inline-diagram}\]
The analogous statement holds for e-e-m cubes when the $\star$-pushouts exist.
\end{cor}

By analogy with m-m-m cubes, we call this square the \emph{southern square} of the m-m-e cube.

\begin{proof}
The kernel of the outer cube is a good m-m-m cube by \cref{cubecorrespondence}, so 
the statement is easily deduced from \cref{goodcubesymmetric} together with the first picture in the proof of \cref{cubecorrespondence}.
\end{proof}

\begin{ex}
This result illustrates an interesting difference between our motivating examples. In a weakly idempotent complete exact category, where  squares are simply commuting squares between admissible monomorphisms and epimorphisms, 
this follows immediately from the  
universal property of the pushout. In finite sets, however, where the squares are pullbacks, this result is precisely the distributivity of intersections over unions among subsets of $D'$: $D \cap (B' \cup C') = (D \cap B') \cup (D \cap C')$. 
\end{ex}

We now show that $\star$-pushouts preserve squares and distinguished squares.

\begin{prop}\label{lemma2.16}
Given an m-span of squares in the double category, where all the other mixed squares involved are squares in the double category and the squares in one of the cube-legs of the span are good, the induced diagram between the $\star$-pushouts is a square in the double category.
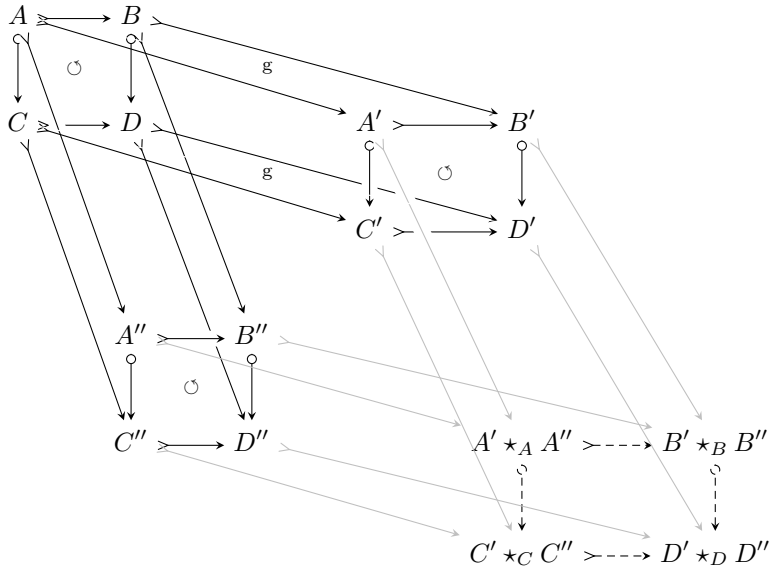
    \begin{diagram}
    {A & B & & & &\\
    C & D & & A' & B' & \\
    & & & C' & D' & \\
    & A'' & B'' & &  &\\
    & C'' & D'' & & A'\star_A A'' & B'\star_B B''\\
    & & & & C'\star_C C'' & D'\star_D D''\\};
    \mto{1-1}{1-2}\mto{2-1}{2-2} \eto{1-1}{2-1} \eto{1-2}{2-2}
    \mto{2-4}{2-5} \mto{3-4}{3-5}  \eto{2-5}{3-5}
     \mto{5-2}{5-3} \eto{4-2}{5-2} \eto{4-3}{5-3}
     \diagArrow{mmor, densely dashed}{5-5}{5-6}  \diagArrow{mmor, densely dashed}{6-5}{6-6}  \diagArrow{emor, densely dashed}{5-5}{6-5}  \diagArrow{emor, densely dashed}{5-6}{6-6}
    \over{1-1}{4-2}
    \mto{1-1}{4-2} \mto{2-1}{5-2}\mto{1-2}{4-3} \mto{2-2}{5-3} 
    \over{4-2}{4-3} \mto{4-2}{4-3}
    \mto{1-1}{2-4} \mto{1-2}{2-5}\mto{2-1}{3-4}\mto{2-2}{3-5}
    \over{2-4}{3-4} \eto{2-4}{3-4}
    \over{2-4}{5-5}
    \diagArrow{mmor, gray!50}{2-4}{5-5} \diagArrow{mmor, gray!50}{2-5}{5-6} \diagArrow{mmor, gray!50}{3-4}{6-5} \diagArrow{mmor, gray!50}{3-5}{6-6}
    \diagArrow{mmor, gray!50}{4-2}{5-5} \diagArrow{mmor, gray!50}{4-3}{5-6} \diagArrow{mmor, gray!50}{5-2}{6-5} \diagArrow{mmor, gray!50}{5-3}{6-6}
    \comm{1-1}{2-2} \comm{2-4}{3-5} \comm{4-2}{5-3}
    \good{1-1}{2-5} \good{2-1}{3-5}
    \end{diagram}
    The same statement holds for e-spans when the $\star$-pushouts exist.
\end{prop}
\begin{proof}
The gray and dashed m-morphisms are obtained from applying \cref{lemma2.14} to the diagrams of m-morphisms on the ``top'' and ``bottom'' rows respectively in the diagram above. In turn, the dashed e-morphism $A'\star_A A'' \erto C'\star_C C''$ is obtained by applying \cref{lemma2.15} to the sub-diagram involving the objects $$A,\ C,\ A',\ C',\ A'',\ C'',\ A'\star_A A'',\ C'\star_C C''.$$ Similarly, we get a map $B'\star_B B'' \erto D'\star_D D''$. 

The result then follows from applying \cref{MDhhv} to the following cube of good squares and squares in the double category, where the resulting  southern square is precisely the desired induced square of $\star$-pushouts.
    \begin{general-diagram}{0.83em}{.1em}
    {A & & & C & & \\
    & & A'' & & & C''\\
    & A' \star_A A'' & & & C' \star_C C'' & \\
    A' & & & C' & & \\
    & & B' \star_B B'' & & & D' \star_D D'' \\};
    \mto{1-1}{4-1}  \mto{1-4}{4-4} \mto{2-6}{5-6}
    \mto{2-3}{3-2} \mto{4-1}{3-2} 
    \over{2-3}{2-6}
    \eto{1-1}{1-4} \eto{2-3}{2-6} \eto{4-1}{4-4} \eto{5-3}{5-6}
    \over{3-2}{3-5}
    \diagArrow{emor}{3-2}{3-5}
    \mto{1-1}{2-3} \mto{1-4}{2-6} \mto{4-1}{5-3} \mto{4-4}{5-6}
    \over{2-3}{5-3}
    \mto{2-3}{5-3}
    \mto{2-6}{3-5} \mto{4-4}{3-5} \mto{3-5}{5-6}
    \over{3-2}{5-3}
    \mto{3-2}{5-3}
    \end{general-diagram}
    \end{proof}

\begin{prop}\label{lemma2.16distinguished}
If the three initial squares in \cref{lemma2.16} are distinguished, then so is the induced square between the $\star$-pushouts.
\end{prop}
\begin{proof}
By \cref{lemma2.16}, we know that the diagram between the $\star$-pushouts is a square. To show it is distinguished, first consider the particular case where $A=A'=A''=\varnothing$; note that then we have $A'\star_A A''=\varnothing$. In this case, we see that $C\mrto D$ is the kernel of $B\erto D$ (and similarly for the other two distinguished squares). Then, by \cref{rmkstarkercoker}, $C'\star_C C'' \mrto D'\star_D D''$ must be the kernel of $B'\star_B B'' \erto D'\star_D D''$, which shows that the desired square is distinguished.

For the general case, we paste distinguished squares besides the given squares as follows
    \[\begin{inline-diagram}
    {\varnothing & A & B\\
    C\bs A & C & D\\};
    \mto{1-1}{1-2} \mto{1-2}{1-3} 
    \mto{2-1}{2-2} \mto{2-2}{2-3}
    \eto{1-1}{2-1} \eto{1-2}{2-2} \eto{1-3}{2-3}
    \dist{1-1}{2-2} \dist{1-2}{2-3}
    \end{inline-diagram}
    \begin{inline-diagram}
    {\varnothing & A' & B'\\
    C'\bs A' & C' & D'\\};
    \mto{1-1}{1-2} \mto{1-2}{1-3} 
    \mto{2-1}{2-2} \mto{2-2}{2-3}
    \eto{1-1}{2-1} \eto{1-2}{2-2} \eto{1-3}{2-3}
    \dist{1-1}{2-2} \dist{1-2}{2-3}
    \end{inline-diagram}
    \begin{inline-diagram}
    {\varnothing & A'' & B''\\
    C''\bs A'' & C'' & D''\\};
    \mto{1-1}{1-2} \mto{1-2}{1-3} 
    \mto{2-1}{2-2} \mto{2-2}{2-3}
    \eto{1-1}{2-1} \eto{1-2}{2-2} \eto{1-3}{2-3}
    \dist{1-1}{2-2} \dist{1-2}{2-3}
    \end{inline-diagram}\]
    and obtain a diagram between $\star$-pushouts 
    \begin{diagram}
    {\varnothing & A'\star_A A'' & B'\star_B B''\\
    (C'\bs A')\star_{(C\bs A)} (C''\bs A'') & C'\star_C C'' & D'\star_D D''\\};
    \mto{1-1}{1-2} \mto{1-2}{1-3} 
    \mto{2-1}{2-2} \mto{2-2}{2-3}
    \eto{1-1}{2-1} \eto{1-2}{2-2} \eto{1-3}{2-3}
    \dist{1-1}{2-2} 
    \end{diagram}
The particular case guarantees that both the left square and the composite are distinguished. Then,  the desired square on the right is also distinguished, due to the 2-out-of-3 property of isomorphisms for the induced maps between
the kernels of the e-morphisms.
\end{proof}

\subsection{$\fcgwa$ categories of functors}\label{appendixpart2}

The aim of this subsection is to show that double categories of functors over an $\fcgwa$ category $\C$ admit an $\fcgwa$ structure themselves. In particular, this allows us to restrict to the special cases of interest: the double categories of staircases $S_n\C$ and the double categories of w-grids $w_{l,m}\C$.

\begin{thm}\label{functorsfcgw}
For any $\fcgwa$ category $\C$ and double category $\D$, the double category $\C^\D$ with structure described in \cref{functordblcat} and \cref{functorfcgw} is an $\fcgwa$ category.
\end{thm}
\begin{proof}
We begin by checking the conditions in \cref{preFCGW}. First of all, note that $\C^\D$ is a double category with shared isomorphisms, since these are defined pointwise, and $\C$ has shared isomorphisms.

To see that good squares are well-defined according to \cref{defn:classofgoodsquares}, first note that good squares in $\C^\D$ are composed of faces which are good squares in $\C$; in particular, all faces are pullbacks in $\C$, and so they are pullbacks in $\C^\D$. 
 We must then verify that $\Ar_\triangle \M\subseteq\Ar_{\g} \M$, for which it suffices to note that the southern square of a cube in $\Ar_\triangle \M$ agrees (up to isomorphism) with one of the faces of the cube, which is a good square.
    \begin{goodcubediagram}
    {A & & & A' & & \\
    & & B & & & B'\\
    & B\star_A C & & & B'\star_{A'} C' & \\
    C & & & C' & & \\
    & & D & & & D'\\};
    \mto{1-1}{4-1}  \mto{1-4}{4-4} \mto{2-6}{5-6}
    \mto{2-3}{3-2} \mto{4-1}{3-2}!\cong 
    \over{2-3}{2-6}
    \mto{1-1}{1-4} \mto{2-3}{2-6} \mto{4-1}{4-4} \mto{5-3}{5-6}
    \over{3-2}{3-5}
    \diagArrow{mmor}{3-2}{3-5}
    \mto{1-1}{2-3}!\cong \mto{1-4}{2-6}!\cong \mto{4-1}{5-3} \mto{4-4}{5-6}
    \over{2-3}{5-3}
    \mto{2-3}{5-3}
    \mto{2-6}{3-5} \mto{4-4}{3-5}!\cong \mto{3-5}{5-6}
    \over{3-2}{5-3}
    \mto{3-2}{5-3}
    \end{goodcubediagram}

We now show that $k:\Ar_{\circlearrowleft} \E\to \Ar_{\g} \M$ is well-defined; the argument for $c$ proceeds analogously. To see that $k$ takes an object in $\Ar_{\circlearrowleft} \E$ to an object in $\Ar_{\g} \M$, we must check that taking pointwise kernels of an e-natural transformation $\eta: A\Rightarrow B$ whose squares between e-morphisms are good produces a functor $C\in\C^\D$, together with an m-natural transformation $\mu: C\Rightarrow B$ whose squares between m-morphisms are good. 

For an object $i\in\D$, $C_i$ and $\mu_i$ are defined as the kernel of $\eta_i :A_i\erto B_i$. For an m-morphism $f: i\mrto j$ in $\D$, let $Cf$ be the induced morphism on kernels
    \begin{diagram}
    {A_i & B_i & C_i\\
    A_j & B_j & C_j\\};
    \mto{1-1}{2-1}_{Af} \mto{1-2}{2-2}^{Bf} \diagArrow{mmor, densely dashed}{1-3}{2-3}
    \eto{1-1}{1-2}^{\eta_i} \mto{1-3}{1-2} 
    \eto{2-1}{2-2}_{\eta_j} \mto{2-3}{2-2}
    \diagArrow{white}{1-1}{2-2}!{\textcolor{black}{\eta_f}}
    \end{diagram}
where the  square on the left exists since $\eta$ is an e-natural transformation. Similarly, given an e-morphism $g: i\erto j$ in $\D$, let $Cg$ be the induced morphism on kernels
    \begin{diagram}
    {A_i & B_i & C_i\\
    A_j & B_j & C_j\\};
    \eto{1-1}{2-1}_{Ag} \eto{1-2}{2-2}^{Bg} \diagArrow{emor, densely dashed}{1-3}{2-3}
    \eto{1-1}{1-2}^{\eta_i} \mto{1-3}{1-2} 
    \eto{2-1}{2-2}_{\eta_j} \mto{2-3}{2-2}
    \end{diagram}
and $\mu_g$ be the induced  square on the right, where the diagram on the left commutes by naturality of $\eta$, and is a good square by the additional assumption on $\eta$.

Finally, we must check that taking pointwise kernels of the leftmost cube below (whose faces are all good squares or squares in the double category) produces a cube as the one on the right (whose faces are all good squares or squares in the double category).
\begin{symsquisheddiagram}
    {A_{i} & & B_{i} & & C_{i} & \\
    & A_{j} & & B_{j}&  & C_{j}\\
    A_{k} & & B_{k}&  & C_{k} & \\
    & A_{l} & & B_{l} & & C_{l}\\};
    \mto{1-1}{2-2} \mto{1-3}{2-4} \mto{1-5}{2-6}
    \mto{3-1}{4-2} \mto{3-3}{4-4} \mto{3-5}{4-6}
    \eto{1-1}{1-3} \eto{3-1}{3-3}
    \mto{1-5}{1-3} \mto{3-5}{3-3}
    \eto{3-1}{1-1}  \eto{3-3}{1-3} \eto{3-5}{1-5}
    \over{4-2}{2-2} \over{4-4}{2-4}
    \eto{4-2}{2-2} \eto{4-4}{2-4} \eto{4-6}{2-6}
    \over{2-6}{2-4}
    \mto{2-6}{2-4} \mto{4-6}{4-4}
    \over{2-2}{2-4}
    \eto{2-2}{2-4} \eto{4-2}{4-4}
    \end{symsquisheddiagram}
Most of these faces are of the correct type by construction; indeed, the only face one needs to check is the rightmost diagram between the $C$'s, which is a square by axiom (PBL). The fact that $k$ takes a morphism in $\Ar_{\circlearrowleft} \E$ to a morphism in $\Ar_{\g} \M$ is further ensured by \cref{cubecorrespondence}.

Since $k$ is defined pointwise from the kernel functor in $\C$, it is clear that it is faithful. Furthermore, the fact that $k$ and $c$ are inverses on objects up to isomorphism, together with \cref{cubecorrespondence}, show that $k$ is essentially surjective and full. 

Axioms (Z) and (M) are trivially satisfied, since m- and e-morphisms in $\C^\D$ are pointwise m- and e-morphisms in $\C$. 
Axioms (D) and (K) are immediate, since the functors $k$ and $c$ are defined pointwise. 

Axiom (PBL) is satisfied, since a diagram in $\C^\D$ is a square in the double category precisely if it is pointwise a square in $\C$. For axiom ($\star$), given a span of m-morphisms $B\mlto A \mrto C$ in $\C^\D$, we can construct their pointwise $\star$-pushots using axiom ($\star$) for $\C$. By \cref{lemma2.14,lemma2.15}, $\star$-pushouts preserve m- and e-morphisms in the appropriate manner. Furthermore, by \cref{lemma2.16}, they preserve  squares. Thus, pointwise $\star$-pushouts are double functors $\D\to \C$.

\cref{lemma2.14,lemma2.15} also imply that the induced maps $B\to B\star_A C$ and $C\to B\star_A C$ are m-morphisms in $\C^\D$, and that the square below is good.
    \begin{diagram}
    {A & B\\
    C & B\star_A C\\};
    \mto{1-1}{1-2} \mto{1-1}{2-1}
    \mto{2-1}{2-2} \mto{1-2}{2-2}
    \end{diagram}
Similarly, we can construct the $\star$-pushout of a span of e-morphisms $B\elto A \erto C$ in $\C^\D$ when we already know the span is part of some good square.

It remains to show the universal property in axiom (PO), since $\star$-pushouts will preserve (co)kernels as $\star$, $k$, and $c$ are all defined pointwise. Consider a good square in $\C^\D$ as below left.
    \[\begin{inline-diagram}
        {A & B\\
        C & D\\};
        \mto{1-1}{1-2} \mto{1-1}{2-1}
        \mto{2-1}{2-2} \mto{1-2}{2-2}
        \good{1-1}{2-2}
    \end{inline-diagram}\qquad
    \begin{inline-diagram}
        {A_i & B_i\\
        C_i & D_i\\};
        \mto{1-1}{1-2} \mto{1-1}{2-1}
        \mto{2-1}{2-2} \mto{1-2}{2-2}
        \good{1-1}{2-2}
    \end{inline-diagram}\]
In particular, for each $i\in \D$ we have a good square in $\C$ as above right, which induce pointwise maps $B_i \star_{A_i} C_i\mrto D_i$, which are unique up to unique isomorphism. We need to show that for each $i\mrto j$ and $i\erto j$ in $\D$, the induced diagrams below are either good squares or squares in the double category.
    \[\begin{inline-diagram}
        {B_i\star_{A_i} C_i & D_i\\
        B_j\star_{A_j} C_j & D_j\\};
        \mto{1-1}{1-2} \mto{1-1}{2-1}
        \mto{2-1}{2-2} \mto{1-2}{2-2}
    \end{inline-diagram}\qquad
    \begin{inline-diagram}
        {B_i\star_{A_i} C_i & D_i\\
        B_j\star_{A_j} C_j & D_j\\};
        \mto{1-1}{1-2} \eto{1-1}{2-1}
        \mto{2-1}{2-2} \eto{1-2}{2-2}
    \end{inline-diagram}\]
    
For the first statement, note that the diagram above left is the southern square of the cube
    \begin{goodcubediagram}
    {A_i & & & A_j & & \\
    & & B_i & & & B_j\\
    & B_i\star_{A_i} C_i & & & B_j\star_{A_j} C_j & \\
    C_i & & & C_j & & \\
    & & D_i & & & D_j\\};
    \mto{1-1}{4-1}  \mto{1-4}{4-4} \mto{2-6}{5-6}
    \mto{2-3}{3-2} \mto{4-1}{3-2} 
    \over{2-3}{2-6}
    \mto{1-1}{1-4} \mto{2-3}{2-6} \mto{4-1}{4-4} \mto{5-3}{5-6}
    \over{3-2}{3-5}
    \diagArrow{mmor}{3-2}{3-5}
    \mto{1-1}{2-3} \mto{1-4}{2-6} \mto{4-1}{5-3} \mto{4-4}{5-6}
    \over{2-3}{5-3}
    \mto{2-3}{5-3}
    \mto{2-6}{3-5} \mto{4-4}{3-5} \mto{3-5}{5-6}
    \over{3-2}{5-3}
    \mto{3-2}{5-3}
     \end{goodcubediagram}
which was assumed to be a good cube; thus, the diagram must be a good square. For the second, note that the diagram above right is the ``southern square'' of the cube
    \begin{goodcubediagram}
    {A_i & & & A_j & & \\
    & & B_i & & & B_j\\
    & B_i\star_{A_i} C_i & & & B_j\star_{A_j} C_j & \\
    C_i & & & C_j & & \\
    & & D_i & & & D_j\\};
    \mto{1-1}{4-1}  \mto{1-4}{4-4} \mto{2-6}{5-6}
   \mto{2-3}{3-2} \mto{4-1}{3-2} 
    \over{2-3}{2-6}
    \eto{1-1}{1-4} \eto{2-3}{2-6} \eto{4-1}{4-4} \eto{5-3}{5-6}
    \over{3-2}{3-5}
    \diagArrow{emor}{3-2}{3-5}
    \mto{1-1}{2-3} \mto{1-4}{2-6} \mto{4-1}{5-3} \mto{4-4}{5-6}
    \over{2-3}{5-3}
    \mto{2-3}{5-3}
    \mto{2-6}{3-5} \mto{4-4}{3-5} \mto{3-5}{5-6}
    \over{3-2}{5-3}
    \mto{3-2}{5-3}
    \end{goodcubediagram}
which, by \cref{MDhhv}, is always a square in the double category.

Finally, for axiom (POL), it suffices to check that in any diagram
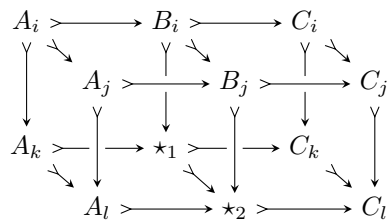
\begin{symsquisheddiagram}
    {A_{i} & & B_{i} & & C_{i} & \\
    & A_{j} & & B_{j}&  & C_{j}\\
    A_{k} & & \star_1 &  & C_{k} & \\
    & A_{l} & & \star_2 & & C_{l}\\};
    \mto{1-1}{2-2} \mto{1-3}{2-4} \mto{1-5}{2-6}
    \mto{3-1}{4-2} \mto{3-3}{4-4} \mto{3-5}{4-6}
    \mto{1-1}{1-3} \mto{3-1}{3-3}
    \mto{1-3}{1-5}\mto{3-3}{3-5}
    \mto{1-1}{3-1}  \mto{1-3}{3-3} \mto{1-5}{3-5}
    \over{4-2}{2-2} \over{4-4}{2-4}
    \mto{2-2}{4-2} \mto{2-4}{4-4} \mto{2-6}{4-6}
    \over{2-4}{2-6}
    \mto{2-4}{2-6} \mto{4-4}{4-6}
    \over{2-2}{2-4}
    \mto{2-2}{2-4} \mto{4-2}{4-4}
    \end{symsquisheddiagram} 
whose outer cube is good, the right cube must be good. Here $\star_1$ denotes $B_i\star_{A_i} A_k$, and $\star_2$ denotes $B_j\star_{A_j} A_l$. Indeed, the back and front faces of the right cube must be good squares due to $\C$ satisfying axiom (POL), and the southern square of the right cube can easily be seen to agree with the southern square of the outer cube, which is good.
\end{proof}

We can further show that we get an $\fcgwa$ structure when restricting the squares in our $\D$-shaped diagrams to be distinguished in $\C$ and requiring certain objects in $\D$ to be sent to $\varnothing$, as in the double subcategory $S_n\C \subset \C^{\mathcal{S}_n}$ of \cref{defnSn}.

\begin{prop}\label{Snfcgw}
$S_n\C$ is an $\fcgwa$ subcategory of $\C^{\mathcal{S}_n}$.
\end{prop}

\begin{proof}
By \cref{fullsubFCGW}, in order to show that this is an $\fcgwa$ subcategory, it suffices to prove that it is closed under $k$, $c$, $\star$, and that it contains the initial object. The latter is trivial, as any square whose boundary consists of isomorphisms is distinguished. Furthermore, since $k$, $c$ and $\star$ are computed pointwise, it is clear that they preserve the condition of sending the objects $A_{i,i}$ to $\varnothing$. It remains to show that each of these preserves distinguished squares.

We first show that $k$ preserves distinguished squares; for this, we show that in the following diagram, where the right cube is the kernel of the left one, the rightmost square is distinguished in $\C$.
    \begin{squisheddiagram}
    {A_{i,j} & & B_{i,j} & & C_{i,j} & \\
    & A_{i+1,j} & & B_{i+1,j}&  & C_{i+1,j}\\
    A_{i,j+1} & & B_{i,j+1}&  & C_{i,j+1} & \\
    & A_{i+1,j+1} & & B_{i+1,j+1} & & C_{i+1,j+1}\\};
    \mto{1-1}{2-2} \mto{1-3}{2-4} \mto{1-5}{2-6}
    \mto{3-1}{4-2} \mto{3-3}{4-4} \mto{3-5}{4-6}
    \eto{1-1}{1-3} \eto{3-1}{3-3}
    \mto{1-5}{1-3} \mto{3-5}{3-3}
    \eto{3-1}{1-1}  \eto{3-3}{1-3} \eto{3-5}{1-5}
    \over{4-2}{2-2} \over{4-4}{2-4}
    \eto{4-2}{2-2} \eto{4-4}{2-4} \eto{4-6}{2-6}
    \over{2-6}{2-4}
    \mto{2-6}{2-4} \mto{4-6}{4-4}
    \over{2-2}{2-4}
    \eto{2-2}{2-4} \eto{4-2}{4-4}
    \dist{1-1}{4-2} \dist{1-3}{4-4}
    \end{squisheddiagram}
    
Note that this is known to be a square in the double category, since it is a face in a kernel cube in the $\fcgwa$ category $\C^{\mathcal{S}_n}$. To prove it is distinguished, we take the kernel of the right cube in the vertical direction 
    \begin{squisheddiagram}
    {B' & & C' & \\
     & B'' &  & C''\\
    B_{i,j} & & C_{i,j} & \\
     & B_{i+1,j}&  & C_{i+1,j}\\
    B_{i,j+1}&  & C_{i,j+1} & \\
     & B_{i+1,j+1} & & C_{i+1,j+1}\\};
    \mto{1-1}{2-2}!\cong \mto{3-1}{4-2} \mto{1-3}{2-4} \mto{3-3}{4-4}
    \eto{5-1}{3-1} \eto{5-3}{3-3}
    \mto{1-1}{3-1} \mto{1-3}{3-3} 
    \mto{5-1}{6-2} \mto{5-3}{6-4}
    \over{2-4}{2-2} \over{4-4}{4-2}
    \mto{1-3}{1-1} \mto{2-4}{2-2} \mto{3-3}{3-1} \mto{4-4}{4-2} \mto{5-3}{5-1} \mto{6-4}{6-2}
    \over{6-2}{4-2}
    \eto{6-2}{4-2} \eto{6-4}{4-4}
    \over{2-2}{4-2}
    \mto{2-2}{4-2} \mto{2-4}{4-4}
    \dist{3-1}{6-2}
    \end{squisheddiagram}
Since the indicated square is distinguished, the induced m-morphism on kernels is an isomorphism. But the top cube is a good cube; in particular, the top face is good, and thus a pullback. This implies that the m-morphism $C'\mrto C''$ must be an isomorphism, which in turn proves that the desired square is distinguished. The proof that $S_n \C$ is closed under $c$ proceeds dually.
    
Finally, we prove that $S_n \C$ is closed under $\star$. For this, we need to show that for any span of m-morphisms 
    \begin{squisheddiagram}
    {A_{i,j} & & B_{i,j} & & C_{i,j} & \\
    & A_{i+1,j} & & B_{i+1,j}&  & C_{i+1,j}\\
    A_{i,j+1} & & B_{i,j+1}&  & C_{i,j+1} & \\
    & A_{i+1,j+1} & & B_{i+1,j+1} & & C_{i+1,j+1}\\};
    \mto{1-1}{2-2} \mto{1-3}{2-4} \mto{1-5}{2-6}
    \mto{3-1}{4-2} \mto{3-3}{4-4} \mto{3-5}{4-6}
    \mto{1-3}{1-1} \mto{3-3}{3-1}
    \mto{1-3}{1-5} \mto{3-3}{3-5}
    \eto{3-1}{1-1}  \eto{3-3}{1-3} \eto{3-5}{1-5}
    \over{4-2}{2-2} \over{4-4}{2-4}
    \eto{4-2}{2-2} \eto{4-4}{2-4} \eto{4-6}{2-6}
    \over{2-4}{2-6}
    \mto{2-4}{2-6} \mto{4-4}{4-6}
    \over{2-2}{2-4}
    \mto{2-4}{2-2} \mto{4-4}{4-2}
    \dist{1-1}{4-2} \dist{1-3}{4-4} \dist{1-5}{4-6}
    \end{squisheddiagram}
the resulting square of $\star$-pushouts below is distinguished,
    \begin{diagram}
    {A_{i,j}\star_{B_{i,j}} C_{i,j} & A_{i+1,j}\star_{B_{i+1,j}} C_{i+1,j}\\
    A_{i,j+1}\star_{B_{i,j+1}} C_{i,j+1} & A_{i+1,j+1}\star_{B_{i+1,j+1}} C_{i+1,j+1}\\};
    \mto{1-1}{1-2} \mto{2-1}{2-2}
    \eto{2-1}{1-1} \eto{2-2}{1-2}
    \end{diagram}
which is ensured by \cref{lemma2.16distinguished}. 
\end{proof}

Lastly, we show that the double category  of w-grids $w_{l,m}\C \subset \C^{\D}$ of \cref{defngrids} is also an $\fcgwa$ category.

\begin{prop}\label{gridsfcgw}
$w_{l,m}\C$ is an $\fcgwa$ subcategory of $\C^\D$, where $\D$ denotes the free double category on an $l \times m$ grid of squares.  Moreover, if $\V$ a refinement of $\W$, then the double subcategory of grids in $\V$ forms an class of acyclic objects in $w_{l,m}\C$.
\end{prop}
\begin{proof}
Once again, by \cref{fullsubFCGW}, it suffices to prove that $w_{l,m}\C$ is closed under $k$, $c$, $\star$, and that it contains the initial object. The latter is trivial, as identity morphisms are always m- and e-equivalences.  

In order to prove that $w_{l,m}\C$ is closed under $k$, we must show that in the following diagram, where the right cube is the kernel of the left one, the maps in the rightmost square are m- and e-equivalences.
    \begin{symsquisheddiagram}
    {A_{i} & & B_{i} & & C_{i} & \\
    & A_{j} & & B_{j}&  & C_{j}\\
    A_{k} & & B_{k}&  & C_{k} & \\
    & A_{l} & & B_{l} & & C_{l}\\};
    \wmto{1-1}{2-2} \wmto{1-3}{2-4} \mto{1-5}{2-6}
    \wmto{3-1}{4-2} \wmto{3-3}{4-4} \mto{3-5}{4-6}
    \eto{1-1}{1-3} \eto{3-1}{3-3}
    \mto{1-5}{1-3} \mto{3-5}{3-3}
    \weto{3-1}{1-1}  
    \path[font=\scriptsize] (m-3-3) edge[emor] node[pos=0.3,above=-0.7ex,sloped] {$\sim$} (m-1-3);
    \eto{3-5}{1-5}
    \over{4-2}{2-2} \over{4-4}{2-4}
    \weto{4-2}{2-2} \weto{4-4}{2-4} \eto{4-6}{2-6}
    \over{2-6}{2-4}
    \mto{2-6}{2-4} \mto{4-6}{4-4}
    \over{2-2}{2-4}
    \eto{2-2}{2-4} \eto{4-2}{4-4}
    \end{symsquisheddiagram}
This is a direct consequence of \cref{parallel:2of3}; the statement for $c$ is analogous.

To show that $w_{l,m}\C$ is closed under $\star$, we need to prove that for any m-span as below left
    \[\begin{symsquishedinlinediagram}
    {A_{i} & & B_{i} & & C_{i} & \\
    & A_{j} & & B_{j}&  & C_{j}\\
    A_{k} & & B_{k}&  & C_{k} & \\
    & A_{l} & & B_{l} & & C_{l}\\};
    \wmto{1-1}{2-2} \wmto{1-3}{2-4} \wmto{1-5}{2-6}
    \wmto{3-1}{4-2} \wmto{3-3}{4-4} \wmto{3-5}{4-6}
    \mto{1-3}{1-1} \mto{3-3}{3-1}
    \mto{1-3}{1-5} \mto{3-3}{3-5}
    \weto{3-1}{1-1}  
    \path[font=\scriptsize] (m-3-3) edge[emor] node[pos=0.3,above=-0.7ex,sloped] {$\sim$} (m-1-3);
    \path[font=\scriptsize] (m-3-5) edge[emor] node[pos=0.3,above=-0.7ex,sloped] {$\sim$} (m-1-5);
    \over{4-2}{2-2} \over{4-4}{2-4}
    \weto{4-2}{2-2} \weto{4-4}{2-4} \weto{4-6}{2-6}
    \over{2-4}{2-6}
    \mto{2-4}{2-6} \mto{4-4}{4-6}
    \over{2-2}{2-4}
    \mto{2-4}{2-2} \mto{4-4}{4-2}
    \end{symsquishedinlinediagram} \qquad
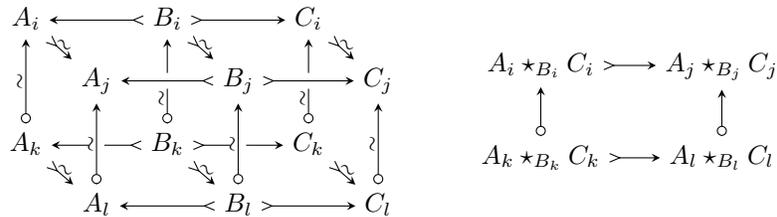
    \begin{inline-diagram}
    {A_{i}\star_{B_{i}} C_{i} & A_{j}\star_{B_{j}} C_{j}\\
    A_{k}\star_{B_{k}} C_{k} & A_{l}\star_{B_{l}} C_{l}\\};
    \mto{1-1}{1-2} \mto{2-1}{2-2}
    \eto{2-1}{1-1} \eto{2-2}{1-2}
    \end{inline-diagram}\]
the resulting square of $\star$-pushouts pictured above right is distinguished. But by \cref{lemma2.16distinguished}, we know that $\star$-pushouts preserve kernel-cokernel sequences; in other words, we have that $$k(A_k\star_{B_k} C_k\erto A_i\star_{B_i} C_i)=(A_k\bs A_i)\star_{B_k \bs B_i} (C_k\bs C_i),$$
$$c(A_i\star_{B_i} C_i\mrto A_j\star_{B_j} C_j)=(A_j / A_i)\star_{B_j / B_i} (C_j / C_i),$$ and similarly for the other two maps. We then conclude that the square above right is made of m- and e-equivalences due to \cref{acyclic:pushout}.
\end{proof}

\section{$K$-Theory of Polytopes}\label{appendixpolytopes}

In \cite[Definition 1.7, Theorem 2.1]{k1assemblers}, Zakharevich shows that the $K$-theory of an assembler agrees with that of a certain Waldhausen category. We will briefly describe this Waldhausen category for $\cG_n$ and $\fG_n$ and then show that it agrees with the $K$-theory of the $\fcgwa$ categories $\cGh_n$ and $\fGh_n$. We will use the term ``polytope'' to refer to a polytope of either type, as the morphisms are defined analogously for both.

\begin{defn}[{\cite[Definitions 1.3, 1.5]{k1assemblers}}]
    A formal sum of polytopes is a set $\{A_1,...,A_k\}$ where all $A_i$ are polytopes of the same type. A morphism of formal sums $$f \colon \{A_1,...,A_k\} \to \{B_1,...,B_\ell\}$$ consists of a function $f \colon \{1,...,k\} \to \{1,...,\ell\}$ along with morphisms of polytopes $A_i \to B_{f(i)}$ for $i=1,...,k$. The underlying morphism of finite sets (which we abuse notation by also denoting $f$) is called the base function.
    
    A morphism $f$ of formal sums is 
    \begin{itemize}
        \item a \emph{sub-map} denoted $\dashrightarrow$ if whenever $f(i)=f(i')=j$ the morphisms $f_i \colon A_i \to B_j$ and $f_{i'} \colon A_{i'} \to B_j$ have disjoint images,
        \item a \emph{covering sub-map} denoted $\begin{tikzcd}[column sep=small] {} \rar[dashed]{\sim} & {} \end{tikzcd}$ if moreover every element of each $B_j$ is in the image of some $A_i$,
        \item a \emph{move} denoted $\mrto$ if each morphism of polytopes $f_i$ is an isomorphism, and
        \item a \emph{monic move} if moreover the base function is injective.
    \end{itemize}
\end{defn}

The maps $P \xleftarrow{\sim} \{A_1,...,A_k\}$ from \cref{piecewisemaps} are an example of covering sub-maps where $P$ is regarded as a singleton formal sum \{P\}, while the maps $\{A_1,...,A_k\} \to Q$ in a piecewise map have no restriction except in a coproduct injection, in which case it is a sub-map.

\begin{defn}[{\cite[Definition 1.7]{k1assemblers}}]
    The category $SC(\cG_n)$ (analogously $SC(\fG_n)$) has as objects formal sums of polytopes and as morphisms isomorphism classes of spans of the form
    \[
    \begin{tikzcd}
        \{A_1,...,A_k\} & \lar[dashed] \{B_1,...,B_\ell\} \rar[tail] & \{C_1,...,C_m\} 
    \end{tikzcd}
    \]
    where the left leg is a sub-map and the right leg is a move. Composition is defined using pullbacks.

    A morphism in $SC(\cG_n)$ (analogously $SC(\fG_n)$) is a cofibration if the sub-map is covering and the move is monic, and is moreover a weak equivalence if the move is an isomorphism.
\end{defn}

Note that a weak equivalence is an isomorphism class of spans which includes one in which the rightward move is an identity.

A close inspection of the definitions of this Waldhausen category shows that the cofiber maps, which appear vertically in the $S_\bullet$-construction, have the form 
\[
\begin{tikzcd}
    \{C_1,...,C_m\} & \lar[dashed,tail] \{D_1,...,D_n\} \rar[equals] & \{D_1,...,D_n\}
\end{tikzcd}
\]
where the leftward arrow is a monic move and hence also a sub-map.

Zakharevich shows that $SC(\cG_n)$ and $SC(\fG_n)$ are Waldhausen categories whose $K$-theory agrees with previous definitions of $K$-theory for polytopes \cite[Theorems 1.9, 2.1]{k1assemblers}. 

\begin{prop}
    The $K$-theories $K(\cGh_n,\varnothing)$ and $K(\fGh_n,\varnothing)$ of $\fcgwa$ categories are equivalent to the $K$-theories of the Waldhausen categories $SC(\cG_n)$ and $SC(\fG_n)$.
\end{prop}

\begin{proof}
     Noting that in $\cGh_n$ and $\fGh_n$ every piecewise coproduct injection is isomorphic to a ``total'' map whose reverse scissors congruence is an identity, we can describe the distinguished squares in our $S_\bullet$ diagram as below left (as opposed to the squares in Zakharevich's $S_\bullet$ construction, depicted below right).
    \[
    \begin{tikzcd}
        A & \lar[swap]{\sim} \{B_1,...,B_{k'}\} \rar[hook] & C  \ar[phantom]{dl}[pos=.1]{\llcorner} \\
        D \uar[hook] & \lar[swap]{\sim} \{E_1,...,E_{\ell'}\} \rar[hook] \uar[hook] \ar[phantom]{ul}[pos=0]{\ulcorner} \ar[phantom]{ur}[pos=0]{\urcorner} & F \uar[hook]
    \end{tikzcd}
    \qquad\qquad
    \begin{tikzcd}[column sep=small]
        \{A_1,...,A_k\} & \lar[swap,dashed]{\sim} \{B_1,...,B_{k'}\} \rar[tail] & \{C_1,...,C_m\} \ar[phantom]{dl}[pos=0]{\llcorner} \\
        \{D_1,...,D_\ell\} \uar[dashed,tail] & \lar[swap,dashed]{\sim} \{E_1,...,E_{\ell'}\} \rar[tail] \uar[dashed,tail] \ar[phantom]{ul}[pos=0]{\ulcorner} \ar[phantom]{ur}[pos=0]{\urcorner} & \{F_1,...,F_{m'}\} \uar[dashed,tail]
    \end{tikzcd}
    \]
    
    We now wish to compare the double category $iS_m\cGh_n$ (respectively, $iS_m\fGh_n$) with the category $wS_mSC(\cG_n)$ (respectively, $wS_mSC(\fG_n)$). For convenience, we will consider the former as categories instead as in the proof of \cref{iS:equals:s}, and argue only for $\cG_n$ as the proof for $\fG_n$ is indistinguishable. Define a functor $\coprod_m \colon wS_mSC(\cG_n) \to iS_m\cGh_n$ induced by the assignment mapping an object in $SC(\cG_n)$, i.e.\  a formal sum of polytopes $\{A_1,...,A_k\}$, to a fixed choice $\coprod_m\{A_1,...,A_k\}$ of their disjoint union in $\cGh_n$, and a cofibration in $SC(\cG_n)$ represented by the span
    \[
    \begin{tikzcd}
         \{A_1,...,A_k\} & \lar[dashed,swap]{\sim} \{B_1,...,B_{k'}\} \rar[tail] & \{C_1,...,C_\ell\} 
    \end{tikzcd}
    \] to the morphism in $\cGh_n$ defined by composition as below.
    \[
    \begin{tikzcd}
        \coprod_m\{A_1,...,A_k\} & \lar[swap]{\sim} \{A_1,...,A_k\} & \lar[dashed,swap]{\sim} \{B_1,...,B_{k'}\} \rar[tail] & \{C_1,...,C_\ell\} \rar{\sim} & \coprod_m\{C_1,...,C_\ell\}
    \end{tikzcd}
    \]
    A weak equivalence of staircases in $\cG_n$ is sent to an isomorphism of staircases in $\cGh_n$ (as scissors congruences are invertible in the latter), making  $\coprod_m$ a functor as claimed. Note that this assignment preserves the initial object, and extends to staircase diagrams based on the descriptions of their squares above. Moreover the functors $\coprod_m$ form a simplicial functor, as the simplicial maps on staircases merely forget or duplicate portions of a staircase and each functor $\coprod_m$ uses the same choice of coproduct for each formal sum of polytopes.
    
    Since scissors congruences are not invertible in $wS_n\cG_n$, we will not be able to define a suitable functor in the opposite direction. We therefore appeal to Quillen's Theorem A \cite[Page 85]{Qui73} and show that for each staircase $X$ of polytopes in $iS_m\cGh_n$, the slice category $\coprod_m/X$ is contractible, which will allow us to conclude that $\coprod_m$ is a weak equivalence for all $m$ and therefore induces an equivalence on $K$-theory spaces. 
    
    Beginning with $m=1$, for a polytope $X=A$ the slice category has as objects morphisms
    \[
    \coprod \{C_1,...,C_k\} \xleftarrow{\sim} \{B_1,...,B_{k'}\} \xrightarrow{\sim} A
    \]
    in $\cGh_n$, and as morphisms weak equivalences
    \[
    \begin{tikzcd}
    \{C_1,...,C_k\} & \lar[dashed,swap]{\sim} \{F_1,...,F_\ell\} \rar[equals] & \{F_1,...,F_\ell\}
    \end{tikzcd}
    \]
    in $SC(\cG_n)$ which commute over $A$ in the sense of admitting a diagram as below.
    \[
    \begin{tikzcd}
    \coprod\{C_1,...,C_k\} & \lar[swap]{\sim} \{C_1,...,C_k\} & \lar[dashed,swap]{\sim} \{F_1,...,F_\ell\} \rar{\sim} & \coprod\{F_1,...,F_\ell\} \\
    \{B_1,...,B_{k'}\} \dar{\sim} \uar[swap]{\sim} & \lar[dashed,swap]{\sim} \{D_1,...,D_m\} \uar[dashed,swap]{\sim} \ar[phantom]{ul}[pos=0]{\ulcorner} & \lar[dashed,swap]{\sim} \{G_1,...,G_{m'}\} \rar[dashed]{\sim} \uar[dashed,swap]{\sim} \ar[phantom]{ul}[pos=0]{\ulcorner} \ar[phantom]{ur}[pos=0]{\urcorner} & \{E_1,...,E_{\ell'}\} \dar{\sim} \uar[swap]{\sim} \\
    A \ar[equals]{rrr} & & & A
    \end{tikzcd}
    \]
    There is a sequence of endofunctors on $\coprod_1/A$ sending $\{C_1,...,C_k\}$ to itself, then $\{D_1,...,D_m\}$, then $\{A\}$ from the above diagram, and a zig-zag of natural transformations between them sending a morphism as above to the zig-zag of naturality squares pictured below,
    \[
    \begin{tikzcd}
        \{C_1,...,C_k\} & \lar[dashed,swap]{\sim} \{D_1,...,D_m\} \rar[dashed]{\sim} & \{B_1,...,B_{k'}\} \rar[dashed]{\sim} & \{A\} \\
        \{F_1,...,F_\ell\} \uar[dashed,swap]{\sim} & \lar[dashed,swap]{\sim} \{G_1,...,G_{m'}\} \rar[dashed]{\sim} \uar[dashed,swap]{\sim} & \{E_1,...,E_{\ell'}\} \rar[dashed]{\sim} & \{A\} \uar[equals] 
    \end{tikzcd}
    \]
    which can be checked to commute over $A$ based on the commutativity of the morphism from $\{C_1,...,C_k\}$ to $\{F_1,...,F_\ell\}$. This zig-zag of natural transformations from the identity functor to a constant functor exhibits $\coprod_1/A$ as contractible.

    For $m > 1$, this argument works similarly for larger staircase diagrams where the role of $\{A\}$ in the $m=1$ case is replaced by a choice of formal sums for each polytope in the staircase such that the polytope $A_{i,j}$ is represented by the formal sum $\{A_{i,i},...,A_{j,j}\}$: this is the ``minimal'' formal sum representation of a staircase $A$ in $\cGh_n$ subject to the conditions imposed by the squares in the staircases in $SC(\cG_n)$.
\end{proof}

\bibliographystyle{alpha}
\bibliography{references}
\end{document}